\documentclass[preprint,authoryear]{elsarticle}
\usepackage[T1]{fontenc}
\usepackage[utf8]{inputenc}
\usepackage{tabularx}
\usepackage{supertabular}
\usepackage{booktabs}
\RequirePackage[authoryear]{natbib}
\usepackage{lmodern}
\usepackage{setspace}
\usepackage{nicefrac}
\usepackage{amsmath}
\usepackage{mathtools,mathrsfs}
\usepackage{a4wide}
\usepackage{amssymb}
\usepackage{amsthm}
\usepackage{framed}
\usepackage{bbm}
\usepackage{changes}
\usepackage[flushleft]{threeparttable}
\definecolor{dblue}{rgb}{0.21,0.21,0.55}
\usepackage{multirow}
\usepackage{lscape}
\usepackage[english]{babel}
\usepackage{graphicx}
\usepackage{arydshln}
\usepackage{rotating}
\usepackage[flushmargin,hang]{footmisc}
\usepackage{amsmath}
\usepackage{lscape}
\usepackage{marginnote,letltxmacro}
\usepackage[pdflinkmargin=5pt,pdfstartview={FitBH -32768},plainpages=false]{hyperref}
\usepackage{xcolor}
\usepackage{bbm}
\renewcommand{\P}{\mathbb{P}}
\newcommand{\F}{\mathcal{F}}

\newcommand{\E}{\mathbb{E}}
\newcommand{\N}{\mathbb{N}}
\newcommand{\R}{\mathbb{R}}
\newcommand{\1}{\mathbbm{1}}
\newcommand{\KLEINO}{{\scriptstyle{\mathcal{O}}}}
\DeclareMathAccent{\verywidehat}{\mathord}{largesymbols}{'144}
\newcommand{\var}{\mathbb{V}\hspace*{-0.05cm}\textnormal{a\hspace*{0.02cm}r}}

\allowdisplaybreaks[3]
\newtheorem{hypo}{Hypothesis}

\newtheorem{alt}{Alternative}

\newtheorem{prop}{Proposition}[section]
\newtheorem{lem}{Lemma}
\newtheorem{cor}[prop]{Corollary}

\newtheorem{ass}[prop]{Assumption}

\newtheorem{theo}[prop]{Theorem}
\newtheorem{rem}[prop]{Remark}

\begin{document}
\renewcommand*{\thefootnote}{\fnsymbol{footnote}}

\title{Change-point inference on volatility in noisy It\^{o} semimartingales}
\author[1]{Markus Bibinger}
\author[2]{Mehmet Madensoy\footnote{Madensoy acknowledges financial support from the Deutsche Forschungsgemeinschaft via RTG 1953.}}
\address[1]{Faculty of Mathematics and Computer Science, Philipps-Universit\"at Marburg} 
\address[2]{School of Business Informatics and Mathematics, Mannheim University}
\normalsize
\begin{frontmatter}
\date{This version: 06.09.2017}

\vspace{-.7cm} 

\begin{abstract}
{{\normalsize \noindent
This work is concerned with tests on structural breaks in the spot volatility process of a general It\^{o} semimartingale based on discrete observations contaminated with i.i.d.\;microstructure noise. We construct a consistent test building up on infill asymptotic results for certain functionals of spectral spot volatility estimates. A weak limit theorem is established under the null hypothesis relying on extreme value theory. We prove consistency of the test and of an associated estimator for the change point. A simulation study illustrates the finite-sample performance of the method and efficiency gains compared to a skip-sampling approach.}}
\begin{keyword}
Change-point analysis\sep high-frequency data \sep market microstructure \sep volatility estimation \sep volatility jump\\[.25cm]
\MSC[2010] 62M10 \sep 62G10
\end{keyword}
\end{abstract}

\end{frontmatter}
\thispagestyle{plain}

\newcommand{\KK}{L}
\newcommand{\aalpha}{\mathfrak{a}}
\newcommand{\xiv}{\mbox{\boldmath$\xi$}}
\newcommand{\etav}{\mbox{\boldmath$\eta$}}
\newcommand{\muv}{\mbox{\boldmath$\mu$}}
\newcommand{\chiv}{\mbox{\boldmath$\chi$}}
\newcommand{\gan}{{\lfloor \alpha_{n}\rfloor}}
\newcommand{\an}{\alpha_{n}}
\newcommand{\hn}{h_{n}}
\newcommand{\ah}{\alpha_{n}h_{n}}
\newcommand{\ahll}{h_{n}(i\alpha_{n}+\left(\ell-1\right))-n^{-1}}
\newcommand{\ahlla}{h_{n}(i\alpha_{n}+\left(\ell-1\right))}
\newcommand{\ahl}{i\alpha_{n}h_{n}+\ell h_{n}}
\newcommand{\maxi}{\mathop{\mathrm{max}}\limits_{i}}
\newcommand{\mini}{\mathop{\mathrm{min}}\limits_{i}}
\newcommand{\x}{\xi^{\left(n\right)}_{ij\ell}}
\newcommand{\nh}{{\lfloor nh_{n}\rfloor}-1}
\newcommand{\Kp}{K^{+}}
\newcommand{\Km}{K^{-}}
\newcommand{\kp}{\kappa^{+}}
\newcommand{\km}{\kappa^{-}}
\newcommand{\w}{w_{ij\ell}}
\newcommand{\SX}{S_{ij\ell}\left(X\right)}
\newcommand{\SW}{S_{ij\ell}\left(W\right)}
\newcommand{\SWsq}{S^2_{ij\ell}\left(W\right)}
\newcommand{\Se}{S_{ij\ell}\left(\varepsilon\right)}
\newcommand{\Sesq}{S^2_{ij\ell}\left(\varepsilon\right)}
\newcommand{\Ri}{\overline{RV}_{n,i}}
\newcommand{\Rii}{\overline{RV}_{n,i+1}}
\newcommand{\Zi}{\overline{Z}_{n,i}}
\newcommand{\Zii}{\overline{Z}_{n,i+1}}
\newcommand{\absRi}{\left|\overline{RV}_{n,i}\right|}
\newcommand{\absRii}{\left|\overline{RV}_{n,i+1}\right|}
\newcommand{\absZi}{\left|\overline{Z}_{n,i}\right|}
\newcommand{\absZii}{\left|\overline{Z}_{n,i+1}\right|}
\newcommand{\absZiiw}{\left|\widetilde{Z}_{n,i+1}\right|}
\newcommand{\Ziiw}{\widetilde{Z}_{n,i+1}}
\newcommand{\Kmk}{K^{-}\hspace{-0.6em}-\kappa^{-}}


\section{Introduction\label{sec:1}}
Inference on structural breaks for discrete-time stochastic processes, particularly in time series analysis, is a very active research field within mathematical statistics. 
Whereas the latter is usually concerned with i.i.d.\;data, important contributions beyond that case are presented in \cite{wuzhao2007}, proving limit theorems for nonparametric change-point analysis under weak dependence. These results serve as an important ingredient for the present work. 
So far inference on structural breaks for continuous-time stochastic processes has attracted less attention. Let us mention the very recent work by \cite{buecher2017}, which also deals with questions of detecting structural breaks of certain continuous-time stochastic processes. Our target of inference is the volatility process. Understanding the structure and dynamics of stochastic volatility processes is a highly important issue in finance and econometrics. Due to the outstanding role of volatility for quantifying financial risk, there is a vast literature on these topics.\\
Motivated by fundamental results in financial mathematics, the process modeling the logarithmic price of an asset belongs to the class of semimartingales. Whereas statistics for general semimartingales is less developed, a lot of work has been done if $\left(X_{t}\right)_{t\in\left[0,1\right]}$ is an It\^{o} semimartingale, that is, a semimartingale with a characteristic triple being absolutely continuous with respect to the Lebesgue measure. An overview on existing theory is available by \cite{JP}. More precisely, our continuous-time model is
\begin{align}\label{itoprocess}
X_{t}=C_{t}+J_{t}\,,
\end{align}
with $\left(C_{t}\right)_{t\in\left[0,1\right]}$ the continuous part,
\begin{align}\label{itocont}
C_{t}=X_{0}+\int_{0}^{t}\sigma_{s}\,dW_{s}+\int_{0}^{t}a_{s}\,ds\,,
\end{align}
with a standard Brownian motion $\left(W_{t}\right)_{t\in\left[0,1\right]}$, the volatility process $\left(\sigma_{t}\right)_{t\in\left[0,1\right]}$ and the drift process $\left(a_{t}\right)_{t\in\left[0,1\right]}$. We define the pure-jump process $\left(J_{t}\right)_{t\in\left[0,1\right]}$ through the Grigelionis representation
\begin{align}\label{grigel}
J_t=\int_0^t \int_{\mathbb{R}}\delta(s,z)\1_{\{|\delta(s,z)|\leq 1 \}}(\mu-\nu)(ds,dz) +\int_0^t \int_{\mathbb{R}}\delta(s,z)\1_{\{|\delta(s,z)|> 1 \}} \mu(ds,dz)
\end{align}
with a Poisson random measure $\mu$ having a compensator of the form $\nu(ds,dz)=\lambda(dz)\otimes ds$ with a $\sigma$-finite measure $\lambda$.
 
In this paper we are going to work with discrete observations and within the framework of infill asymptotics. That is, our data is generated by discretizing a path of the continuous-time stochastic process $\left(X_{t}\right)_{t\in\left[0,1\right]}$ on a regular, equidistant grid. Though the model \eqref{itoprocess} is quite flexible, empirical evidence suggests that the recorded financial high-frequency data in applications does not follow a `true' semimartingale. Therefore, an extension of model \eqref{itoprocess} incorporating microstructure noise is necessary. Market microstructure noise is caused by various trading mechanisms as discreteness of prices and bid-ask spread bounce effects. The observed data is modeled through
\begin{align}\label{data}
Y_{i\Delta_{n}},\quad i=0,\ldots,\Delta^{-1}_{n}=n\in\mathbb{N}\,,
\end{align}
as a discretization of a continuous-time stochastic process $\left(Y_{t}\right)_{t\in\left[0,1\right]}$, given by a superposition
\begin{align}\label{smn}
Y_{t}=X_{t}+\varepsilon_{t}\,,
\end{align}
with $\left(\varepsilon_{t}\right)_{t\in\left[0,1\right]}$ being a centered white noise process modeling the microstructure noise. This prominent additive noise model has attained considerable attention in the econometrics and statistics literature, let us refer to the book by \cite{sahaliajacodbook} for an overview. Infill asymptotics implies $\Delta_{n}\rightarrow 0$ or, equivalently, $n\rightarrow +\infty$.
Whereas the drift process $\left(a_{t}\right)_{t\in\left[0,1\right]}$ is not identifiable in a high-frequency framework, also without noise, quantities as the spot volatility process $\left(\sigma^{2}_{t}\right)_{t\in\left[0,1\right]}$ and the integrated volatility process $\int_{0}^{.}\sigma^{2}_{s}\,ds$, respectively, are identifiable. Since they constitute key quantities for an econometric risk analysis, there exists a rich literature on estimation theory. We refer to \cite{JP} for a comprehensive presentation of these topics. This work is aimed to increase the understanding of the structure of the spot volatility process and to complement existing literature. The recent work by \cite{Bibinger2017} presents results on change-point detection for the model \eqref{itoprocess} without noise. We focus on a test that distinguishes continuous volatility paths from paths with volatility jumps. Inference on volatility jumps is currently of great interest in the literature, see, for instance, \cite{jacodtodorov} and \cite{voljumps}. Moreover, it provides a necessary ingredient to analyze possible discontinuous leverage effects, see \cite{DLE} for a recent approach to this question. Inference on the volatility poses a challenging statistical problem, the volatility being latent and not directly observable. In the model \eqref{smn} with microstructure noise, this becomes even more involved. The only work we are aware of addressing inference on volatility jumps in this model is by \cite{BibingerWinkel2018} who extend the test for contemporaneous price and volatility jumps by \cite{jacodtodorov} to noisy observations. Restricting to finitely many price-jump times, their results do not render general inference on volatility jumps. In this work we will extend the methods and results presented in \cite{Bibinger2017} in order to construct a general test for volatility jumps based on the model \eqref{smn}. Our statistics are functionals of spectral spot volatility estimates building up on the local Fourier method of moments in \cite{Altmeyer2015} extending the volatility estimation approach introduced by \cite{Reis2011}. While several linear estimators for functionals of the volatility have by now been generalized to noise-robust approaches, the considered change-point test is based on maximum statistics and its extension to an efficient method under noise requires new techniques.

The key theorem for the test is a limit theorem under the null hypothesis with an extreme value limit distribution of Gumbel-type. In particular, a clever rescaling of differences of local spot volatility estimates, quite different from the statistics considered in \cite{Bibinger2017}, yields an asymptotic distribution-free test. In a certain sense, our test for volatility jumps complements the prominent Gumbel-test for price jumps proposed by \cite{leemykland} and further studied by \cite{Palmes2016} and \cite{woerner}. An extension of the Gumbel-test for price jumps to noisy observations is given in \cite{leemykland2}. We prove that our Gumbel-test for volatility jumps is consistent. Similar to the price-jump test, it facilitates also detection of the jump times -- the change points. One main difficulty to prove the limit theorem is to uniformly control the spot volatility estimation errors.

The paper is organized as follows. Section 2 introduces the testing problem and the assumptions. Section 3 constructs the test. We begin with the test for a continuous semimartingale $(X_t)_{t\in[0,1]}$ which is then extended to the general case utilizing truncation techniques. Section 4 establishes the asymptotic theory including the limit theorem under the null hypothesis, consistency of the test and consistent estimation of the change point under the alternative hypothesis. In Section 5 we conduct a Monte Carlo simulation study. The main insight is that the new test considerably increases the power compared to (optimally) skip sampling the noisy data to lower frequencies and applying the not noise robust method by \cite{Bibinger2017} directly. Section 6 gathers the proofs.

\section{Testing problem and theoretical setup\label{sec:2}}
We will develop a test for volatility jumps. We aim to test for some c\`adl\`ag squared volatility process $(\sigma_t^2)_{t\in[0,1]}$ hypotheses of the form

\begin{equation}\label{hyptest}
\begin{aligned}
&H_{0}: \text{there is no jump in }\sigma^{2}_{t}\quad\text{vs.}\\
&H_{1}: \text{there is at least one }\theta\in\left(0,1\right)\text{ such that }\big|\sigma^{2}_{\theta}-\lim_{s\to\theta,s<\theta}\sigma_s^2\big|>0\,.
\end{aligned}
\end{equation}

It is standard in the theory of statistics of high-frequency data to address such questions \emph{path-wise}. This means that $H_0$ and $H_1$ are formulated for one particular path of the squared volatility $(\sigma_t^2(\omega))_{t\in[0,1]}$ and we strive to make a decision based on discrete observations of the given path of $(Y_t(\omega))_{t\in[0,1]}$. 
The semimartingale $\left(X_{t}\right)_{t\in\left[0,1\right]}$ is defined on a filtered probability space $(\Omega^{(0)},\mathcal{F}^{(0)},(\mathcal{F}^{(0)})_{t\in\left[0,1\right]},\mathbb{P}^{(0)})$. We need further assumptions on the coefficient processes of $(X_t)$.
\begin{ass} \label{assVola}
The processes $a$ and $\sigma$ are locally bounded. $\sigma$ is almost surely strictly positive, that is, $\inf_{t\in[0,1]}\sigma_t^2\ge \Km>0$.
\end{ass}
Our notation for jump processes follows \cite{JP}.
\begin{ass} \label{assjumps}
Suppose $\sup_{\omega,x}|\delta(s,x)|/\gamma(x)$ is locally bounded for some deterministic non-negative function $\gamma$ which satisfies for some $r\in[0,2]$:
\begin{align}\label{bg}\int_{\R}(1\wedge \gamma^r(x))\lambda(dx)<\infty\,.\end{align}
\end{ass}
The smaller $r$, the more restrictive Assumption \ref{assjumps}. The case $r=0$ is tantamount to jumps of finite activity.\\
On the null hypothesis, we allow for very general and rough continuous stochastic volatility processes. 
\begin{hypo}\label{hypo} Under the null hypothesis, 
	the modulus of continuity
\[
w_\delta(\sigma)_t = \sup_{s,r \leq t} \{ |\sigma_s - \sigma_r| : |s-r| < \delta \}
\]
is locally bounded in the sense that there exists $\aalpha > 0$ and a sequence of stopping times $T_n \to \infty$, such that $w_\delta(\sigma)_{(T_n\wedge 1)} \leq \KK_n \delta^\aalpha$, for some $\aalpha > 0$ and some (almost surely finite) random variables $\KK_n$.	
\end{hypo}
The regularity exponent $\aalpha\in(0,1]$ is selected for the testing problem. The test can be repeated for different values also. The regularity exponent coincides with a usual Hölder exponent when $\KK_n$ is a fix constant. Integrating a sequence $\KK_n$ enables us to include stochastic volatility processes in our theory. Since stochastic processes as Brownian motion are not in some fix Hölder class, it is crucial to work with (slightly) more general smoothness classes determined by the exponent $\aalpha > 0$ and by $\KK_n$. Observe that if
\begin{align*}
\mathbb{E}\left[\left|\sigma^{2}_{t}-\sigma^{2}_{s}\right|^{\mathfrak{b}}\right]\leq C\left|t-s\right|^{\gamma+\mathfrak{a}\mathfrak{b}}\,,\text{for some }\mathfrak{b},\,C>0\text{ and }\gamma>1\,,
\end{align*}
then the Kolmogorov \v{C}entsov Theorem implies that
\begin{align*}
\lim_{n \to +\infty} \mathbb{P}\Big(\sup_{\substack{s,t\in\left[0,1\right]\\ \left|t-s\right|\leq\delta}}\left|\sigma^{2}_{t}-\sigma^{2}_{s}\right|\leq L_{n}\delta^{\mathfrak{a}}\Big) = 1
\end{align*}
if $L_{n}\rightarrow+\infty$ arbitrarily slowly. In particular, we can impose that $L_n=\mathcal{O}(\log(n))$ for our derivation of upper bounds in the sections below. The null hypothesis is the same as in Assumption 3.1 of \cite{Bibinger2017}. Our test distinguishes the null hypothesis from alternative hypotheses of the following type.
\begin{alt}\label{alt} Under the alternative hypothesis, there exists at least one $\theta\in\left(0,1\right)$, such that 
\[\big|\Delta\sigma_{\theta}^2\big|=\big|\sigma_{\theta}^2-\lim_{s\to\theta,s<\theta}\sigma_s^2\big|=\delta>0\,.\]
We suppose that $\sigma_t^2=\sigma_t^{2,(c)}+\sigma_t^{2,(j)}$, where $(\sigma_t^{2,(c)})_{t\in[0,1]}$ satisfies \ref{hypo}. The jump component $(\sigma_t^{2,(j)})_{t\in[0,1]}$ is a pure-jump semimartingale which satisfies Assumption \ref{assjumps} with $r\le 2$.
\end{alt}
In particular, the alternative hypothesis does not restrict to only one jump. We establish a consistent test when \emph{at least} one non-negligible jump is present. Multiple jumps and quite general jump components are possible. Consistency of our test only requires that in a small vicinity of $\theta$, $(\sigma_t^{2,(c)})$ and $(\sigma_t^{2,(j)}-\Delta\sigma_{\theta}^2\1_{[\theta,1]}(t))$ are sufficiently regular such that the jump $\Delta \sigma_{\theta}^2$ is detected. \cite{Bibinger2017} impose in their Theorem 4.3 the condition that all volatility jumps are positive. This condition is replaced here by the semimartingale assumption on $(\sigma_t^{2,(j)})_{t\in[0,1]}$. Both ensure that $\Delta\sigma_{\theta}^2$ can not be compensated by opposite jumps in an asymptotically small vicinity.  
In order to incorporate microstructure noise, we have to extend the original probability space.
We set $\mathcal{G}_{t}=\mathcal{F}^{(0)}_{t}\otimes\sigma\left(\varepsilon_{s}:\, s\leq t\right)$. The data generating process $\left(Y_{t}\right)_{t\in\left[0,1\right]}$ is defined on the filtered probability space $\left(\Omega,\mathcal{G},(\mathcal{G}_{t})_{t\in\left[0,1\right]},\mathbb{P}\right)$. The construction can be pursued such that the process $(X_t)_{t\in[0,1]}$ remains a semimartingale on the extension with the same characteristic triplet and the same Grigelionis representation. For the details of the construction we refer to Chapter 16 in \cite{JP}. For the noise process, we impose further assumptions.
\begin{ass} \label{assnoise}
The stochastic process $\left(\varepsilon_{t}\right)_{t\in\left[0,1\right]}$ is defined on $\left(\Omega,\mathcal{G},(\mathcal{G}_{t})_{t\in\left[0,1\right]},\mathbb{P}\right)$ and fulfills the following conditions.
\begin{enumerate}
	\item[(1)] $\left(\varepsilon_{t}\right)_{t\in\left[0,1\right]}$ is a centered white noise process, $\E[\varepsilon_{t}]\equiv 0$, and with
	\begin{align*}
	\mathbb{E}\left[\varepsilon^{2}_{t}\right]=\eta^{2}\,.
	\end{align*}
	\item[(2)] The following moment condition holds.
	\begin{align}\label{momnoise}
\mathbb{E}\left[\left|\varepsilon_{t}\right|^{m}\right]<\infty,\text{ for all }m\in\mathbb{N}.
\end{align}
\end{enumerate}
\end{ass}
It is well-known that $\eta^2$ can be estimated in this model with $\sqrt{n}$-rate by either a rescaled realized volatility or from the negative first-lag autocovariances of the noisy increments. Under Assumption \ref{assnoise}, \cite{Zhang2005} provide a rate-optimal consistent estimator for $\eta^2$:
\begin{align}\label{etahat}\hat\eta^2=\frac{1}{2n}\sum_{i=1}^n(Y_{i\Delta_n}-Y_{(i-1)\Delta_n})^2=\eta^2+\mathcal{O}_{\P}\big(n^{-1/2}\big)\,.\end{align}
\begin{rem}\rm
The moment condition \eqref{momnoise} is standard in related literature, see, for instance, Assumption (WN) of \cite[page 221]{sahaliajacodbook} or Assumption 16.1.1 of \cite{JP}, but in a certain sense purely technical. Let us stress that in our setting, we do impose as less assumptions as possible on the volatility process $\left(\sigma_{t}\right)_{t\in\left[0,1\right]}$. More precisely, the regularity under \ref{hypo}, for arbitrarily small $\mathfrak{a}\in\left(0,1\right]$, requires the existence of all moments in \eqref{momnoise}. More precisely, the smaller $\mathfrak{a}$, the larger $m$ has to be chosen. Nevertheless, we point out that the moment condition is not that restrictive for standard models of volatility. In the usual case, for instance, where $\left(\sigma_{t}\right)_{t\in\left[0,1\right]}$ itself is assumed to be an It\^{o} semimartingale, when $\mathfrak{a}\approx 1/2$, only the existence of moments up to order $m=8$ has to be imposed.
\end{rem}
\begin{rem}\rm
While Assumption \ref{assnoise} is in line with standard conditions on the additive noise component in the literature, possible generalizations with respect to the structure of the noise process $(\varepsilon_{t})_{t\in\left[0,1\right]}$ in three directions are of interest: \emph{serial dependence, heterogeneity and endogeneity}. Such generalizations are also motivated by stylized facts in econometrics, see \cite{hansenlunde} for a detailed discussion. For instance, Chapter 16 in \cite{JP} includes conditional i.i.d.\ noise, endogenous as it may depend (in a certain way) on $(X_t)$, in the theory of pre-average estimators. This allows to model phenomena as noise by price discreteness (rounding). \cite{BibingerWinkel2018} provide some first extensions of spectral spot volatility estimation to serially correlated and heterogeneous noise. Though the possible extensions appear to be relevant for applications, we work in the framework formulated in Assumption \ref{assnoise}, mainly due to the lack of groundwork sufficient for the present work. Since we exploit some ingredients from previous works on spectral volatility estimation, particularly the form of the efficient asymptotic variance based on \cite{Altmeyer2015}, a generalization of our results requires non-trivial generalizations of these ingredients first. Furthermore, more general noise processes ask for extensive work on the estimation of the local long-run variance replacing \eqref{etahat}. This topic, however, is beyond the scope of this work. Let us remark that it is as well not obvious how to apply strong embedding principles in these cases to generalize our proofs. Since \cite{wuzhao2007} provide strong approximation results for weakly dependent time series, we nevertheless conjecture that certain generalizations in the three directions are possible.
\end{rem}

\section{The statistical methods\label{sec:3}}
\subsection{The continuous case\label{subsec:3.1}}
In this paragraph, we construct the test first for the model $\left(X_{t}\right)_{t\in\left[0,1\right]}$ without jumps, that is, we assume that
\begin{align*}
J_{t}\equiv 0.
\end{align*}
The construction of the test is based on a combination of the techniques by \cite{Altmeyer2015} and \cite{Bibinger2017}. In order to do so, we pick a sequence $h_{n}$ with
\begin{align}\label{eq:hn}
h_{n}\varpropto n^{-1/2}\log\left(n\right)
\end{align}
and $h_{n}^{-1}\in\mathbb{N}$.\\
The observation interval $\left[0,1\right]$ is split into $h_{n}^{-1}$ bins of length $h_{n}$, such that each bin is given by
\begin{align*}
\left[\left(k-1\right)h_{n},kh_{n}\right],\quad k=1,\ldots,h_{n}^{-1}.
\end{align*}
Furthermore, we consider the $L^{2}\left(\left[0,1\right]\right)$ orthonormal systems, given by
\begin{align*}
\Phi_{jk}\left(t\right)&=\Phi_{j}\left(t-\left(k-1\right)h_{n}\right)\\
\varphi_{jk}\left(t\right)&=\varphi_{j}\left(t-\left(k-1\right)h_{n}\right)
\end{align*}
with
\begin{align*}
\Phi_{j}\left(t\right)&=\sqrt{\frac{2}{h_{n}}}\sin\left(j\pi h_{n}^{-1}t\right)\mathbbm{1}_{\left[0,h_{n}\right]}\left(t\right),\quad j\geq 0, 0\leq t\leq 1\,,\\
\varphi_{j}\left(t\right)&=2n\sqrt{\frac{2}{h_{n}}}\sin\left(\frac{j\pi}{2nh_{n}}\right)\cos\left(j\pi h_{n}^{-1}t\right)\mathbbm{1}_{\left[0,h_{n}\right]}\left(t\right).
\end{align*}
We define, for any stochastic process $\left(L_{t}\right)_{t\in\left[0,1\right]}$, the increments $\Delta_{i}^{n}L$ by
\begin{align*}
\Delta_{i}^{n}L=L_{\frac{i}{n}}-L_{\frac{i-1}{n}}\,,i=1,\ldots,n,
\end{align*}
and the \emph{spectral statistics}
\begin{align*}
S_{jk}\left(L\right):=\sum_{i=1}^{n}\Delta_{i}^{n}L\,\Phi_{jk}\left(\frac{i}{n}\right).
\end{align*}
The squared volatility $\sigma^{2}_{\left(k-1\right)h_{n}}$ can be estimated locally by a parametric estimator through oracle versions of bias corrected linear combinations of the squared spectral statistics,
\begin{align}\label{sigmahat}
\hat{\sigma}^{2}_{\left(k-1\right)h_{n}}=\sum_{j=1}^{\nh}w_{jk}\left(S^{2}_{jk}\left(Y\right)-\left[\varphi_{jk},\varphi_{jk}\right]_{n}\frac{\eta^{2}}{n}\right)\,,
\end{align}
with variance minimizing oracle weights $w_{jk}$, given by
\begin{align}\label{weights}
w_{jk}=\frac{\left(\sigma^{2}_{\left(k-1\right)h_{n}}+\frac{\eta^{2}}{n}\left[\varphi_{jk},\varphi_{jk}\right]_{n}\right)^{-2}}{\sum_{m=1}^{\nh}\left(\sigma^{2}_{\left(k-1\right)h_{n}}+\frac{\eta^{2}}{n}\left[\varphi_{mk},\varphi_{mk}\right]_{n}\right)^{-2}}\,.
\end{align}
The empirical scalar products $\left[f,g\right]_{n}$, for any functions $f$ and $g$, are given by
\begin{align*}
\left[f,g\right]_{n}=\frac{1}{n}\sum_{i=1}^{n}f\left(\frac{i-\frac{1}{2}}{n}\right)g\left(\frac{i-\frac{1}{2}}{n}\right).
\end{align*}
The order in \eqref{eq:hn} ensures that the error by discretization of the signal part and the error due to noise are balanced.\\
In a second step we split the observation interval $\left[0,1\right]$ by some ``big blocks'' with length $\ah$:
\begin{align*}
\left[i\ah,\left(i+1\right)\ah\right],\quad i=0,\ldots,{\lfloor \left(\ah\right)^{-1}\rfloor}-1,
\end{align*}
where $\left(\an\right)_{n\in\mathbb{N}}$ is some $\mathbb{N}$-valued sequence fulfilling as $n\rightarrow +\infty$:
\begin{align}\label{eq:an2}
\sqrt{\an}\left(\ah\right)^{\mathfrak{a}}\sqrt{\log\left(n\right)}\longrightarrow 0 \text{ and }h_{n}^{-\varpi}/\an\longrightarrow 0
\end{align}
for some $\varpi>0$ and the regularity exponent $\mathfrak{a}\in\left(0,1\right]$ under the null hypothesis \ref{hypo}.
Using spectral estimators and averaging within each big block $\left[i\ah,\left(i+1\right)\ah\right]$ provides a \emph{consistent} estimator for $\sigma^{2}_{i\ah}$:
\begin{align}\label{rv}
\overline{RV}_{n,i}=\frac{1}{\an}\sum_{\ell=1}^{\an}\hat{\sigma}^{2}_{\ahlla},\quad i=0,\ldots,{\lfloor \left(\ah\right)^{-1}\rfloor}-1\,.
\end{align}
A feasible adaptive estimation is obtained by a two-stage method where $\hat\eta^2$ from \eqref{etahat} and
\begin{align}\label{pilotsigmahat}\frac{1}{\an}\sum_{l=k-\an\vee 0}^{(k-1)\vee (\an-1)}\,\sum_{j=1}^J\frac{1}{J}\left(S^{2}_{jl}\left(Y\right)-\left[\varphi_{jl},\varphi_{jl}\right]_{n}\frac{\hat\eta^{2}}{n}\right)=\sigma^2_{(k-1)h_n}+\KLEINO_{\P}(1)\end{align}
are inserted in the oracle weights to derive feasible estimated weights $\hat w_{jk}$. The result \eqref{pilotsigmahat} has been established and used in previous works on spectral volatility estimation, see \cite{BibingerWinkel2018}. The pilot volatility estimator \eqref{pilotsigmahat} is an average of squared bias corrected spectral statistics over $J$ Fourier frequencies and $\an$ bins. For some fix $J\in\N$ and an optimal choice of $\an\propto n^{\aalpha/(2\aalpha+1)}/\log(n)$, it renders a rate-optimal estimator for which the $\KLEINO_{\P}(1)$-term in \eqref{pilotsigmahat} is $\mathcal{O}_{\P}(n^{-\aalpha/(4\aalpha+2)})$. A sub-optimal choice of $\an$ will not affect our results, however. Other weights than \eqref{weights} do not yield an asymptotically efficient estimator with minimal asymptotic variance. With estimated versions of the optimal weights \eqref{weights}, \cite{Altmeyer2015} show that a Riemann sum over the estimates \eqref{sigmahat} yields a quasi-efficient estimator for the integrated squared volatility. Hence, we use the statistics \eqref{sigmahat} with exactly these weights and the orthogonal sine basis $(\Phi_{jk})$ motivated by the efficiency results of \cite{Reis2011}. Finally, with adaptive versions of the local volatility estimators \eqref{rv}
\begin{align}\label{rvad}
\overline{RV}_{n,i}^{ad}&=\frac{1}{\an}\sum_{\ell=1}^{\an}\hat{\sigma}^{2,ad}_{\ahlla},\quad i=0,\ldots,{\lfloor \left(\ah\right)^{-1}\rfloor}-1\,,\\
\notag \hat{\sigma}^{2,ad}_{(k-1)h_n}&=\sum_{j=1}^{\nh}\hat w_{jk}\left(S^{2}_{jk}\left(Y\right)-\left[\varphi_{jk},\varphi_{jk}\right]_{n}\frac{\hat\eta^{2}}{n}\right)\,,
\end{align} 
our test statistic is given by
\begin{align}\label{eq:mainVn}
\overline{V}_{n}=\mathop{\mathrm{max}}\limits_{i=0,\ldots,{\lfloor \left(\ah\right)^{-1}\rfloor}-2}\left|\frac{\overline{RV}_{n,i}^{ad}-\overline{RV}_{n,i+1}^{ad}}{\sqrt{8\hat\eta}\;\big|\overline{RV}^{ad}_{n,i+1}\big|^{3/4}}\right|\,,
\end{align}
where $\hat\eta=\sqrt{\hat\eta^2}$, with $\hat\eta^2$ from \eqref{etahat}.
We write the absolute value in the denominator, since due to the bias correction in \eqref{sigmahat} the statistics $(\overline{RV}_{n,i})$ and $(\overline{RV}^{ad}_{n,i}),i=0,\ldots,{\lfloor \left(\ah\right)^{-1}\rfloor}-1$ are not guaranteed to be positive. 
\begin{rem} \rm
\begin{enumerate}[(1)]
\item The construction of the test statistic \eqref{eq:mainVn} is based on the idea to compare the values of the spot volatility process $\left(\sigma^{2}_{t}\right)_{t\in\left[0,1\right]}$ on intervals $\left[i\ah,\left(i+1\right)\ah\right]$ and 
$\left[\left(i+1\right)\ah,\left(i+2\right)\ah\right]$ and to reject the null hypothesis of no jumps, if the test statistic $\overline{V}_{n}$ fulfills $\overline{V}_{n}\geq c_{n}$ for some accurate sequence $c_{n}$.
\item The statistic \eqref{eq:mainVn} significantly differs from the statistic $V_{n}$ given in Equation (13) of \cite{Bibinger2017} beyond replacing spot volatility estimates by noise-robust spot volatility estimates. Though both statistics are quotients, the underlying structure of them is different. Whereas in \cite{Bibinger2017} the simple structure of the (asymptotic) variance of spot volatility estimates allows to use statistics based on their quotients, \eqref{eq:mainVn} is based on differences rescaled with their estimated variances. The statistics which are used to wipe out the influence of the noise process imply that volatility does not simply ``cancel out'' in our case as in Proposition A.3 of \cite{Bibinger2017}. The construction of \eqref{eq:mainVn} is particularly appropriate from an implementation point of view, since it scales to obtain an asymptotic distribution-free test and makes it possible to avoid pre-estimation of higher order moments.\qed
\end{enumerate}
\end{rem}
In order to increase the performance of the statistic, we also include a statistic $\overline{V}^{ov}_{n}$ based on \emph{overlapping} big blocks:
\begin{align}\label{eq:mainVno}
\overline{V}^{ov}_{n}=\mathop{\mathrm{max}}\limits_{i=\an,\ldots,h_{n}^{-1}-\an}\left|\frac{\overline{RV}^{ov}_{n,i}-\overline{RV}^{ov}_{n,i+\an}}{\sqrt{8\hat\eta}\;\big|\overline{RV}^{ov}_{n,i+\an}\big|^{3/4}}\right|.
\end{align}
with $\overline{RV}^{ov}_{n,i}$ given by
\begin{align*}
\overline{RV}^{ov}_{n,i}=\frac{1}{\an}\sum_{\ell=i-\an+1}^{i}\hat{\sigma}^{2,ad}_{\left(\ell-1\right)\hn},\quad i=\an,\ldots,h_{n}^{-1}\,.
\end{align*}

\subsection{The discontinuous case\label{subsec:3.2}}
In this paragraph, we generalize the method to be robust in the presence of jumps in \eqref{itoprocess}. When $(\sigma_t)_{t\in[0,1]}$ is our target of inference, the jumps are a nuisance quantity. In order to eliminate jumps of $(X_t)_{t\in[0,1]}$ in the approach, we consider truncated spot volatility estimates
\begin{align}\label{rvtr}
\overline{RV}_{n,i}^{tr}=\frac{1}{\an}\sum_{\ell=i-\an+1}^{i}\hat{\sigma}^{2,ad}_{\left(\ell-1\right)\hn}\mathbbm{1}_{\{|\hat{\sigma}^{2,ad}_{\left(\ell-1\right)\hn}|\le h_n^{\tau-1}\}},\quad i=\an,\ldots,h_n^{-1}\,,
\end{align}
with a truncation exponent $\tau\in(0,1)$. Truncated volatility estimators have been introduced first for integrated volatility estimation by \cite{mancini} and \cite{jacodjumps}. We define the test statistics with the truncated spot volatility estimates \eqref{rvtr}
\begin{subequations}
\begin{align}
\label{test}\overline{V}_{n}^{\tau}&=\mathop{\mathrm{max}}\limits_{i=1,\ldots,{\lfloor \left(\ah\right)^{-1}\rfloor}-1}\left|\frac{\overline{RV}_{n,i\an}^{tr}-\overline{RV}_{n,(i+1)\an}^{tr}}{\sqrt{8\hat\eta}\;\big|\overline{RV}^{tr}_{n,(i+1)\an}\big|^{3/4}}\right|\,,\\
\label{testo}\overline{V}^{ov,\tau}_{n}&=\mathop{\mathrm{max}}\limits_{i=\an,\ldots,h_n^{-1}-\an}\left|\frac{\overline{RV}_{n,i}^{tr}-\overline{RV}_{n,i+\an}^{tr}}{\sqrt{8\hat\eta}\;\big|\overline{RV}^{tr}_{n,i+\an}\big|^{3/4}}\right|\,.
\end{align}
\end{subequations}
\section{Asymptotic theory}\label{sec:4}
\subsection{Limit theorem under the null hypothesis\label{subsec:4.1}}
The hypothesis test formulated in Section \ref{sec:2} is based on asymptotic results for the statistics $\overline{V}_{n}$ and $\overline{V}^{ov}_{n}$, constructed in Section $\ref{sec:3}$.
\begin{theo} \label{thm:1}
Set $m_n = \lfloor \left(\ah\right)^{-1}\rfloor $, $\gamma_{m_n} = [4 \log(m_n) - 2 \log(\log (m_n))]^{1/2}$ and assume that $J_{t}\equiv 0$. If Assumptions \ref{assVola} and \ref{assnoise} hold and $\an$ satisfies condition \eqref{eq:an2}, then we have under \ref{hypo} that
\begin{align}\label{eq:limitnov}
\sqrt{\log\left(m_{n}\right)}\left(\sqrt{\an}\overline{V}_{n}-\gamma_{m_{n}}\right)\overset{\mathrm{d}}{\longrightarrow}V
\end{align}
where $V$ follows an extreme value distribution with distribution function
\begin{align*}
\P(V \leq x) = \exp(-\pi^{-1/2} \exp(-x)).\
\end{align*}
\end{theo}
Theorem \ref{thm:1} is a key tool tackling the testing problem which is based on non-overlapping big blocks. The following result covers the case of overlapping big blocks. 
\begin{cor} \label{cor:1}
Given the assumptions of Theorem \ref{thm:1}, the following weak convergence holds under \ref{hypo}:
\begin{align}\label{eq:limitov}
\sqrt{\log\left(m_{n}\right)}\sqrt{\an}\overline{V}^{ov}_{n}-2\log\left(m_{n}\right)-\frac{1}{2}\log\left(\log\left(m_{n}\right)\right)-\log\left(3\right)\overset{\mathrm{d}}{\longrightarrow}V,
\end{align}
with $V$ as in Theorem \ref{thm:1}.
\end{cor}
We extend this result to the setup with jumps in $(X_t)_{t\in[0,1]}$ when using truncated functionals.
\begin{prop} \label{prop:withjumps}
Let $m_n$ and $\gamma_{m_n} $ be the sequences defined in Theorem \ref{thm:1}. Suppose $\an=\kappa h_n^{-\beta}$ for a constant $\kappa$ and with $0<\beta<1$, Assumption \ref{assVola}, Assumption \ref{assnoise} and Assumption \ref{assjumps} with 
\begin{align}\label{resjumps}r<\min{\Big(2-\frac{\beta}{\tau},2\tau^{-1}(1-\beta),\tau^{-1},\frac34\big(1+\tau-\frac{\beta}{2}\big)\Big)}\,.\end{align} 
Then we have under \ref{hypo} that
\begin{subequations}
\begin{align}
\label{eq:limitnovj}\sqrt{\log\left(m_{n}\right)}\left(\sqrt{\an}\,\overline{V}^{\tau}_{n}-\gamma_{m_{n}}\right)\overset{\mathrm{d}}{\longrightarrow}V\,,\\
\label{eq:limitovj}\sqrt{\log\left(m_{n}\right)}\sqrt{\an}\,\overline{V}^{ov,\tau}_{n}-2\log\left(m_{n}\right)-\frac{1}{2}\log\left(\log\left(m_{n}\right)\right)-\log\left(3\right)\overset{\mathrm{d}}{\longrightarrow}V\,,
\end{align}
\end{subequations}
with $V$ as in Theorem \ref{thm:1}.
\end{prop}
It is natural that we derive the same limit results as above, since the truncation aims to eliminate the nuisance jumps. Proposition \ref{prop:withjumps} gives rather minimal conditions, in particular \eqref{resjumps}, under that we can guarantee that the truncation works in this sense. 
\begin{rem}\label{remjumps}Condition \eqref{resjumps} ensures that different error terms in the proof of Proposition \ref{prop:withjumps} are asymptotically negligible. Though we state it in terms of upper bounds on the jump activity $r$, it rather puts restrictions on the interplay between $r$, $\tau$ and $\beta$. Given $\aalpha$ from \ref{hypo}, we choose $\beta$ close to $2\aalpha/(2\aalpha+1)$ to attain the highest possible power of the test. This results in $0<\beta<2/3$, where the case $\beta\approx 1/2$ for $\aalpha=1/2$ appears the most relevant one including a test for jumps in a semimartingale volatility process. Rewritten in terms of bounds on $\tau$, \eqref{resjumps} gives:
\[\max\Big(\frac{\beta}{2-r},\frac43 r+\frac{\beta}{2}-1\Big)<\tau<\min\Big(r^{-1},2r^{-1}(1-\beta)\Big)\,.\]
For finite activity, $r=0$, we only have mild lower bounds on the choice of $\tau$. Usually, a choice of $\tau$ close to 1 is advocated in previous works on truncated volatility estimation. For $\beta\approx 1/2$, this requires $r<1$. The different error terms under noise for the maximum obtained here actually suggest that $\tau=3/4$ is an even better choice when we require only $r<4/3$. Overall, the conditions on the jumps are not much more restrictive than required for central limit theorems of linear volatility estimators, see Chapter 13 of \cite{JP}. Compared to Proposition 3.5 of \cite{Bibinger2017}, we relax the conditions on $(J_t)$ by a more sophisticated strategy of our proof. In particular, we do not have to restrict to a L\'{e}vy-type process with independent increments, since we work with Doob's submartingale maximal inequality instead of Kolmogorov’s maximal inequality. With this strategy it is also possible to generalize the result in Proposition 3.5 of \cite{Bibinger2017}.
\end{rem}
\subsection{Key ideas of the proof of the limit results\label{subsec:4.2}}
Since the proofs of the results stated in Section \ref{subsec:4.1} are quite long, we want to sketch the key ideas of the proof shortly. The details are worked out in Section \ref{sec:6}.\\
Starting with the \textit{continuous case}, for the results given in Theorem \ref{thm:1} and Corollary \ref{cor:1}, the main ingredients are described as follows. In the \textit{first step} we carry out the crucial approximation where we show that the error, replacing the true log-price increments of $(X_t)_{t\in[0,1]}$ by Brownian increments multiplied with a locally constant approximated volatility, is negligible. More precisely, we show that the spectral statistics $S_{jk}\left(Y\right)$ are adequately approximated through $\sigma_{{\lfloor \an^{-1}\left(k-1\right)\rfloor}\ah}S_{jk}\left(W\right)+S_{jk}\left(\varepsilon\right)$ with the volatility approximated constant over the big blocks. The analogues of $\overline{RV}_{n,i}$ after the approximation are denoted $\overline{Z}_{n,i}$, given in \eqref{eq:Zni}.\\ 
In the \textit{second step}, we conduct a time shift with respect to the volatility in $\overline{Z}_{n,i+1}$ to approximate the volatility by the same constant in the differences $\overline{Z}_{n,i}-\overline{Z}_{n,i+1}$.\\
The \textit{third step} is to replace the estimated asymptotic standard deviation in the denominator in \eqref{eq:mainVn} by its stochastic limit. The latter step is essentially completed by a Taylor expansion. Finally, we establish in a \textit{fourth step} that the difference between the statistics using \eqref{rv} with oracle weights and the statistics using \eqref{rvad} with adaptive weights is sufficiently small to extend the results to the feasible statistics.\\
The approximation steps combine Fourier analysis for the spectral estimation with methods from stochastic calculus. Disentangling the approximation errors of maximum statistics requires a deeper study than for linear statistics. After an appropriate decomposition of the terms, we frequently use Burkholder, Jensen, Rosenthal and Minkowski inequalities to derive upper bounds.\\
The final step is to apply strong invariance principles by \cite{KMT2} and to apply results from \cite{sakhanenko} to conclude with Lemma 1 and Lemma 2, respectively, in \cite{wuzhao2007}. Concerning the non-overlapping statistics we need Lemma 1, whereas the overlapping case needs the more involved limit result presented in Lemma 2 of \cite{wuzhao2007}.\\
In order to prove Proposition \ref{prop:withjumps}, we show that under the stated conditions the jump robust statistics provide the same limit as in the continuous case. That is, the jumps do not affect the limit at all. We decompose the additional error term by truncation in several terms of different structure which we prove to be asymptotically negligible under the mild conditions \eqref{resjumps} on the jump activity and its interplay with the truncation and smoothing parameters. We use Doob's maximal submartingale inequality to bound one crucial remainder without imposing a more restrictive L\'{e}vy structural assumption as has been used in \cite{Bibinger2017}.
\subsection{Rejection rules and consistency\label{subsec:4.3}}
Based on the limit results presented in Section \ref{subsec:4.1}, we can summarize the following rejection rules.
Thereto, let $c_{\alpha}$ be the $(1-\alpha)$-quantile of the Gumbel-type limit law $\mathbb{P}_{V}$ of $V$ in the limit theorems. Since the latter is absolutely continuous with respect to the Lebesgue measure, there is a unique solution, given by
\begin{align*}
c_{\alpha}=-\log\left(-\log\left(1-\alpha\right)\right)-\frac{1}{2}\log\left(\pi\right).
\end{align*}
\begin{enumerate}
\item[(R)]Based on Theorem \ref{thm:1} and the notations used there, we
\begin{align}\label{rejectrule}
\text{reject }~ H_0\text{-}\mathfrak{a}\text{ if }~ \overline{V}_{n}\geq \an^{-1/2}\left(\left(\log(m_{n})\right)^{-1/2} c_{\alpha}+\gamma_{m_{n}}\right)\,.
\end{align}
\item[($\text{R}^{ov}$)] Based on Corollary \ref{cor:1} and the notations used there, we
\begin{align}\label{rejectrule_ov}
\text{reject }~ H_{0}\text{-}\mathfrak{a}\text{ if }~ \overline{V}^{ov}_{n}\geq \frac{\left(c_{\alpha}+2\log\left(m_{n}\right)+\frac{1}{2}\log\left(\log\left(m_{n}\right)\right)+\log\left(3\right)\right)}{\left(\log(m_{n})\an\right)^{1/2}}\,.
\end{align}
\item[($\text{R}^{\tau}$)]Based on Proposition \ref{prop:withjumps} and the notations used there, we
\begin{align}\label{rejectrule_jump}
\text{reject }~ H_{0}\text{-}\mathfrak{a}\text{ if }~ \overline{V}^{\tau}_{n}\geq \an^{-1/2}\left(\left(\log(m_{n})\right)^{-1/2} c_{\alpha}+\gamma_{m_{n}}\right)\,.
\end{align}
\item[($\text{R}^{ov,\tau}$)]Based on Proposition \ref{prop:withjumps} and the notations used there, we
\begin{align}\label{rejectrule_jump_ov}
\text{reject }~ H_{0}\text{-}\mathfrak{a}\text{ if }~ \overline{V}^{ov,\tau}_{n}\geq \frac{\left(c_{\alpha}+2\log\left(m_{n}\right)+\frac{1}{2}\log\left(\log\left(m_{n}\right)\right)+\log\left(3\right)\right)}{\left(\log(m_{n})\an\right)^{1/2}}\,.
\end{align}
\end{enumerate}
\begin{theo}\label{thm:2}Suppose Assumption \ref{assVola}, Assumption \ref{assnoise}, and Assumption \ref{assjumps} with \eqref{resjumps} in the case with jumps.
The decision rules \eqref{rejectrule}, \eqref{rejectrule_ov}, \eqref{rejectrule_jump} and \eqref{rejectrule_jump_ov} provide consistent tests to distinguish the null hypothesis \ref{hypo} from the alternative hypothesis \ref{alt} for the testing problem \eqref{hyptest}.
\end{theo} 
Consistency of the test means that under the alternative hypothesis, if for some $\theta\in\left(0,1\right)$ we have that $\big|\sigma^{2}_{\theta}-\sigma^{2}_{\theta-}\big|=\delta>0$ for some fix $\delta>0$, the power of the test, for instance by \eqref{rejectrule}, tends to one as $n\rightarrow \infty$:
\[\P_{H_1}\big(\overline{V}_{n}\geq \an^{-1/2}\big(\left(\log(m_{n})\right)^{-1/2} c_{\alpha}+\gamma_{m_{n}}\big)\big)\stackrel{n\rightarrow\infty}\longrightarrow 1\,.\]
Theorem \ref{thm:1} ensures that \eqref{rejectrule} facilitates an asymptotic level-$\alpha$-test that correctly controls the type 1 error, that is,
\[\P_{H_0}\big(\overline{V}_{n}\geq \an^{-1/2}\big(\left(\log(m_{n})\right)^{-1/2} c_{\alpha}+\gamma_{m_{n}}\big)\big)\stackrel{n\rightarrow\infty}\longrightarrow \alpha\,.\]
Thereby, even for small $\aalpha>0$, the test can distinguish continuous volatility paths from paths with jumps. 
\begin{rem}The rate $\sqrt{\log(m_{n})\an}$ in \eqref{eq:limitnov}, \eqref{eq:limitov}, \eqref{eq:limitnovj} and \eqref{eq:limitovj} determines how fast the power of the test increases in the sample size $n$. The convergence rate, for $\an$ close to the upper bound in \eqref{eq:an2}, is close to $n^{\aalpha/(4\aalpha+2)}$. The latter coincides with the optimal convergence rate for spot volatility estimation under noise, see \cite{munk2010}. In light of the lower bound for the testing problem without noise established in \cite{Bibinger2017} and the relation of the models with and without noise studied in \cite{Gloter2001}, we conjecture that the above test yields an asymptotic minimax-optimal decision rule. A formal generalization of the proof for the detection boundary from Theorem 4.1 of \cite{Bibinger2017} to our setting however appears not to be feasible, since it heavily exploits simple $\chi^2$-approximations of squared increments. 
\end{rem}
\subsection{Consistent estimation of the change point\label{subsec:4.4}}In this subsection, we present an estimator for the change point $\theta$, which is of importance, once we have decided to reject \ref{hypo}. Therefore, we suppose \ref{alt} and that there exists one $\theta\in(0,1)$ with $|\Delta\sigma^2_{\theta}|>0$. The aim is to estimate $\theta$, in general referred to as the change point or break date in change-point statistics, which here gives the time of the volatility jump. We suggest the estimator $\hat{\theta}_{n}$, given by
\begin{align}
\hat{\theta}_{n}=\hn\operatorname{argmax}_{i=\an,\ldots, \hn^{-1}-\an} \overline{V}_{n,i}^{\diamond}\label{hattheta}
\end{align}
where
\begin{align*}
\overline{V}_{n,i}^{\diamond}=\an^{-1/2}\left|\sum_{\ell=i-\an+1}^{i}\hat{\sigma}^{2,ad}_{(\ell-1)\hn}-\sum_{\ell=i+1}^{i+\an}\hat{\sigma}^{2,ad}_{(\ell-1)\hn}\right|\,.
\end{align*}
It is sufficient to use these modified non-rescaled versions of the statistics in \eqref{eq:mainVno}. We prove the following consistency result for our estimator.
\begin{prop}\label{propCP}
Given the assumptions of Theorem \ref{thm:1}, that is, Assumptions \ref{assVola} and \ref{assnoise}, $J_t\equiv 0$ and $\an$ satisfies \eqref{eq:an2}, and assume that \ref{alt} applies with one jump time  $\theta\in(0,1)$. For $\Delta\sigma^2_{\theta}=\delta\ne 0$, it holds that
\begin{align*}
\big|\hat{\theta}_{n}-\theta\big|=\mathcal{O}_{\P}\big(h_n|\delta^{-1}|\sqrt{\an\log(n)}\big)\,.
\end{align*}
In particular, $\hat{\theta}_{n}\overset{\mathbb{P}}{\longrightarrow}\theta\,.$
\end{prop}
\begin{rem}
Put another way, we can detect jump times associated with sequences of jump sizes $\delta_n\to 0$ as $n\to+\infty$ as long as $h_n^{-1}(\an\log(n))^{-1/2}=\KLEINO(\delta_n)$ in the sense of weak consistency. Choosing $\an$ as small as possible, such that \eqref{eq:an2} is satisfied, yields the best possible rate, while for the testing problem in Theorem \ref{thm:1} we select $\an$ as large as possible. In the optimal case, a jump with fix size $\delta\ne 0$ can be detected with a convergence rate close to $h_n^{-1}$. This provides important information how precisely volatility jump times can be located under noisy observations. With jumps in $(X_t)$, we conjecture that an analogous results holds true under the conditions of Proposition \ref{prop:withjumps}. A sequential application of our methods allows for testing and the estimation of multiple change points. The extension of the estimation from the one change to the multiple change-point alternative is accomplished similarly to Algorithm 4.9 from paragraph 4.2.2.\ in \cite{Bibinger2017}.
\end{rem}

\section{Simulations and a bootstrap adjustment\label{sec:5}}
In this section we investigate the finite-sample performance of the new method in a simulation study. We also analyze the efficiency gains of our noise-robust approach based on the spectral volatility estimation methodology in comparison to simply skip sampling the data and applying the non noise-robust method from \cite{Bibinger2017}. Skip sampling the data, which means we only consider every 60th datapoint, reduces the dilution by the noise and is a standard way to deal with high-frequency data in practice. We consider $n=30,000$ observations of \eqref{smn}, a typical sample size of high-frequency returns over one trading day. The noise is centered and normally distributed with a realistic magnitude, $\eta=0.005$, see, for instance, \cite{BHMR}.
We implement the same volatility model as in Section 5 of \cite{Bibinger2017}, where
\begin{align}\label{volasim2}\sigma_t=\left(\int_0^t c\cdot\rho \,dW_s+\int_0^t \sqrt{1-\rho^2}\cdot c \,dW_s^{\bot}\right)\cdot v_t\,\end{align}
is a semimartingale volatility process fluctuating around the seasonality function
\begin{align}\label{volasim}v_t=1-0.2\sin\big(\tfrac{3}{4}\pi\,t),\,t\in[0,1]\,,\end{align}
\begin{figure}
\fbox{
\includegraphics[width=6.75cm]{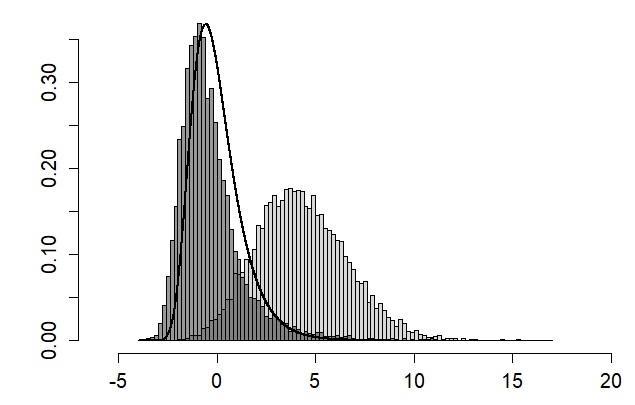}~\includegraphics[width=6.75cm]{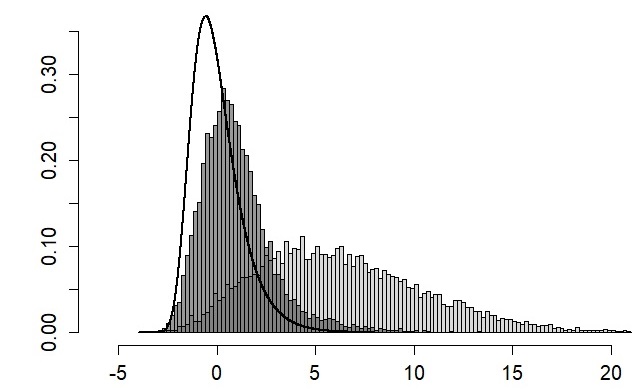}}
\caption{\label{Fig:1}Left: Histogram of statistics left-hand side in \eqref{eq:limitovj} for $h_n^{-1}=120$ and $\alpha_n=15$, $n=30,000$, under null hypothesis and alternative hypothesis and limit law density marked by the line. Right: Histogram of corresponding not noise-robust statistics from \cite{Bibinger2017} applied after (most efficiently) skip sampling to a subset with $n_{skip}=500$ observations, $k_n=125$, and limit law density marked by the line.}
\end{figure}
\hspace*{-.15cm}where $c=0.1$ and $\rho=0.5$, with $W^{\bot}$ a standard Brownian motion independent of $W$. We set $X_0=4$ and the drift $a=0.1$. We perform the simulations in R using an Euler-Maruyama discretization scheme.\\

\subsection{Performance of the test, comparison to skip sampling, bootstrap adjustment and sensitivity analysis\label{subsec:5.1}}
Concerning the jumps of $(X_t)$ and $(\sigma_t)$ under the alternative hypothesis, we implement two different model configurations. In order to grant a good comparison to \cite{Bibinger2017} in the evaluation of the efficiency gains by our method instead of a skip-sample approach, we adopt in Section \ref{subsec:5.1} the setup from Section 5 of \cite{Bibinger2017}. There, under the alternative hypothesis, the volatility admits one jump of size $0.2$ at time $t=2/3$. The jump size equals the range of the expected continuous movement. Under the alternative hypothesis, $(X_t)$ admits a jump at the same time $t=2/3$. Under the null hypothesis and the alternative hypothesis, $(X_t)$ also jumps at some uniformly drawn time. All price jumps are normally distributed with expected size $0.5$ and variance $0.1$. More general jumps are considered in Section \ref{subsec:5.2}.

We consider the test statistic \eqref{testo} with overlapping blocks and truncation. Section \ref{subsec:5.2} confirms that it outperforms the non-overlapping version \eqref{test}. We set $h_n^{-1}=120$ and $\an=15$. Robustness with respect to different choices of $h_n$ and $\an$ is discussed below. For the truncation, we set $\tau=3/4$ according to Remark \ref{remjumps}. In all cases, we compute the adaptive feasible statistics and do not make use of the generated volatility paths to derive the weights \eqref{weights}. We rather rely on the two-stage method and insert \eqref{pilotsigmahat} with $J=20$ and \eqref{etahat} in the statistics. The spectral estimates from \eqref{sigmahat} are computed as sums up to the spectral cut-off $J_n=50$, smaller than $\lfloor nh_n\rfloor -1=254$, as the fast decay of the weights \eqref{weights} in $j$, compare also \eqref{wOrder}, renders higher frequencies completely negligible. The investigated test statistics will be identically feasible in data applications.
 
Figure \ref{Fig:1} visualizes the empirical distribution from $10,000$ Monte Carlo iterations under the null hypothesis and the alternative hypothesis. The left plot shows our statistics while the right plot gives the results for the statistics from \cite{Bibinger2017} applied to a skip sample of 500 observations. The skip-sampling frequency has been chosen to maximize the performance of these statistics. While they are reasonably robust to minor modifications, too large samples lead to an explosion of the statistics also under the null hypothesis and much smaller samples result in poor power. The length of the smoothing window $k_n$ for the statistics given in Equation (24) of \cite{Bibinger2017} is set $k_n=125$, adopted from the simulations in \cite{Bibinger2017}. In the optimal case, null and alternative hypothesis are reasonably well distinguished by the skip-sampling method -- but the two plots confirm that our approach improves the finite-sample power considerably. For the spectral approach, $88\%$ of the outcomes under $H_1$ exceed the $90\%$-decile of the empirical distribution under $H_0$. For the optimized skip-sample approach this number reduces to $75\%$. The approximation of the limit law appears somewhat imprecise. The relevant high quantiles, however, fit their empirical counterparts quite well.\\
Nevertheless, we propose a bootstrap procedure to fit the distribution of $\overline{V}^{ov,\tau}_{n}$ under $H_0$ with improved finite-sample accuracy. We start with an estimator for the spot volatility $\overline{RV}_{n,i}^{tr},\quad i=\an,\ldots,h_n^{-1}$, from $\eqref{rvtr}$, using the same $h_n^{-1}$ and $\an$ as for the test. We also define and compute $\overline{RV}_{n,i}^{tr},\quad i=1,\ldots,\an-1$, averaging over the available number of blocks, smaller than $\an$, back in time. In order to smooth the random fluctuations of the spot volatility pre-estimates, we apply a filter to the estimates of length 30 with equal weights and denote $\tilde{\sigma}_{n,i}^2,\quad i=1,\ldots,h_n^{-1}$, the resulting estimated volatility path. At the boundaries we interpolate linearly to $\overline{RV}_{n,1}^{tr}$ and $\overline{RV}_{n,h_n^{-1}}^{tr}$, respectively. Repeating each entry $nh_n=250$ times, we obtain a (bin-wise constant) estimator $\tilde{\sigma}_{n,i}^2,\quad i=1,\ldots,n$. For two sequences of i.i.d.\ standard normals $\{Z_i\}_{1 \leq i \leq n},\{E_i\}_{1 \leq i \leq n}$, and $X_0^*=Y_0^*=Y_0$, denote with
\[X_i^*=X_{i-1}^*+\sqrt{\frac{\tilde{\sigma}_{n,i}^2}{n}}\cdot Z_i~,Y_i^*=X_i^*+\hat\eta \cdot E_i~,1\le i\le n\,,\]
a pseudo path $Y^*$ generated with the estimated volatility path and estimated noise variance and the $(Z_i,E_i)$. We can iterate the procedure as a Monte Carlo simulation and produce $N=10,000$ different pseudo paths $Y^*$ using independent generalizations of random variables $(Z_i,E_i)$. With
\begin{align*}
\hat{\sigma}^{2*}_{\left(k-1\right)h_{n}}&=\sum_{j=1}^{\nh}\hat w_{jk}\left(S^{2}_{jk}\left(Y^*\right)-\left[\varphi_{jk},\varphi_{jk}\right]_{n}\frac{\hat\eta^{2}}{n}\right)\,,\\
\overline{RV}_{n,i}^{*}&=\frac{1}{\an}\sum_{\ell=i-\an+1}^{i}\hat{\sigma}^{2*}_{\left(\ell-1\right)\hn}\mathbbm{1}_{\{|\hat{\sigma}^{2*}_{\left(\ell-1\right)\hn}|\le h_n^{\tau-1}\}},\quad i=\an,\ldots,h_n^{-1}\,,
\end{align*}
we derive the pseudo test statistic
\[\hat{V}_n^{\dagger}=\mathop{\mathrm{max}}\limits_{i=\an,\ldots,h_n^{-1}-1}\left|\frac{\overline{RV}_{n,i}^{*}-\overline{RV}_{n,i+1}^{*}}{\sqrt{8\hat\eta}\;\big|\overline{RV}^{*}_{n,i+1}\big|^{3/4}}\right|\,.\]
In fact, the truncation with the indicator function is obsolete, since we do not have jumps in the pseudo samples. For a test, we can use the approximative (conditional) quantiles 
\begin{align*}
\hat{q}_{\alpha}(\hat{V}_n^{\dagger}|\F) = \inf\bigl\{x \geq 0\, : \, \P(\hat{V}_n^{\dagger} \leq x | \F) \geq \alpha \bigr\}
\end{align*}
and compute $\hat{q}_{\alpha}(\hat{V}_n^{\dagger}|\F)$ based on Monte Carlo approximation. We reject $H_0$ when
\[\overline{V}^{ov,\tau}_{n}> \hat{q}_{1-\alpha}(\hat{V}_n^{\dagger}|\F)\,.\]

\begin{figure}
\fbox{
\includegraphics[width=6.75cm]{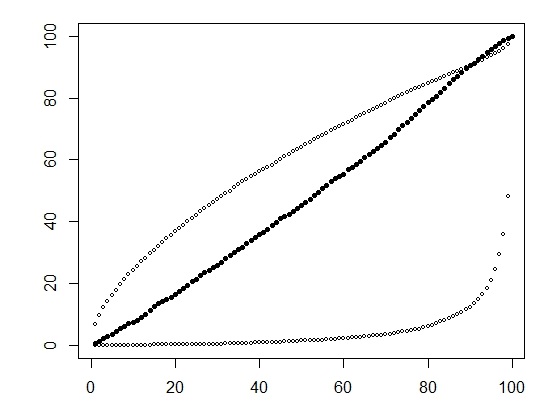}~\includegraphics[width=6.75cm]{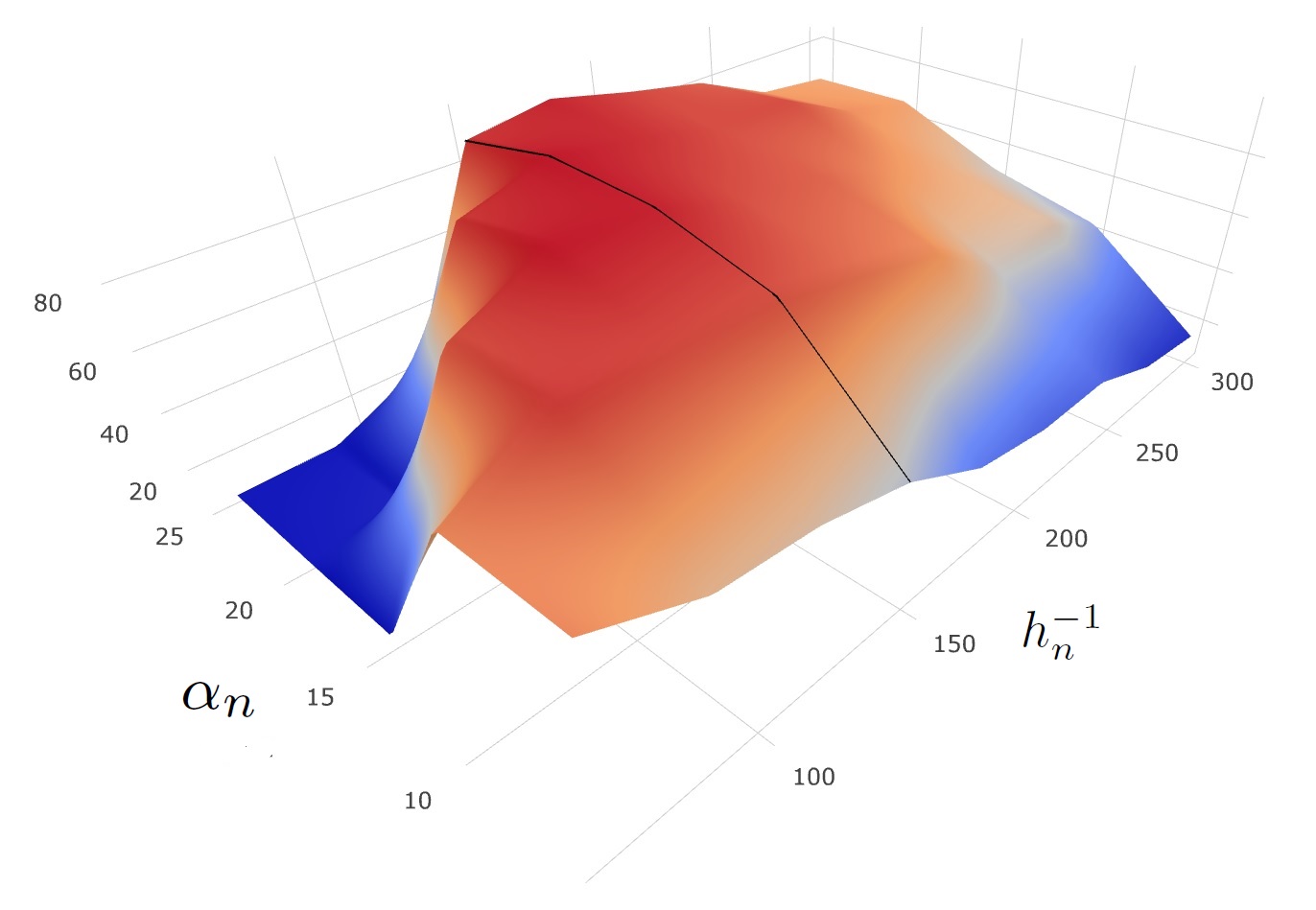}}
\caption{\label{Fig:2}Left: Empirical size and power of the new test for $h_n^{-1}=120$ and $\alpha_n=15$, $n=30,000$, by comparing empirical percentiles to ones of limit law under $H_0$ and $H_1$ (light points). Empirical percentiles compared to bootstrapped percentiles under $H_0$ (dark points). 
Right: Percentage of exceedances under $H_1$ of the $90\%$ empirical quantile under $H_0$ for $h_n^{-1}=60,90,120,150,180,210,240,270,300$ and $\an=5,10,15,20,25$. The line gives the marginal curve for $h_n^{-1}=120$.}
\end{figure}

In the left plot of Figure \ref{Fig:2} the black dots compare the empirical percentiles of the left-hand side in \eqref{eq:limitovj}, the standardized versions of $\overline{V}^{ov,\tau}_{n}$, under $H_0$ to the ones of the bootstrap, i.e. $\hat{q}_{\alpha}(\hat{V}_n^{\dagger}|\F)$. The finite-sample accuracy of the bootstrap for the distribution under $H_0$ is significantly better than the limit law (light points). Since the high percentiles of bootstrap and limit law are quite close, the power of both tests is comparable. For a level $\alpha=10\%$ test, we obtain approx.\,$88\%$ power using the limit law and $89\%$ power using the bootstrap. For a level $\alpha=5\%$ test, we obtain approx.\,$79\%$ and $75\%$, respectively. 

Finally, we consider different parameter configurations $(h_n^{-1},\an)$. Since we can exploit the bootstrap to ensure a good fit under $H_0$, we concentrate on the ability of $\overline{V}^{ov,\tau}_{n}$ to distinguish hypothesis and alternative. To quantify the ability to separate $H_0$ and $H_1$, we visualize the relative number of exceedances under $H_1$ of the $90\%$ empirical quantile under $H_0$. We plot the percentage numbers in the right plot of Figure \ref{Fig:2} over a grid of different values for $(h_n^{-1},\an)$. Additionally, we draw the marginal curve for $h_n^{-1}=120$ at points $\an=5,10,15,20,25$ (black line). Choosing a different (reasonable) quantile under $H_0$ does not change the shape of the surface with respect to the values of $(h_n^{-1},\an)$. Figure \ref{Fig:2} confirms that the test is reasonably robust with respect to different values of the tuning parameters. For $h_n^{-1}$ sufficiently large, setting $\an$ between $15$ and $20$ yields the highest power. For $h_n^{-1}$, values between 60 and 180 grant a good performance. Hence, we choose $h_n^{-1}=120$ and $\an=15$ as suitable configuration for the simulation study. 
\subsection{Comparison of tests with overlapping and non-overlapping statistics\label{subsec:5.2}}
We illustrate the improvement in the power of the test based on \eqref{testo} compared to the non-overlapping version \eqref{test}. Here, we use a prominent general model for jumps of $(X_t)$ often considered in related literature, including \cite{jacodtodorov}, with a predictable compensator \(\nu(ds,dz)=(\1_{\{z\in[-1,-0.2]\cup[0.2,1]\}})/1.6\,dt\,dz\). Since jumps of very small absolute sizes are not generated, the truncation works well and we do not see a manipulation of the empirical distribution of the test statistics due to errors in the truncation step. We investigate the power of the tests for different volatility-jump sizes under the alternative, $\Delta\sigma^2_{\theta}=(10+5\cdot i)/100,i=1,\ldots,7$. The volatility-jump time $\theta$ is randomly generated in each run according to a uniform distribution on $(\an h_n,1-\an h_n)$. Note that not excluding the boundary intervals $[0,\an h_n]\cup [1-\an h_n,1]$ would slightly reduce the power in all configurations, since the test is not able to detect jumps in these boundary blocks. In order to include common price and volatility jumps, we add an additional price jump at $\theta$ with uniformly distributed size as according to $\nu$ above. We keep to the parameters $h_n^{-1}=120$, $\an=15$ and $\tau=3/4$ and compute the adaptive statistics as in Section \ref{subsec:5.1} in $10,000$ iterations.

\begin{figure}
\begin{center}
\fbox{
\includegraphics[width=13cm]{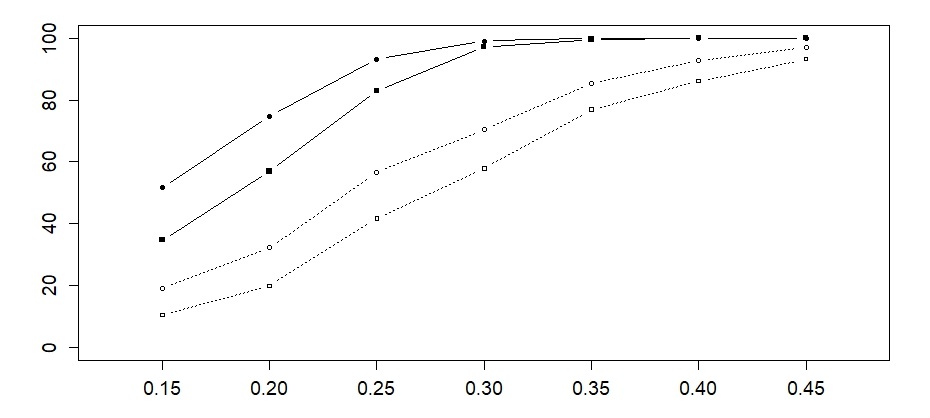}}
\caption{\label{Fig:3}Empirical power of the test based on \eqref{testo} with overlapping statistics (dark, solid) and \eqref{test} with non-overlapping statistics (light, dashed) for the level 10\% (points) and 5\% (squares) under the alternative hypothesis as function of the volatility-jump size $\Delta\sigma^2_{\theta}$. The plot gives empirical percentiles exceeding the bootstrapped percentiles under the null hypothesis for $h_n^{-1}=120$, $\alpha_n=15$ and $n=30,000$.}
\end{center}\end{figure}

Figure \ref{Fig:3} confirms that the test using \eqref{testo} with overlapping statistics has a significantly higher power than the test based on \eqref{test} and non-overlapping statistics. The largest difference for $\Delta\sigma^2_{\theta}=0.2$ is 17.8\% at 10\% testing level and for $\Delta\sigma^2_{\theta}=0.25$, 14.8\% at 5\% testing level. Thus, for volatility jumps with moderate absolute size in the range considered in Figure \ref{Fig:3}, the overlapping statistics attain relevant efficiency gains. The location of the volatility jump -- when the boundaries are excluded -- does not affect the power of the tests. Figure \ref{Fig:3} illustrates increasing power of both tests as $\Delta\sigma^2_{\theta}$ gets larger. It also reveals that volatility jumps with $\Delta\sigma^2_{\theta}\le 0.15$ are difficult to detect in our setting where this corresponds only to approximately 20 times the average absolute increment $|\Delta_i^n Y|$. Due to the required smoothing over blocks we cannot expect to detect such small volatility jumps with good power. We can further report a better accuracy of the theoretical limit law under $H_0$ from Proposition \ref{prop:withjumps} for the empirical distribution of the statistics with overlapping compared to non-overlapping blocks. The average amount of realizations of simulated statistics \eqref{test} exceeding the theoretical 90\%-percentile is 9.99\% and exceeding the 95\%-percentile 6.41\%. For the statistics \eqref{testo} these values are 21.00\% and 11.11\%, respectively.
\section{Proofs}\label{sec:6}
\subsection{Proof of Theorem \ref{thm:1}}
For notational convenience, we replace
\begin{align*}
\mathop{\mathrm{max}}\limits_{i=0,\ldots,{\lfloor \left(\ah\right)^{-1}\rfloor}-2}\quad\text{by}\quad\maxi,
\end{align*}
and
\begin{align*}
\mathop{\mathrm{min}}\limits_{i=0,\ldots,{\lfloor \left(\ah\right)^{-1}\rfloor}-2}\quad\text{by}\quad\mini,
\end{align*}
respectively.\\
We also introduce the following notation, adapting the elements of the spectral statistics on each big block.
Set
\begin{align*}
\Phi_{ij\ell}\left(t\right)&=\Phi_{j}\left(t-\left(\ahlla\right)\right)
\end{align*}
and
\begin{align*}
\varphi_{ij\ell}\left(t\right)&=\varphi_{j}\left(t-\left(\ahlla\right)\right)\,.
\end{align*}
Furthermore, we define the big block-wise spectral statistics 
\begin{align*}
S_{ij\ell}\left(L\right)=\sum_{\nu=1}^{n}\Delta_{i}^{n}\Phi_{ij\ell}\left(\frac{\nu}{n}\right)
\end{align*} and the associated variance minimizing oracle weights
\begin{align*}
w_{ij\ell}=\frac{\left(\sigma^{2}_{\ahlla}+\frac{\eta^{2}}{n}\left[\varphi_{ij\ell},\varphi_{ij\ell}\right]_{n}\right)^{-2}}{\sum_{j=1}^{\nh}\left(\sigma^{2}_{\ahlla}+\frac{\eta^{2}}{n}\left[\varphi_{ij\ell},\varphi_{ij\ell}\right]_{n}\right)^{-2}}\,.
\end{align*}
We further introduce the bias correction terms
\[\mu_{ij\ell}=[\varphi_{ij\ell},\varphi_{ij\ell}]_n\;\frac{\eta^2}{n}\,.\] 
We can strengthen the assumptions, presented in Assumption \ref{assVola} and \ref{hypo}, as follows. We replace local boundedness of $\left(\sigma_{t}\right)_{t\in\left[0,1\right]}$, $\left(a_{t}\right)_{t\in\left[0,1\right]}$, and the modulus of continuity $(w_\delta(\sigma)_t)_{t\in\left[0,1\right]}$ under \ref{hypo} by global boundedness. We refer to Section
4.4.1 of \cite{JP} for a proof and the construction through localization. We set
\begin{align*}
\overline{U}_{n}=\mathop{\mathrm{max}}\limits_{i=0,\ldots,{\lfloor \left(\ah\right)^{-1}\rfloor}-2}\left|\frac{\overline{Z}_{n,i}-\overline{Z}_{n,i+1}}{\big|\overline{Z}_{n,i+1}\big|^{3/4}}\right|\,,
\end{align*}
with 
\begin{align}\label{eq:Zni}
\overline{Z}_{n,i}:=\frac{1}{\an}\sum_{\ell=1}^{\an}\sum_{j=1}^{\nh}w_{ij\ell}\left(\left(\sigma_{i\ah}S_{ij\ell}\left(W\right)+S_{ij\ell}\left(\varepsilon\right)\right)^{2}-\mu_{ij\ell}\right).
\end{align}
Finally we fix some constants $\Kp, \Km>0$, such that almost surely
\begin{align*}
\Km<\inf_{t\in[0,1]}\sigma^{2}_{t}\quad\text{and}\quad\Kp>\sup_{t\in[0,1]}\sigma^{2}_{t}\,.
\end{align*}
The \emph{first step} described in Section \ref{subsec:4.2} is accomplished in the next proposition.
\begin{prop}\label{prop1}
Given the assumptions of Theorem \ref{thm:1}, it holds under \ref{hypo} that
\begin{align*}
\sqrt{\an\log\left(h_n^{-1}\right)}\maxi\left|\bigg|\frac{\Ri-\Rii}{\absRii^{3/4}}\bigg|-\bigg|\frac{\Zi-\Zii}{\absZii^{3/4}}\bigg|\right|\overset{\mathbb{P}}{\longrightarrow}0\,.
\end{align*}
\end{prop}
\noindent
\textit{Proof of Proposition \ref{prop1}}.\\
Since $\hn\propto n^{-1/2}\log(n)$ we can proceed as follows. 
The reverse triangle inequality and the decomposition 
\begin{align*}
&\frac{\overline{RV}_{n,i}-\overline{RV}_{n,i+1}}{\absRii^{3/4}}-\frac{\overline{Z}_{n,i}-\overline{Z}_{n,i+1}}{\absZii^{3/4}}\\
&=\frac{\overline{RV}_{n,i}}{\absRii^{3/4}}-\frac{\overline{RV}_{n,i}}{\absZii^{3/4}}+\frac{\overline{RV}_{n,i}-\overline{Z}_{n,i}}{\absZii^{3/4}}-\frac{\overline{RV}_{n,i+1}}{\absRii^{3/4}}+\frac{\overline{Z}_{n,i+1}-\overline{RV}_{n,i+1}}{\absZii^{3/4}}
\end{align*}
yield the following decomposition:
\begin{align}
&\maxi\left|\overline{RV}_{n,i}\left(\frac{1}{\absRii^{3/4}}-\frac{1}{\absZii^{3/4}}\right)\right|+\maxi\left|\frac{\overline{RV}_{n,i}-\overline{Z}_{n,i}}{\absZii^{3/4}}\right|\nonumber\\
&\quad + \maxi\left|\overline{RV}_{n,i+1}\left(\frac{1}{\absZii^{3/4}}-\frac{1}{\big|\overline{RV}_{n,i+1}\big|^{3/4}}\right)\right|+\maxi\left|\frac{\overline{Z}_{n,i+1}-\overline{RV}_{n,i+1}}{\absZii^{3/4}}\right|\nonumber\\
&=:\left(\textbf{I}\right)+\left(\textbf{II}\right)+\left(\textbf{III}\right)+\left(\textbf{IV}\right)\,.\label{eq:1}
\end{align}
Starting with $\left(\textbf{II}\right)$ in \eqref{eq:1} we proceed as follows.\\
For all $\delta>0$ and $\km>0$, such that $\km\in\left(0,\Km\right)$, the following holds:
\begin{align}
&\mathbb{P}\left[\maxi\left|\frac{\sqrt{\an\log\left(n\right)}\left(\overline{RV}_{n,i}-\overline{Z}_{n,i}\right)}{\absZii^{3/4}}\right|>\delta\right]\nonumber\\
&= \mathbb{P}\left[\maxi\left|\frac{\sqrt{\an\log\left(n\right)}\left(\overline{RV}_{n,i}-\overline{Z}_{n,i}\right)}{\absZii^{3/4}}\right|>\delta,\,\mini\absZii\geq \Kmk\right]\nonumber\\
&+\mathbb{P}\left[\maxi\left|\frac{\sqrt{\an\log\left(n\right)}\left(\overline{RV}_{n,i}-\overline{Z}_{n,i}\right)}{\absZii^{3/4}}\right|>\delta,\,\mini\absZii< \Kmk\right]\nonumber\\
&\leq \mathbb{P}\left[\maxi\sqrt{\an\log\left(n\right)}\left|\overline{RV}_{n,i}-\overline{Z}_{n,i}\right|>\delta\left(\Kmk\right)^{3/4}\right]+\mathbb{P}\left[\mini\,\absZii<\Km\hspace{-0.6em}-\km\right]\nonumber\\
&=:A_{n}+B_{n}.\label{eq:2}
\end{align}
In \eqref{eq:2} we dropped the dependence on the constants $\delta$ and $\Km$ for notational convenience.
We start with the term $A_{n}$. We split the term into various summands in the following way:
\begin{align*}
\overline{RV}_{n,i}-\overline{Z}_{n,i}=&\frac{1}{\an}\sum_{\ell=1}^{\an}\sum_{j=1}^{\nh}w_{ij\ell}\left(S^{2}_{ij\ell}\left(X\right)-\sigma^{2}_{i\ah}S^{2}_{ij\ell}\left(W\right)\right)\\
&\quad +\frac{2}{\an}\sum_{\ell=1}^{\an}\sum_{j=1}^{\nh}w_{ij\ell}S_{ij\ell}\left(\varepsilon\right)\left(S_{ij\ell}\left(X\right)-\sigma_{i\ah}S_{ij\ell}\left(W\right)\right).
\end{align*}
That yields
\begin{align}
\notag A_{n}&\leq\mathbb{P}\left[\maxi\left|\frac{\sqrt{\log\left(n\right)}}{\sqrt{\an}}\sum_{\ell=1}^{\an}\sum_{j=1}^{\nh}w_{ij\ell}\left(S^{2}_{ij\ell}\left(X\right)-\sigma^{2}_{i\ah}S^{2}_{ij\ell}\left(W\right)\right)\right|>\frac{\delta\left(\Kmk\right)^{3/4}}{2}\right]\\
&\notag \quad +\mathbb{P}\left[\maxi\left|\frac{\sqrt{\log\left(n\right)}}{\sqrt{\an}}\sum_{\ell=1}^{\an}\sum_{j=1}^{\nh}w_{ij\ell}\Se\left(S_{ij\ell}\left(X\right)-\sigma_{i\ah}S_{ij\ell}\left(W\right)\right)\right|>\frac{\delta\left(\Kmk\right)^{3/4}}{4}\right]\\
&\label{Adecomp}=:A^{1}_{n}+A^{2}_{n}\,.
\end{align}
In order to handle $A^{1}_{n,}$, we rewrite the spectral statistics $S_{ij\ell}\left(L\right)$, for any stochastic process $\left(L_{t}\right)_{t\in\left[0,1\right]}$, using step functions $\x$, given by
\begin{align*}
\x\left(t\right):&=\sum_{\nu=1}^{n}\Phi_{ij\ell}\left(\frac{\nu}{n}\right)\mathbbm{1}_{\left(\frac{\nu-1}{n},\frac{\nu}{n}\right]}(t)
\end{align*}
which yield
\begin{align*}
S_{ij\ell}\left(L\right)=\int_{\ahll}^{\ahl}\x\left(s\right)\,dL_{s}
\end{align*}
for any semimartingale $L=\left(L_{t}\right)_{t\in\left[0,1\right]}$.\\
By virtue of the It\^{o} process structure of $(X_t)$, we obtain that
\begin{align*}
\int_{\ahll}^{\ahl}\x\left(s\right)\,dX_{s}=\int_{\ahll}^{\ahl}\x\left(s\right)a_{s}\,ds+\int_{\ahll}^{\ahl}\x\left(s\right)\sigma_{s}\,dW_{s}\,.
\end{align*}
It\^{o}'s formula yields
\begin{align*}
S^{2}_{ij\ell}\left(X\right)=&2\int_{\ahll}^{\ahl}\left(\widetilde{X}_{\tau}-\widetilde{X}_{\ahll}\right)\x\left(\tau\right)\sigma_{\tau}\,dW_{\tau}\\
&\quad +2\int_{\ahll}^{\ahl}\left(\widetilde{X}_{\tau}-\widetilde{X}_{\ahll}\right)\x\left(\tau\right)a_{\tau}\,d\tau\\
&\quad +\int_{\ahll}^{\ahl}\left(\x\left(\tau\right)\right)^{2}\sigma^{2}_{\tau}\,d\tau
\end{align*}
with
\begin{align*}
\widetilde{X}_{t}:=X_{0}+\int_{0}^{t}\x\left(s\right)a_{s}\,ds+\int_{0}^{t}\x\left(s\right)\sigma_{s}\,dW_{s}.
\end{align*}
Similarly,
\begin{align*}
S^{2}_{ij\ell}\left(W\right)=2\int_{\ahll}^{\ahl}\widetilde{W}_{\tau}\,\x\left(\tau\right)\,dW_{\tau}+\int_{\ahll}^{\ahl}\left(\x\left(\tau\right)\right)^{2}\,d\tau
\end{align*}
with
\begin{align*}
\widetilde{W}_{t}:=\int_{0}^{t}\x\left(s\right)\,dW_{s}\,.
\end{align*}
For notational brevity, we suppress the dependence of $\widetilde{X}_{t}$ and $\widetilde{W}_{t}$, respectively, on $\left(i,j,\ell,n\right)$.
We bound $A_{n}^{1}$ via
\begin{align*}
A^{1}_{n}\leq A^{1,1}_{n} + A^{1,2}_{n} + A^{1,3}_{n}\,,
\end{align*}
where
\begin{align*}
A^{1,1}_{n}&=\mathbb{P}\bigg[\maxi\bigg|\frac{\sqrt{\log\left(n\right)}}{\sqrt{\an}}\sum_{\ell=1}^{\an}\sum_{j=1}^{\nh}w_{ij\ell}\int_{\ahll}^{\ahl}\left(\widetilde{X}_{\tau}-\widetilde{X}_{\ahll}\right)\\
&\hspace*{6cm}\times \x\left(\tau\right)a_{\tau}\,d\tau\bigg|>\frac{\delta\left(\Kmk\right)^{3/4}}{12}\bigg]\,,
\end{align*}
\begin{align*}
&A_{n}^{1,2}=\mathbb{P}\Bigg[\maxi\bigg|\frac{\sqrt{\log\left(n\right)}}{\sqrt{\an}}\sum_{\ell=1}^{\an}\sum_{j=1}^{\nh}w_{ij\ell}\int_{\ahll}^{\ahl}\left(\x\left(\tau\right)\right)^{2}\\
&\hspace*{4cm}\times \left(\sigma^{2}_{\tau}-\sigma^{2}_{i\ah}\right)\,d\tau\bigg|>\frac{\delta\left(\Kmk\right)^{3/4}}{6}\Bigg]\,,
\end{align*}
and
\begin{align*}
A_{n}^{1,3} & =\mathbb{P} \left[ \maxi \left|\frac{\sqrt{\log\left(n\right)}}{\sqrt{\an}}\sum_{\ell=1}^{\an}\sum_{j=1}^{\nh}\w\int_{\ahll}^{\ahl}\x\left(\tau\right)\right. \right. \notag \\
      & \quad  \left.\left. {\vphantom{\maxi\frac{\sqrt{\log\left(n\right)}}{\sqrt{\an}}\sum_{\ell=1}^{\an}}} \times \left(\left(\widetilde{X}_{\tau}-\widetilde{X}_{\ahll}\right)\sigma_{\tau}-\sigma^{2}_{i\ah}\widetilde{W}_{\tau}\right)dW_{\tau}\right|>\frac{\delta\left(\Kmk\right)^{3/4}}{12}\right]\,.
\end{align*}
Starting with $A_{n}^{1,1}$ we employ Markov's inequality, applied to the function $z\mapsto\left|z\right|^{r}$, $r>0$ and $r\in\mathbb{N}$:
\begin{align*}
&\mathbb{P}\hspace*{-0.05cm}\bigg[\bigg|\frac{\sqrt{\log\left(n\right)}}{\sqrt{\an}}\hspace*{-0.05cm}\sum_{\ell=1}^{\an}\hspace*{-0.1cm}\sum_{j=1}^{\nh}\hspace*{-0.2cm} w_{ij\ell}\hspace*{-0.05cm}\int_{\ahll}^{\ahl}\hspace*{-0.1cm}\left(\widetilde{X}_{\tau}\hspace*{-0.05cm}-\hspace*{-0.05cm}\widetilde{X}_{\ahll}\right)\x\hspace*{-0.05cm}\left(\tau\right)a_{\tau}\,d\tau\bigg|>\frac{\delta(\Kmk)^{3/4}}{12}\bigg]\\
&\leq \left(\frac{12}{\delta\left(\Kmk\right)^{3/4}}\right)^{r}\left(\log\left(n\right)\right)^{r/2}\an^{-r/2}\\
&\quad\quad\quad \times\mathbb{E}\left[\left|\sum_{\ell=1}^{\an}\sum_{j=1}^{\nh}w_{ij\ell}\int_{\ahll}^{\ahl}\left(\widetilde{X}_{\tau}-\widetilde{X}_{\ahll}\right)\x\left(\tau\right)a_{\tau}\,d\tau\right|^{r}\right].
\end{align*}
The identity
\begin{align*}
&\mathbb{E}\left[\left|\sum_{\ell=1}^{\an}\sum_{j=1}^{\nh}w_{ij\ell}\int_{\ahll}^{\ahl}\left(\widetilde{X}_{\tau}-\widetilde{X}_{\ahll}\right)\x\left(\tau\right)a_{\tau}\,d\tau\right|^{r}\right]\\
&=\an^{r}\,\mathbb{E}\left[\left|\sum_{\ell=1}^{\an}\frac{1}{\an}\sum_{j=1}^{\nh}w_{ij\ell}\int_{\ahll}^{\ahl}\left(\widetilde{X}_{\tau}-\widetilde{X}_{\ahll}\right)\x\left(\tau\right)a_{\tau}\,d\tau\right|^{r}\right]
\end{align*}
implies, together with Jensen's inequality, that
\begin{align*}
&\an^{r/2}\mathbb{E}\left[\left|\sum_{\ell=1}^{\an}\frac{1}{\an}\sum_{j=1}^{\nh}w_{ij\ell}\int_{\ahll}^{\ahl}\left(\widetilde{X}_{\tau}-\widetilde{X}_{\ahll}\right)\x\left(\tau\right)a_{\tau}\,d\tau\right|^{r}\right]\\
&\le \an^{r/2-1}\sum_{\ell=1}^{\an}\mathbb{E}\left[\left|\sum_{j=1}^{\nh}w_{ij\ell}\int_{\ahll}^{\ahl}\left(\widetilde{X}_{\tau}-\widetilde{X}_{\ahll}\right)\x\left(\tau\right)a_{\tau}\,d\tau\right|^{r}\right]\\
&\le \an^{r/2 -1}\sum_{\ell=1}^{\an}\sum_{j=1}^{\nh}w_{ij\ell}\mathbb{E}\left[\left|\int_{\ahll}^{\ahl}\left(\widetilde{X}_{\tau}-\widetilde{X}_{\ahll}\right)\x\left(\tau\right)a_{\tau}\,d\tau\right|^{r}\right].
\end{align*}
Concerning the second inequality we have taken into account, that $\sum_{j}^{}w_{ij\ell}=1$ in order to apply Jensen's inequality a second time.\\
We employ the generalized Minkowski inequality for double measure integrals, which implies
\begin{align}
&\mathbb{E}\left[\left|\int_{\ahll}^{\ahl}\left(\widetilde{X}_{\tau}-\widetilde{X}_{\ahll}\right)\x\left(\tau\right)a_{\tau}\,d\tau\right|^{r}\right]\nonumber\\
\leq &\left(\int_{\ahll}^{\ahl}\mathbb{E}\left[\left|\left(\widetilde{X}_{\tau}-\widetilde{X}_{\ahll}\right)\x\left(\tau\right)a_{\tau}\right|^{r}\right]^{1/r}\,d\tau\right)^{r}\,.\label{ep_mink}
\end{align}
In order to bound the expectation in \eqref{ep_mink}, we apply Burkolder's inequality to the local martingale part. The general case can be handled via the elementary inequality $\left|a+b\right|^{p}\leq 2^{p}\left(\left|a\right|^{p}+\left|b\right|^{p}\right)$ and the standard bound for Lebesgue integrals
\begin{align}\label{Leb}
\int_{\Omega}^{}f(s)\,d\mu(s)\leq \mu\left(\Omega\right)\sup_{s}\left|f(s)\right|\,,
\end{align}
applied to the finite variation part. 
Taking into account, that the quadratic variation process, $\big([\widetilde{X},\widetilde{X}]_t\big)_{t\in[0,1]}$, is given by
\begin{align*}
[\widetilde{X},\widetilde{X}]_t=\int_{0}^{t}\left(\x\left(s\right)\right)^{2}\sigma^{2}_{s}\,ds.
\end{align*}
yields
\begin{align}
\mathbb{E}\left[\left|\widetilde{X}_{\tau}-\widetilde{X}_{\ahll}\right|^{r}\right]\nonumber\leq &C_{r}\mathbb{E}\left[\left(\int_{\ahll}^{\tau}\left(\x\left(s\right)\right)^{2}\sigma^{2}_{s}\,ds\right)^{r/2}\right]\nonumber\\
\leq &C_{r}\mathbb{E}\left[\left(\int_{\ahll}^{\ahl}\left(\x\left(s\right)\right)^{2}\sigma^{2}_{s}\,ds\right)^{r/2}\right]\nonumber\\
\leq &C_{r}h_{n}^{r/2}h_{n}^{-r/2}=\mathcal{O}(1)\,.\label{st_bound_leb}
\end{align}
\eqref{st_bound_leb} is a consequence of \eqref{Leb},
\begin{align}\label{xibound}
\x\left(x\right)=\mathcal{O}\left(\frac{1}{\sqrt{h_{n}}}\right)\,,
\end{align}
and the global boundedness of $\left(\sigma^{2}_{t}\right)_{t\in\left[0,1\right]}$. Consequently the above yields
\begin{align*}
&\frac{\left(\log\left(n\right)\right)^{r/2}}{\an^{1-r/2}}\sum_{\ell=1}^{\an}\sum_{j=1}^{\nh}\hspace*{-.15cm}w_{ij\ell}\mathbb{E}\left[\left|\int_{\ahll}^{\ahl}\left(\widetilde{X}_{\tau}\hspace*{-.05cm}-\hspace*{-.05cm}\widetilde{X}_{\ahll}\hspace*{-.05cm}\right)\x\left(\tau\right)a_{\tau}\,d\tau\right|^{r}\right]\\
&\leq\left(\log\left(n\right)\right)^{r/2}\an^{r/2-1}\sum_{\ell=1}^{\an}\sum_{j=1}^{\nh}w_{ij\ell}h^{r}_{n}h^{-r/2}_{n}=\mathcal{O}\left(\log\left(n\right)^{r/2}\an^{r/2}h^{r/2}_{n}\right)\,.
\end{align*}
Taking into account that
\begin{align*}
\ah=\mathcal{O}\left(n^{-\frac{1}{4\mathfrak{a}+2}}\left(\log\left(n\right)\right)^{1-\frac{2\mathfrak{a}}{2\mathfrak{a}+1}}\right)\,,
\end{align*}
we can conclude, if $r>2$, that
\begin{align*}
A^{1,1}_n&=\mathcal{O}\left(\left(\ah\right)^{-1}\log\left(n\right)^{r/2}\an^{r/2}h^{r/2}_{n}\right)=\KLEINO\left(1\right)\,,\text{ as }n\rightarrow\infty\,.
\end{align*}
For the term $A^{1,2}_{n}$, we start with
\begin{align}
\frac{\sqrt{\log\left(n\right)}}{\sqrt{\an}}\left|\sum_{\ell=1}^{\an}\sum_{j=1}^{\nh}w_{ij\ell}\int_{\ahll}^{\ahl}\left(\x\left(\tau\right)\right)^{2}\left(\sigma^{2}_{\tau}-\sigma^{2}_{i\ah}\right)\,d\tau\right|\nonumber\\
\leq \frac{\sqrt{\log\left(n\right)}}{\sqrt{\an}}\sum_{\ell=1}^{\an}\sum_{j=1}^{\nh}w_{ij\ell}\int_{\ahll}^{\ahl}\left(\x\left(\tau\right)\right)^{2}\left|\sigma^{2}_{\tau}-\sigma^{2}_{i\ah}\right|\,d\tau\,.\label{ep_ani12}
\end{align}
In \eqref{ep_ani12} the triangle inequality and Jensen's inequality are applied. Combining the regularity under \ref{hypo} and \eqref{eq:an2} gives
\begin{align*}
&\frac{\sqrt{\log\left(n\right)}}{\sqrt{\an}}\sum_{\ell=1}^{\an}\sum_{j=1}^{\nh}w_{ij\ell}\int_{\ahll}^{\ahl}\left(\x\left(\tau\right)\right)^{2}\left|\sigma^{2}_{\tau}-\sigma^{2}_{i\ah}\right|\,d\tau\\
&=\mathcal{O}_{\mathbb{P}}\left(\sqrt{\log\left(n\right)}\sqrt{\an}\left(\an h_{n}\right)^{\mathfrak{a}}\right)\,,\,\text{uniformly in $i$}\\
&=\KLEINO\left(1\right)\,,\text{ as }n\rightarrow\infty.
\end{align*}
Concerning $A^{1,3}_{n}$ we use further decompositions rewriting
\begin{align}
\left(\widetilde{X}_{\tau}-\widetilde{X}_{\ahll}\right)\sigma_{\tau}-\widetilde{W}_{\tau}\sigma^{2}_{i\ah}
\end{align}
in the following way:
\begin{align}
&\left(\widetilde{X}_{\tau}-\widetilde{X}_{\ahll}\right)\sigma_{\tau}-\widetilde{W}_{\tau}\sigma^{2}_{i\ah}\nonumber\\
&=\sigma_{\tau}\int_{\ahll}^{\tau}\x\left(s\right)a_{s}\,ds\nonumber\\
&+\sigma_{\tau}\int_{\ahll}^{\tau}\x\left(s\right)\sigma_{s}\,dW_{s}-\sigma_{i\ah}\int_{\ahll}^{\tau}\x\left(s\right)\sigma_{s}\,dW_{s}\nonumber\\
&+\sigma_{i\ah}\int_{\ahll}^{\tau}\x\left(s\right)\sigma_{s}\,dW_{s}-\sigma_{i\ah}\int_{\ahll}^{\tau}\x\left(s\right)\sigma_{i\ah}\,dW_{s}\nonumber\\
&=\sigma_{\tau}\int_{\ahll}^{\tau}\x\left(s\right)a_{s}\,ds\label{an13a}\\
&+\left(\sigma_{\tau}-\sigma_{i\ah}\right)\int_{\ahll}^{\tau}\x\left(s\right)\sigma_{s}\,dW_{s}\label{an13b}\\
&+\sigma_{i\ah}\int_{\ahll}^{\tau}\x\left(s\right)\left(\sigma_{s}-\sigma_{i\ah}\right)\,dW_{s}\,.\label{an13c}
\end{align}
Using this decomposition, we can bound $A_{n}^{1,3}$ via
\begin{align*}
A_{n}^{1,3}\leq A_{n}^{1,3a}+A_{n}^{1,3b}+A_{n}^{1,3c}\,.
\end{align*}
We start with the probability involving the summand \eqref{an13a}. We have to bound the probability
\begin{align*}
&\mathbb{P} \left[\left|\frac{\sqrt{\log\left(n\right)}}{\sqrt{\an}}\sum_{\ell=1}^{\an}\sum_{j=1}^{\nh}\int_{\ahll}^{\ahl}\w\x\left(\tau\right)\sigma_{\tau}\right. \right. \notag \\
      & \left.\left. \hspace*{4cm}
      \quad\quad\times \int_{\ahll}^{\tau}\x\left(s\right)a_{s}\,ds\,dW_{\tau}\right|>\frac{\delta\left(\Kmk\right)^{3/4}}{36}\right]\\
      &\leq \left(\frac{36}{\delta\left(\Kmk\right)^{3/4}}\right)^{r}\left(\frac{\sqrt{\log\left(n\right)}}{\sqrt{\an}}\right)^{r}\nonumber\\
&\quad\times\mathbb{E}\left[\left|\sum_{\ell=1}^{\an}\sum_{j=1}^{\nh}w_{ij\ell}\int_{\ahll}^{\ahl}\x\left(\tau\right)\sigma_{\tau}\int_{\ahll}^{\tau}\x\left(s\right)a_{s}\,ds\,dW_{\tau}\right|^{r}\right]\,,
\end{align*}
where we have applied Markov's inequality with some exponent $r>0$ and $r\in\mathbb{N}$. Set
\begin{align*}
c_{n,i}\left(\tau\right)=\sum_{\ell=1}^{\an}\sigma_{\tau}\sum_{j=1}^{\nh}w_{ij\ell}\x\left(\tau\right)\int_{\ahll}^{\tau}\x\left(s\right)a_{s}\,ds\\
\times\mathbbm{1}_{\left(\ahll,\ahl\right]}\left(\tau\right)\notag\,.
\end{align*}
In order to apply It\^{o} isometry, we set $r=2m$, with some $m>0$ and $m\in\mathbb{N}$.
We derive that
\begin{align*}
&\mathbb{E}\left[\left|\sum_{\ell=1}^{\an}\sum_{j=1}^{\nh}w_{ij\ell}\int_{\ahll}^{\ahl}\x\left(\tau\right)\sigma_{\tau}\int_{\ahll}^{\tau}\x\left(s\right)a_{s}dsdW_{\tau}\right|^{r}\right]\\
&=\mathbb{E}\left[\left(\int_{i\ah}^{\left(i+1\right)\ah}(c_{n,i}\left(\tau\right))^{2}\,d\tau\right)^{m}\right]\leq \left(\int_{i\ah}^{\left(i+1\right)\ah}\mathbb{E}\left[\left(c_{n,i}\left(\tau\right)\right)^{2m}\right]^{1/m}\,d\tau\right)^{m}\,,
\end{align*}
where we have again used the Minkowski inequality for double measure integrals. Since $\left(i\ah+\left(\ell_{1}-1\right)h_{n},i\ah+\ell_{1}h_{n}\right]$ and $\left(\ah+\left(\ell_{2}-1\right)h_{n},i\ah+\ell_{2}h_{n}\right]$ are disjoint if $\ell_{1}\neq\ell_{2}$ and $\tau$ is fixed, we get
\begin{align*}
\mathbb{E}\left[\left(c_{n,i}\left(\tau\right)\right)^{2m}\right]&= \sum_{\ell=1}^{\an}\mathbb{E}\left[\sigma^{2m}_{\tau}\left(\sum_{j=1}^{\nh}w_{ij\ell}\x\left(\tau\right)\int_{\ahll}^{\tau}\x\left(s\right)a_{s}\,ds\right)^{2m}\right]\\
&\hspace*{6.5cm}   \times\mathbbm{1}_{(\ahll,\ahl]}\left(\tau\right).
\end{align*}
We proceed with Jensen's inequality, which yields
\begin{align*}
 &\sum_{\ell=1}^{\an}\mathbb{E}\left[\sigma^{2m}_{\tau}\left(\sum_{j=1}^{\nh}w_{ij\ell}\x\left(\tau\right)\int_{\ahll}^{\tau}\x\left(s\right)a_{s}\,ds\right)^{2m}\right]\\
&\quad\quad\quad\quad\quad\quad\quad\quad\quad\quad\quad\quad\quad\quad   \times\mathbbm{1}_{\left(\ahll,\ahl\right]}\left(\tau\right)\\
&\leq \sum_{\ell=1}^{\an}\mathbb{E}\left[\sigma^{2m}_{\tau}\sum_{j=1}^{\nh}w_{ij\ell}\left(\x\left(\tau\right)\int_{\ahll}^{\tau}\x\left(s\right)a_{s}\,ds\right)^{2m}\right]\\
&\quad\quad\quad\quad\quad\quad\quad\quad\quad\quad\quad\quad\quad\quad   \times\mathbbm{1}_{\left(\ahll,\ahl\right]}\left(\tau\right)\,.
\end{align*}
Using \eqref{xibound}, global boundedness of the volatility and \eqref{st_bound_leb} we can conclude that
\begin{align*}
\mathbb{E}\left[\left(c_{n,i}\left(\tau\right)\right)^{2m}\right]=\mathcal{O}\left(1\right)\,.
\end{align*} 
Consequently, we can conclude as follows using \eqref{st_bound_leb}:
\begin{align}
&\left(\int_{i\ah}^{\left(i+1\right)\ah}\mathbb{E}\left[(c_{n,i}(\tau))^{2m}_{\tau}\right]^{1/m}\,d\tau\right)^{m}=\mathcal{O}\left(\left(\an  h_{n}\right)^{m}\right)\,.
\end{align}
That yields the following bound for $A_{n}^{1,3a}$
\begin{align*}
A_{n}^{1,3a}&=\mathcal{O}\left(\left(\log\left(n\right)\right)^{m}\hn^{m}\left(\ah\right)^{-1}\right)=\KLEINO\left(1\right)\,,\text{ as }n\rightarrow \infty\,,
\end{align*}
for $m$ sufficiently large.\\
We proceed with the probability $A^{1,3b}_{n}$ involving the term \eqref{an13b}. We first get the standard bound by the Markov inequality with some exponent $r>0$ and $r\in\mathbb{N}$:
\begin{align*}
&\mathbb{P}\left[\frac{\sqrt{\log\left(n\right)}}{\sqrt{\an}}\left|\sum_{\ell=1}^{\an}\sum_{j=1}^{\nh}w_{ij\ell}\int_{\ahll}^{\ahl}\x\left(\tau\right)\left(\sigma_{\tau}-\sigma_{i\ah}\right)\right.\right.\notag\\
&\left.\left.{\vphantom{\frac{\sqrt{\log\left(n\right)}}{\sqrt{\an}}\sum_{\ell=1}^{\an}}}\quad\quad\quad\quad\quad\quad\times\int_{\ahll}^{\tau}\x\left(s\right)\sigma_{s}\,dW_{s}dW_{\tau}\right|>\frac{\delta\left(\Kmk\right)^{3/4}}{36}\right]\nonumber\\
&\leq \left(\frac{36}{\delta\left(\Kmk\right)^{3/4}}\right)^{r}\left(\frac{\sqrt{\log\left(n\right)}}{\sqrt{\an}}\right)^{r}\mathbb{E}\left[\left|\sum_{\ell=1}^{\an}\sum_{j=1}^{\nh}w_{ij\ell}\int_{\ahll}^{\ahl}\x\left(\tau\right)\right.\right.\notag\\
&\quad\quad\quad\quad\quad\quad\quad\quad\quad\quad\quad\quad\quad\quad\times\left(\sigma_{\tau}-\sigma_{i\ah}\right){\vphantom{\sum_{\ell=1}^{\an}}}\left.\left.\int_{\ahll}^{\tau}\x\left(s\right)\sigma_{s}\,dW_{s}dW_{\tau}\right|^{r}\right]\,.
\end{align*}
We define
\begin{align*}
s_{n,i}\left(\tau\right)=\sum_{\ell=1}^{\an}\left(\sigma_{\tau}-\sigma_{i\ah}\right)\sum_{j=1}^{\nh}w_{ij\ell}\x\left(\tau\right)\int_{\ahll}^{\tau}\x\left(s\right)\sigma_{s}\,dW_{s}\\
\quad\quad\quad\quad\times\mathbbm{1}_{\left(\ahll,\ahl\right]}\left(\tau\right).
\end{align*}
In order to apply It\^{o} isometry, we set $r=2m$, with $m>0$ and $m\in\mathbb{N}$. We obtain that
\begin{align*}
\mathbb{E}\left[\left|\int_{i\ah}^{\left(i+1\right)\ah}s_{n,i}\left(\tau\right)\,dW_{\tau}\right|^{2m}\right]&=\mathbb{E}\left[\left(\int_{i\ah}^{\left(i+1\right)\ah}\left(s_{n,i}\left(\tau\right)\right)^{2}\,d\tau\right)^{m}\right]\\
&\leq \left(\int_{i\ah}^{\left(i+1\right)\ah}\mathbb{E}\left[\left(s_{n,i}\left(\tau\right)\right)^{2m}\right]^{1/m}\,d\tau\right)^{m}.
\end{align*}
We have that
\begin{align*}
&\mathbb{E}\left[\left(s_{n,i}\left(\tau\right)\right)^{2m}\right]\\
&=\mathbb{E}\left[\sum_{\ell=1}^{\an}\left(\sigma_{\tau}-\sigma_{i\ah}\right)^{2m}\left(\sum_{j=1}^{\nh}w_{ij\ell}\x(s)\int_{\ahll}^{\tau}\x\left(s\right)\sigma_{s}\,dW_{s}\right)^{2m}\right]\\
&\quad\quad\quad\quad\quad\quad\quad\quad\quad\quad\quad\quad\quad\quad   \times\mathbbm{1}_{\left(\ahll,\ahl\right]}\left(\tau\right)\\
&=\mathbb{E}\left[\sum_{\ell=1}^{\an}\left(\sigma_{\tau}-\sigma_{i\ah}\right)^{2m}\sum_{j=1}^{\nh}w_{ij\ell}\x(s)\left(\int_{\ahll}^{\tau}\x\left(s\right)\sigma_{s}\,dW_{s}\right)^{2m}\right]\\
&\quad\quad\quad\quad\quad\quad\quad\quad\quad\quad\quad\quad\quad\quad   \times\mathbbm{1}_{\left(\ahll,\ahl\right]}\left(\tau\right)\\
&\leq\mathbb{E}\left[\sum_{\ell=1}^{\an}\left(\ah\right)^{2m\mathfrak{a}}\sum_{j=1}^{\nh}w_{ij\ell}\left(\x\left(\tau\right)\int_{\ahll}^{\tau}\x\left(s\right)^{2}\sigma^{2}_{s}\,d{s}\right)^{m}\right]\\
&\quad\quad\quad\quad\quad\quad\quad\quad\quad\quad\quad\quad\quad\quad   \times\mathbbm{1}_{\left(\ahll,\ahl\right]}\left(\tau\right)\\
&=\mathcal{O}\left(\left(\ah \right)^{2m\mathfrak{a}}h^{-m}_{n}\right)\,,
\end{align*}
by the regularity under \ref{hypo}. Overall, we can deduce for $A^{1,3b}_{n}$ that
\begin{align*}
A^{1,3b}_{n}&=\mathcal{O}\left(\left(\ah\right)^{-1}\left(\ah \right)^{2m\mathfrak{a}}\log(n)^{m}\right)=\KLEINO\left(1\right)\,,\text{ as }n\rightarrow\infty\,,
\end{align*}
if $m\in\mathbb{N}$ sufficiently large.
Proceeding with $A^{1,3c}_{n}$, we have with $r>0$ and $r\in\mathbb{N}$:
\begin{align}
&\mathbb{P}\left[\left|\frac{\sqrt{\log\left(n\right)}}{\sqrt{\an}}\sum_{\ell=1}	^{\an}\sum_{j=1}^{\nh}w_{ij\ell}\int_{\ahll}^{\ahl}\x\left(\tau\right)\sigma_{i\ah}\right.\right.\notag\\
&\left.\left.{\vphantom{\frac{\sqrt{\log\left(n\right)}}{\sqrt{\an}}}}\quad\quad\quad\times\int_{\ahll}^{\ahl}\x\left(s\right)\left(\sigma_{s}-\sigma_{i\ah}\right)dW_{s}dW_{\tau}\right|>\frac{\delta\left(\Kmk\right)^{3/4}}{36}\right]\label{pan13c}\\
&\leq \left(\frac{36}{\delta\left(\Kmk\right)^{3/4}}\right)^{r}\left(\frac{\sqrt{\log\left(n\right)}}{\sqrt{\an}}\right)^{r}\mathbb{E}\left[\left|\sum_{\ell=1}^{\an}\sum_{j=1}^{\nh}w_{ij\ell}\int_{\ahll}^{\ahl}\x\left(\tau\right)\sigma_{i\ah}\right.\right.\notag\\
&\left.\left.{\vphantom{\sum_{\ell=1}^{\an}}}\hspace*{5cm}\times\int_{\ahll}^{\ahl}\x\left(s\right)\left(\sigma_{s}-\sigma_{i\ah}\right)dW_{s}dW_{\tau}\right|^{r}\right].\nonumber
\end{align}
Analogously, we set
\begin{align*}
&q_{n,i}(\tau):=\sum_{\ell=1}^{\an}\sigma_{i\ah}\sum_{j=1}^{\nh}w_{ij\ell}\x\left(\tau\right)\int_{\ahll}^{\tau}\x\left(s\right)\left(\sigma_{s}-\sigma_{i\ah}\right)dW_{s}\\
&\quad\quad\quad\quad\quad\quad\quad\quad\quad\quad\quad\quad\quad\quad   \times\mathbbm{1}_{\left(\ahll,\ahl\right]}\left(\tau\right).
\end{align*}
With $r=2m$, $r>0$ and $r\in\mathbb{N}$ we apply It\^{o} isometry and Minkowski inequality.\\
\begin{align*}
\mathbb{E}\left[\left(\int_{i\ah}^{\left(i+1\right)\ah}q_{n,i}(\tau)\,dW_{\tau}\right)^{2m}\right]&=\mathbb{E}\left[\left(\int_{i\ah}^{\left(i+1\right)\ah}\left(q_{n,i}(\tau)\right)^{2}\,d\tau\right)^{m}\right]\\&\leq \left(\int_{i\ah}^{\left(i+1\right)\ah}\mathbb{E}\left[\left(q_{n,i}(\tau)\right)^{2m}\right]^{1/m}d\tau\right)^{m}.
\end{align*}
Since $\left(i\ah+\left(\ell_{1}-1\right)h_{n},i\ah+\ell_{1}h_{n}\right]$ and $\left(i\ah+\left(\ell_{2}-1\right)h_{n},i\ah+\ell_{2}h_{n}\right]$ are disjoint if $\ell_{1}\neq\ell_{2}$ and $\tau$ is fixed, we get
\begin{align*}
&\mathbb{E}\left[\left(q_{n,i}(\tau)\right)^{2m}\right]\\
&=\mathbb{E}\left[\sum_{\ell=1}^{\an}\sigma^{2m}_{i\ah}\left(\sum_{j=1}^{\nh}w_{ij\ell}\x\left(\tau\right)\int_{\ahll}^{\tau}\x\left(s\right)\left(\sigma_{s}-\sigma_{i\ah}\right)dW_{s}\right)^{2m}\right]\\
&\quad\quad\quad\quad\quad\quad\quad\quad\quad\quad\quad\quad\quad\quad   \times\mathbbm{1}_{\left(\ahll,\ahl\right]}\left(\tau\right)\\
&\leq \mathbb{E}\left[\sum_{\ell=1}^{\an}\sigma^{2m}_{i\ah}\sum_{j=1}^{\nh}w_{ij\ell}\left(\x\left(\tau\right)\right)^{2m}\left(\int_{\ahll}^{\tau}\x\left(s\right)\left(\sigma_{s}-\sigma_{i\ah}\right)dW_{s}\right)^{2m}\right]\\
&\quad\quad\quad\quad\quad\quad\quad\quad\quad\quad\quad\quad\quad\quad   \times\mathbbm{1}_{\left(\ahll,\ahl\right]}\left(\tau\right)
\end{align*}
where we applied Jensen's inequality. Proceeding with Burkholder's inequality and \eqref{Leb}, we get
\begin{align*}
&\mathbb{E}\left[\sum_{\ell=1}^{\an}\sigma^{2m}_{i\ah}\sum_{j=1}^{\nh}w_{ij\ell}\left(\x\left(\tau\right)\right)^{2m}\left(\int_{\ahll}^{\tau}\x\left(s\right)\left(\sigma_{s}-\sigma_{i\ah}\right)dW_{s}\right)^{2m}\right]\\
&\quad\quad\quad\quad\quad\quad\quad\quad\quad\quad\quad\quad\quad\quad   \times\mathbbm{1}_{\left(\ahll,\ahl\right]}\left(\tau\right)\\
&\leq K_{m}\mathbb{E}\left[\sum_{\ell=1}^{\an}\sigma^{2m}_{i\ah}\sum_{j=1}^{\nh}w_{ij\ell}\left(\x\left(\tau\right)\right)^{2m}\left(\int_{\ahll}^{\tau}\left(\x\left(s\right)\right)^{2}\left(\sigma_{s}-\sigma_{i\ah}\right)^{2}d{s}\right)^{m}\right]\\
&\quad\quad\quad\quad\quad\quad\quad\quad\quad\quad\quad\quad\quad\quad   \times\mathbbm{1}_{\left(\ahll,\ahl\right]}\left(\tau\right)\\
&=\mathcal{O}\left(\left(\an h_{n}\right)^{2m\mathfrak{a}}h_{n}^{-m}\right)\,,
\end{align*}
which gives the following bound concerning $A^{1,3c}_{n}$:
\begin{align*}
A^{1,3c}_{n}&=\mathcal{O}\left(\left(\ah\right)^{-1}\left(\log\left(n\right)\right)^{m}\left(\ah\right)^{2m\mathfrak{a}}\right)=\KLEINO\left(1\right)\,,\text{ as }n\rightarrow\infty\,,
\end{align*}
if $m$ is sufficiently large.
We have completed the third term $A^{1,3c}_{n}$ and so $A^{1,3}_{n}$. Overall, the term $A^{1}_{n}$ has shown to be negligible. 
We proceed with $A_{n}^{2}$ from \eqref{Adecomp}. Therefore, we take into account that
\begin{align}
&\SX-\sigma_{i\ah}\SW\nonumber\\
&=\int_{\ahll}^{\ahl}\x(s)\,dX_{s}-\sigma_{i\ah}\int_{\ahll}^{\ahl}\x(s)\,dW_{s}\nonumber\\
&=\int_{\ahll}^{\ahl}\x(s)\left(\sigma_{s}-\sigma_{i\ah}\right)dW_{s}+\int_{\ahll}^{\ahl}\x(s)a_{s}\,ds\,.\label{an2}
\end{align}
Using this identity, we bound $A^{2}_{n}$ by
\begin{align*}
A^{2}_{n}\leq A^{2,1}_{n}+A^{2,2}_{n}\,,
\end{align*}
where the probability $A^{2,1}_{n}$ is based on the local martingale part in \eqref{an2} and $A^{2,2}_{n}$ is based on the finite variation part. The elementary inequality $\left|a+b\right|^{p}\leq 2^{p}\left(\left|a\right|^{p}+\left|b\right|^{p}\right)$ allows to split the discussion of $A^{2}_{n}$.
Starting with $A^{2,1}_{n}$ we proceed as follows using Markov's inequality with an exponent $r>0$ and $r\in\mathbb{N}$.
\begin{align*}
&\mathbb{P}\left[\left|\sum_{\ell=1}^{\an}\sum_{j=1}^{\nh}\w\Se\int_{\ahll}^{\ahl}\x\left(s\right)\left(\sigma_{s}-\sigma_{i\ah}\right)dW_{s}\right|>\frac{\delta\left(\Kmk\right)^{3/4}}{8}\right]\\
&\leq \left(\frac{8}{\delta\left(\Kmk\right)^{3/4}}\right)^{r}\mathbb{E}\left[\left|\sum_{\ell=1}^{\an}\sum_{j=1}^{\nh}\w\Se\int_{\ahll}^{\ahl}\x\left(s\right)\left(\sigma_{s}-\sigma_{i\ah}\right)dW_{s}\right|^{r}\right]
\end{align*}
We define
\begin{align*}
v_{n,i}\left(\tau\right)&=\frac{1}{\hn}\sum_{\ell=1}^{\an}\sum_{j=1}^{\nh}\w\Se\int_{\ahll}^{\tau}\x\left(s\right)\left(\sigma_{s}-\sigma_{i\ah}\right)dW_{s}\\
&\quad\quad\quad\quad\quad\quad\quad\quad\quad\quad\quad\quad\quad\quad\times\mathbbm{1}_{\left(\ahll,\ahl\right]}\left(\tau\right)\,,
\end{align*}
such that with $r=2m$, $m>0$ and $m\in\mathbb{N}$
\begin{align*}
&=\mathbb{E}\left[\left|\sum_{\ell=1}^{\an}\sum_{j=1}^{\nh}\w\Se\int_{\ahll}^{\ahl}\x\left(s\right)\left(\sigma_{s}-\sigma_{i\ah}\right)dW_{s}\right|^{r}\right]\\
&=\mathbb{E}\left[\left|\int_{i\ah}^{(i+1)\ah}v_{n,i}\left(\tau\right)d\tau\right|^{2m}\right]\leq \left(\int_{i\ah}^{(i+1)\ah}\mathbb{E}\left[\left(v_{n,i}\left(\tau\right)\right)^{2m}\right]^{1/2m}d\tau\right)^{2m}\,.
\end{align*}
In order to bound this expectation, we split the $j$-sum using the elementary inequality $\left|a+b\right|^{p}\leq 2^{p}\left(\left|a\right|^{p}+\left|b\right|^{p}\right)$ and that the weights fulfill the following growth behaviour:
\begin{equation}
w_{jk} \propto \begin{cases}
        \hphantom{-}1, & \text{for $j \leq \sqrt{n}\hn$}\label{wOrder}\\
        j^{-4}n^{2}\hn^{4}, & \text{for $j > \sqrt{n}\hn$}\,.
      \end{cases}
\end{equation} 
That yields
\begin{align*}
&\mathbb{E}\left[\left(v_{n,i}\left(\tau\right)\right)^{2m}\right]\\
&=\mathbb{E}\left[\frac{1}{\hn^{2m}}\sum_{\ell=1}^{\an}\left(\sum_{j=1}^{\nh}\w\Se\int_{\ahll}^{\ahl}\x\left(s\right)\left(\sigma_{s}-\sigma_{i\ah}\right)dW_{s}\right)^{2m}\right]\\
&\quad\quad\quad\quad\quad\quad\quad\quad\quad\quad\quad\quad\quad\quad\times\mathbbm{1}_{\left(\ahll,\ahl\right]}\left(\tau\right)\\
&\leq\frac{2^{2m}}{\hn^{2m}}\sum_{\ell=1}^{\an}\mathbb{E}\left[\left(\sum_{j=1}^{\sqrt{n}\hn}\Se\int_{\ahll}^{\tau}\x(s)\left(\sigma_{s}-\sigma_{i\ah}\right)dW_{s}\right)^{2m}\right]\\
&\quad\quad\quad\quad\quad\quad\quad\quad\quad\quad\quad\quad\quad\quad\times\mathbbm{1}_{\left(\ahll,\ahl\right]}\left(\tau\right)\\
&+\frac{2^{2m}}{\hn^{2m}}\sum_{\ell=1}^{\an}\mathbb{E}\left[\left(\sum_{j=\sqrt{n}\hn+1}^{\nh}j^{-4}n\hn\Se\int_{\ahll}^{\tau}\x(s)\left(\sigma_{s}-\sigma_{i\ah}\right)dW_{s}\right)^{2m}\right]\\
&\quad\quad\quad\quad\quad\quad\quad\quad\quad\quad\quad\quad\quad\quad\times\mathbbm{1}_{\left(\ahll,\ahl\right]}\left(\tau\right)\,.
\end{align*}
Since 
\begin{align*}
j^{-4}n^{2}\hn^{4}=\mathcal{O}\left(1\right)\quad\text{for}\quad \sqrt{n}\hn\leq j\leq n\hn\,,
\end{align*}
it is sufficient to consider the first summand only.\\
The calculations pursued in Lemma 2 in \cite{BibingerWinkel2018} imply the following, using the fact, that $\left(\varepsilon_{t}\right)_{t\in\left[0,1\right]}$ is independent of $\mathcal{F}^{(0)}$.
\begin{align*}
&\mathbb{E}\left[\frac{1}{\hn^{2m}}\sum_{\ell=1}^{\an}\left(\sum_{j=1}^{\sqrt{n}\hn}\w\Se\int_{\ahll}^{\ahl}\x\left(s\right)\left(\sigma_{s}-\sigma_{i\ah}\right)dW_{s}\right)^{2m}\right]\\
&\leq C_{m}\mathbb{E}\left[\left(\int_{\ahll}^{\ahl}\left(\x(s)\right)^{2}\left(\sigma_{s}-\sigma_{i\ah}\right)^{2}ds\right)^{m}\right]\\
&=\mathcal{O}\left(\left(\ah\right)^{2m\mathfrak{a}}\hn^{-2m}\right)\,,
\end{align*}
such that
\begin{align*}
\mathbb{E}\left[\left(v_{n,i}\left(\tau\right)\right)^{2}\right]=\mathcal{O}\left(\left(\ah\right)^{2m\mathfrak{a}}\hn^{-2m}\right)\,.
\end{align*}
We can conclude that
\begin{align*}
\mathbb{E}\left[\left|\int_{i\ah}^{(i+1)\ah}v_{n,i}\left(\tau\right)d\tau\right|^{2m}\right]=\mathcal{O}\left(\left(\ah\right)^{2m\mathfrak{a}}\an^{2m}\right)\,.
\end{align*}
Overall, we get
\begin{align*}
A_{n}^{2,1}&=\mathcal{O}\left(\left(\ah\right)^{2m\mathfrak{a}}\left(\log\left(n\right)\right)^{m}\an^{m}\left(\ah\right)^{-1}\right)=\KLEINO(1)\,,\text{ as }n\rightarrow\infty\,,
\end{align*}
if $m\in\mathbb{N}$ is sufficiently large.\\
The term $A_{n}^{2,2,}$ can be handled easier, using \eqref{Leb} instead of Burkholder's inequality. Overall, it is shown that $A_{n}^{2}=\KLEINO(1)$.\\
We can proceed with $B_{n}$ from \eqref{eq:2}. Note that
\begin{align*}
\mathbb{P}\left[\mini\absZii< \Kmk\right]&\leq\mathbb{P}\left[\mini\Zii< \Kmk\right]\\
&\leq \sum_{i=0}^{{\lfloor \left(\ah\right)^{-1}\rfloor}-2}\mathbb{P}\left[\Zii< \Kmk\right]\,.
\end{align*}
It is sufficient to bound the probability
\begin{align*}
&\mathbb{P}\left[\overline{Z}_{n,i+1}< \Kmk\right]\\
&=\mathbb{P}\left[\frac{1}{\an}\sum_{\ell=1}^{\an}\sum_{j=1}^{\nh}\w\left(\left(\sigma_{i\ah}\SW+\Se\right)^{2}-\mu_{ij\ell}\right)<\Kmk\right]\\
&=\mathbb{P}\left[\frac{1}{\an}\sum_{\ell=1}^{\an}\sum_{j=1}^{\nh}\w\left(\left(\sigma_{i\ah}\SW+\Se\right)^{2}-\mu_{ij\ell}\right)-\sigma^{2}_{i\ah}<\Km\hspace{-0.6em}-\sigma^{2}_{i\ah}\hspace{-0.6em}-\km\right].
\end{align*}
Note that $\Km-\sigma^{2}_{i\ah}<0$ and $\km>0$, such that we can proceed with Markov's inequality with an exponent $r>0$ and the elementary inequality $\left|a+b+c\right|^{r}\leq 3^{r}\left(\left|a\right|^{r}+\left|b\right|^{r}+\left|c\right|^{r}\right)$:
\begin{align*}
&\mathbb{P}\left[\frac{1}{\an}\sum_{\ell=1}^{\an}\sum_{j=1}^{\nh}\w\left(\left(\sigma_{i\ah}\SW+\Se\right)^{2}-\mu_{ij\ell}\right)-\sigma^{2}_{i\ah}<\Km\hspace{-0.6em}-\sigma^{2}_{i\ah}\hspace{-0.6em}-\km\right]\\
&\leq \mathbb{P}\left[\frac{1}{\an}\left|\sum_{\ell=1}^{\an}\sum_{j=1}^{\nh}\w\left(\left(\sigma_{i\ah}\SW+\Se\right)^{2}-\mu_{ij\ell}\right)-\sigma^{2}_{i\ah}\right|>\km\right]\\
&\leq\left(\frac{3^r}{\km}\right)^{r} \mathbb{E}\left[\left|\frac{1}{\an}\sum_{\ell=1}^{\an}\sum_{j=1}^{\nh}\w\sigma^{2}_{i\ah}\left(\SWsq-1\right)\right|^{r}\right]\\
&+\left(\frac{2\cdot3^r}{\km}\right)^{r}\mathbb{E}\left[\left|\frac{1}{\an}\sum_{\ell=1}^{\an}\sum_{j=1}^{\nh}\w\sigma_{i\ah}\SW\Se\right|^{r}\right]\\
&+\left(\frac{3^{r}}{\km}\right)^{r} \mathbb{E}\left[\left|\frac{1}{\an}\sum_{\ell=1}^{\an}\sum_{j=1}^{\nh}\w\left(\Sesq-\mu_{ij\ell}\right)^{2}\right|^{r}\right]\\
&\leq\left(\frac{3^{r}\Kp}{\km}\right)^{r} \mathbb{E}\left[\frac{1}{\an^{r/2}}\left|\frac{1}{\sqrt{\an}}\sum_{\ell=1}^{\an}\sum_{j=1}^{\nh}\w\left(\SWsq-1\right)\right|^{r}\right]\\
&+\left(\frac{2\Kp3^{r}}{\km}\right)^{r}\mathbb{E}\left[\frac{1}{\an^{r/2}}\left|\frac{1}{\sqrt{\an}}\sum_{\ell=1}^{\an}\sum_{j=1}^{\nh}\w\SW\Se\right|^{r}\right]\\
&+\left(\frac{3^{r}}{\km}\right)^{r} \mathbb{E}\left[\frac{1}{\an^{r/2}}\left|\frac{1}{\sqrt{\an}}\sum_{\ell=1}^{\an}\sum_{j=1}^{\nh}\w\left(\Sesq-\mu_{ij\ell}\right)^{2}\right|^{r}\right]\\
&=\mathcal{O}\left(\an^{-r/2}\right)\,,
\end{align*}
by the classical central limit theorem. This implies
\begin{align*}
\mathbb{P}\left[\mini\absZii< \Kmk\right]&=\mathcal{O}\left(\an^{-r/2}\left(\ah\right)^{-1}\right)=\KLEINO(1)\,,\text{ as }n\rightarrow\infty\,,
\end{align*}
if $r>0$ sufficiently large. Thus, we have completed the term $B_{n}$, and so the term (\textbf{II}).\\
We proceed with (\textbf{I}) from \eqref{eq:1}. It holds that
\begin{align*}
\overline{RV}_{n,i}\left(\frac{1}{\absRii^{3/4}}-\frac{1}{\absZii^{3/4}}\right)=\overline{RV}_{n,i}\left(\frac{\absZii^{3/4}-\absRii^{3/4}}{\absRii^{3/4}\absZii^{3/4}}\right)\,,
\end{align*}
such that for every $\delta>0$ we have
\begin{align}
&\mathbb{P}\left[\maxi\sqrt{\an\log\left(n\right)}\left|\overline{RV}_{n,i}\left(\frac{1}{\absRii^{3/4}}-\frac{1}{\absZii^{3/4}}\right)\right|>\delta\right]\nonumber\\
&=\mathbb{P}\left[\maxi\sqrt{\an\log\left(n\right)}\left|\overline{RV}_{n,i}\left(\frac{\absZii^{3/4}-\absRii^{3/4}}{\absRii^{3/4}\absZii^{3/4}}\right)\right|>\delta,\right.\nonumber\\
&\quad\quad\quad\quad\quad\quad\quad\quad\quad\quad\quad\quad\quad\quad\quad\left.\mini\absRii\absZii\geq\frac{\left(\Kmk\right)^{2}}{2}\right]\nonumber\\
&+\mathbb{P}\left[\maxi\sqrt{\an\log\left(n\right)}\left|\overline{RV}_{n,i}\left(\frac{\absZii^{3/4}-\absRii^{3/4}}{\absRii^{3/4}\absZii^{3/4}}\right)\right|>\delta,\right.\nonumber\\
&\quad\quad\quad\quad\quad\quad\quad\quad\quad\quad\quad\quad\quad\quad\quad\left.\mini\absRii\absZii<\frac{\left(\Kmk\right)^{2}}{2}\right]\nonumber\\
&\leq \mathbb{P}\left[\maxi\sqrt{\an\log\left(n\right)}\left|\overline{RV}_{n,i}\left(\absZii^{3/4}-\absRii^{3/4}\right)\right|>4\delta\left(\Kmk\right)^{3/2}\right]\label{I1}\\
&\quad\quad\quad+\mathbb{P}\left[\mini\absRii\absZii<\frac{\left(\Kmk\right)^{2}}{4}\right]\,.\label{I2}
\end{align}
We start with the second probability \eqref{I2}. 
\begin{align*}
&\mathbb{P}\left[\mini\absRii\absZii<\frac{\left(\Kmk\right)^{2}}{4}\right]\\
&\quad \leq \mathbb{P}\left[\mini\absRii<\frac{\Kmk}{2}\right]+\mathbb{P}\left[\mini\absZii<\frac{\Kmk}{2}\right]
\end{align*}
The second probability has already been considered, since
\begin{align*}
\mathbb{P}\left[\mini\absZii<\frac{\Kmk}{2}\right]&=\mathbb{P}\left[\mini\Zii<\frac{\Kmk}{2}\right]=\mathcal{O}\left(B_{n}\right)\,.
\end{align*} 
Concerning the first one, it holds that
\begin{align*}
&\mathbb{P}\left[\mini\absRii<\frac{\Kmk}{2}\right]\leq \mathbb{P}\left[\mini\Rii<\frac{\Kmk}{2}\right]\\
&=\mathbb{P}\left[\mini\Rii<\frac{\Kmk}{2},\maxi\left|\Rii-\Zii\right|\leq \frac{\Kmk}{2}\right]\\
&+\mathbb{P}\left[\mini\Rii<\frac{\Kmk}{2},\maxi\left|\Rii-\Zii\right|> \frac{\Kmk}{2}\right]\\
&\leq \mathbb{P}\left[\mini\Zii<\Kmk\right]+\mathbb{P}\left[\maxi\left|\Rii-\Zii\right|> \frac{\Kmk}{2}\right]\,.
\end{align*}
Since 
\begin{align*}
\mathbb{P}\left[\maxi\left|\Rii-\Zii\right|> \frac{\Kmk}{2}\right]=\mathcal{O}\left(A_{n}\right)\,,
\end{align*}
we can proceed with \eqref{I1}. For every $\delta>0$ and $\kp\in\left(\Kp,\infty\right)$, it holds that 
\begin{align}
&\mathbb{P}\left[\maxi\sqrt{\an\log\left(n\right)}\left|\overline{RV}_{n,i}\left(\absZii^{3/4}-\absRii^{3/4}\right)\right|>4\delta\left(\Kmk\right)^{3/2}\right]\nonumber\\
&=\mathbb{P}\left[\maxi\sqrt{\an\log\left(n\right)}\left|\overline{RV}_{n,i}\left(\absZii^{3/4}-\absRii^{3/4}\right)\right|>4\delta\left(\Kmk\right)^{3/2},\right.\label{I11}\\
&\hspace*{8.5cm}\left.\maxi \overline{Z}_{n,i}\leq\Kp+\kp\right]\nonumber\\
&+\mathbb{P}\left[\maxi\sqrt{\an\log\left(n\right)}\left|\overline{RV}_{n,i}\left(\absZii^{3/4}-\absRii^{3/4}\right)\right|>4\delta\left(\Kmk\right)^{3/2},\right.\label{I12}\\
&\hspace*{8.5cm}\left.\maxi \overline{Z}_{n,i}>\Kp+\kp\right]\nonumber
\end{align}
We start with \eqref{I11}. 
\begin{align}
&\mathbb{P}\left[\maxi\sqrt{\an\log\left(n\right)}\left|\overline{RV}_{n,i}\left(\absZii^{3/4}-\absRii^{3/4}\right)\right|>4\delta\left(\Kmk\right)^{3/2},\right.\nonumber\\
&\quad\quad\quad\quad\quad\quad\quad\quad\quad\quad\quad\quad\quad\quad\quad\left.\maxi \overline{Z}_{n,i}\leq\Kp\hspace{-0.4em}+\kp\right]\nonumber\\
&\leq \mathbb{P}\left[\maxi\left|\overline{RV}_{n,i}\right|>2\left(\Kp\hspace{-0.4em}+\kp\right)\right]\nonumber\\
&+\mathbb{P}\left[\maxi\sqrt{\an\log\left(n\right)}\left|\left(\absZii^{3/4}-\absRii^{3/4}\right)\right|>\frac{2\delta\left(\Kmk\right)^{3/2}}{\Kp\hspace{-0.4em}+\kp}\right]\nonumber\\
&\leq \mathbb{P}\left[\maxi\left|\overline{RV}_{n,i}-\overline{Z}_{n,i}\right|+\left|\overline{Z}_{n,i}\right|>2\left(\Kp+\kp\right)\right]\nonumber\\
&+\mathbb{P}\left[\maxi\sqrt{\an\log\left(n\right)}\left|\left(\absZii^{3/4}-\absRii)^{3/4}\right)\right|>\frac{2\delta\left(\Kmk\right)^{3/2}}{\Kp\hspace{-0.4em}+\kp}\right]\nonumber\\
&\leq \mathbb{P}\left[\maxi\left|\overline{RV}_{n,i}-\overline{Z}_{n,i}\right|>\Kp\hspace{-0.4em}+\kp\right]+\mathbb{P}\left[\maxi\left|\overline{Z}_{n,i}\right|>\Kp\hspace{-0.4em}+\kp\right]\label{I111}\\
&+\mathbb{P}\left[\maxi\sqrt{\an\log\left(n\right)}\left|\left(\absZii^{3/4}-\absRii^{3/4}\right)\right|>\frac{2\delta\left(\Kmk\right)^{3/2}}{\Kp+\kp}\right]\label{I121}
\end{align}
Note that
\begin{align*}
\mathbb{P}\left[\maxi\left|\overline{Z}_{n,i}\right|>\Kp+\kp\right]\leq\sum_{i=0}^{{\lfloor \left(\ah\right)^{-1}\rfloor}-2}\mathbb{P}\left[\left|\overline{Z}_{n,i}\right|>\Kp+\kp\right]
\end{align*}
holds. We proceed with the triangle inequality and using that $\Kp-\sigma^{2}_{i\ah}>0$ uniformly in $i$,
\begin{align*}
&\mathbb{P}\left[\maxi\left|\overline{Z}_{n,i}\right|>\Kp+\kp\right]\leq\mathbb{P}\left[\left|\frac{1}{\an}\sum_{\ell=1}^{\an}\sum_{j=1}^{\nh}\w\sigma^{2}_{i\ah}\big(\SW-1\big)\right|>\kp\right]\\
&\quad\quad\quad\quad\quad\quad\quad\quad\quad\quad\quad\quad+\mathbb{P}\left[\left|\frac{1}{\an}\sum_{\ell=1}^{\an}\sum_{j=1}^{\nh}\w\sigma_{i\ah}\SW\Se\right|>\kp\right]\\
&\quad\quad\quad\quad\quad\quad\quad\quad\quad\quad\quad\quad+\mathbb{P}\left[\left|\frac{1}{\an}\sum_{\ell=1}^{\an}\sum_{j=1}^{\nh}\w\big(\Se-\mu_{ij\ell}\big)\right|>\kp\right]\,.
\end{align*}
Applying the Markov inequality, bounding the volatility from above, and concluding with a classical central limit theorem argument, yields the bound 
\begin{align*}
\mathbb{P}\left[\left|\overline{Z}_{n,i}\right|>\Kp+\kp\right]=\mathcal{O}\left(\an^{-r/2}\right)\,,
\end{align*}
such that
\begin{align*}
\mathbb{P}\left[\maxi\left|\overline{Z}_{n,i}\right|>\Kp+\kp\right]&=\mathcal{O}\left(\an^{-r/2}\left(\ah\right)^{-1}\right)=\KLEINO(1)\,,\text{ as }n\rightarrow\infty\,,
\end{align*}
holds if the exponent $r>0$ is sufficiently large. This completes \eqref{I111}, since the first probability therein is included in $A_{n}$. \\
We proceed with \eqref{I121}. The discussion of this term can be traced back to $A_{n}$ with a Taylor expansion. More precisely, 
we set $\psi\left(x\right)=x^{3/4}$ and expand around the point $a=\absZii$,
\begin{align*}
\psi\left(\absRii\right)-\psi\left(\absZii\right)&=\psi'\left(\absZii\right)\left(\absRii-\absZii\right)\\
&+\left(\absRii-\absZii\right)\mathcal{R}\left(\absRii-\absZii\right)\,.
\end{align*}
Since $\begin{aligned}
\psi'\left(\absZii\right)=\mathcal{O}_{\mathbb{P}}\left(1\right)
\end{aligned}$, and since the remainder $\mathcal{R}$ is negligible, 
\begin{align*}
\mathcal{R}\left(\absRii-\absZii\right)=\KLEINO_{\mathbb{P}}(1)\,,
\end{align*}
by the reverse triangle inequality and the estimates for $A_{n}$.
Therefore, only 
\begin{align*}
\absRii-\absZii
\end{align*}
is crucial. But, using the reverse triangle inequality again this has already been worked out in $A_{n}$, too. So we have completed \eqref{I121} and so \eqref{I11}. We proceed with \eqref{I12}.
It holds that
\begin{align*}
&\mathbb{P}\left[\maxi\sqrt{\an\log\left(n\right)}\left|\overline{RV}_{n,i}\left(\absZii^{3/4}-\absRii^{3/4}\right)\right|>4\delta\left(\Kmk\right)^{3/2},\right.\nonumber\\
&\hspace*{8.5cm}\left.\maxi \overline{Z}_{n,i}>\Kp+\kp\right]\\
&\leq\mathbb{P}\left[\maxi \overline{Z}_{n,i}>\Kp+\kp\right]\leq \mathbb{P}\left[\maxi \left|\overline{Z}_{n,i}\right|>\Kp+\kp\right]\,.
\end{align*}
Thus, this probability has already been considered within \eqref{I111}. Therefore, we also have completed \eqref{I12}, such that we are done with $(\textbf{I})$
The terms $(\textbf{III})$ and $\left(\textbf{IV}\right)$ in \eqref{eq:1} are only shifted in $i$. So we have finished the proof of Proposition \ref{prop1}.
\qed
\\
For the \emph{second step} described in Section \ref{subsec:4.2} we approximate the volatility locally constant over two consecutive blocks by shifting the index of $\sigma_{(i+1)\alpha_nh_n}$ in $\overline{Z}_{n,i+1}$ as follows: $i+1\mapsto i$. We set
\begin{align*}
\widetilde{Z}_{n,i}:=\frac{1}{\an}\sum_{\ell=1}^{\an}\sum_{j=1}^{\nh}w_{ij\ell}\left(\left(\sigma_{\left(i-1\right)\ah}S_{ij\ell}\left(W\right)+S_{ij\ell}\left(\varepsilon\right)\right)^{2}-\mu_{ij\ell}\right)\,.
\end{align*}
\begin{prop}\label{prop2}
Given the assumptions of Theorem \ref{thm:1}, it holds under \ref{hypo} that
\begin{align*}
\sqrt{\an\log\left(h_n^{-1}\right)}\maxi\left|\bigg|\frac{\overline{Z}_{n,i}-\overline{Z}_{n,i+1}}{\absZii^{3/4}}\bigg|-\bigg|\frac{\Zi-\Ziiw}{\absZiiw^{3/4}}\bigg|\right|\overset{\mathbb{P}}{\longrightarrow}0\,.
\end{align*}
\end{prop}
\noindent
\textit{Proof of Proposition \ref{prop2}}.\\
The decomposition 
\begin{align*}
&\frac{\overline{Z}_{n,i}-\overline{Z}_{n,i+1}}{\absZii^{3/4}}-\frac{\overline{Z}_{n,i}-\widetilde{Z}_{n,i+1}}{\absZiiw^{3/4}}\\
&=\frac{\overline{Z}_{n,i}}{\absZii^{3/4}}-\frac{\overline{Z}_{n,i}}{\absZiiw^{3/4}}+\frac{\overline{Z}_{n,i+1}}{\absZiiw^{3/4}}-\frac{\overline{Z}_{n,i+1}}{\absZii^{3/4}}+\frac{\widetilde{Z}_{n,i+1}-\overline{Z}_{n,i+1}}{\absZii^{3/4}}
\end{align*}
yields, via the triangle inequality, the three terms 
\begin{align*}
&\maxi\Bigg|\overline{Z}_{n,i}\bigg(\frac{1}{\absZii^{3/4}}-\frac{1}{\absZiiw^{3/4}}\bigg)\Bigg|+\maxi\Bigg|\overline{Z}_{n,i+1}\bigg(\frac{1}{\absZiiw^{3/4}}-\frac{1}{\absZii^{3/4}}\bigg)\Bigg|\\
&\quad\quad\quad\quad+\maxi\left|\frac{\widetilde{Z}_{n,i+1}-\overline{Z}_{n,i+1}}{\absZii^{3/4}}\right|\\
&=:\text{(\textbf{I})}+\text{(\textbf{II})}+\text{(\textbf{III})}\,.
\end{align*}
We start with (\textbf{III}). For any $\delta>0$ it holds that 
\begin{align}
&\mathbb{P}\left[\maxi\left|\frac{\sqrt{n\log\left(n\right)}\left(\widetilde{Z}_{n,i+1}-\overline{Z}_{n,i+1}\right)}{\absZiiw^{3/4}}\right|>\delta\right]\nonumber\\
&=\mathbb{P}\left[\maxi\left|\frac{\sqrt{n\log\left(n\right)}\left(\widetilde{Z}_{n,i+1}-\overline{Z}_{n,i+1}\right)}{\absZiiw^{3/4}}\right|>\delta,\mini\absZiiw\geq\Kmk\right]\nonumber\\
&+\mathbb{P}\left[\maxi\left|\frac{\sqrt{n\log\left(n\right)}\left(\widetilde{Z}_{n,i+1}-\overline{Z}_{n,i+1}\right)}{\absZiiw^{3/4}}\right|>\delta,\mini\absZiiw<\Kmk\right]\nonumber\\
&\leq \mathbb{P}\left[\maxi\sqrt{n\log\left(n\right)}\left|\widetilde{Z}_{n,i+1}-\overline{Z}_{n,i+1}\right|>\delta\big(\Kmk\big)^{3/4}\right]\label{prop2III1}\\
&+\mathbb{P}\left[\mini\absZiiw<\Kmk\right]\label{prop2III2}\,.
\end{align}
The probability \eqref{prop2III2} has already been done, since it only differs by a shift in $i$ with respect to the volatility from the term in Proposition \ref{prop1}. We continue with \eqref{prop2III1}.
It holds that
\begin{align*}
\widetilde{Z}_{n,i}-\overline{Z}_{n,i}&=\left(\sigma^{2}_{(i-1)\ah}-\sigma^{2}_{i\ah}\right)\frac{1}{\an}\sum_{\ell=1}^{\an}\sum_{j=1}^{\nh}\w\SWsq\\
&+\left(\sigma_{(i-1)\ah}-\sigma_{i\ah}\right)\frac{1}{\an}\sum_{\ell=1}^{\an}\sum_{j=1}^{\nh}\w\SW\Se\,.
\end{align*}
That yields
\begin{align*}
&\mathbb{P}\left[\maxi\sqrt{\an\log\left(n\right)}\left|\widetilde{Z}_{n,i+1}-\overline{Z}_{n,i+1}\right|>\delta\big(\Kmk\big)^{3/4}\right]\\
&\leq \mathbb{P}\left[\maxi\sqrt{\an\log\left(n\right)}\left(\sigma^{2}_{i\ah}-\sigma^{2}_{(i+1)\ah}\right)\right.\\
&\quad\quad\quad\quad\quad\quad\quad\quad\times\left.\frac{1}{\an}\left|\sum_{\ell=1}^{\an}\sum_{j=1}^{\nh}w_{(i+1)j\ell}S^{2}_{(i+1)j\ell}\left(W\right)\right|>\frac{\big(\Kmk\big)^{3/4}}{2}\right]\\
&+\mathbb{P}\left[\maxi\sqrt{\an\log\left(n\right)}\left(\sigma_{i\ah}-\sigma_{(i+1)\ah}\right)\right.\\
&\quad\quad\quad\quad\quad\quad\quad\quad\times\left.\frac{1}{\an}\left|\sum_{\ell=1}^{\an}\sum_{j=1}^{\nh}w_{(i+1)j\ell}S_{(i+1)j\ell}\left(W\right)S_{(i+1)j\ell}\left(\varepsilon\right)\right|>\frac{\big(\Kmk\big)^{3/4}}{4}\right]\\
&\leq \mathbb{P}\left[\maxi\sqrt{\an\log\left(n\right)}\left|\sigma^{2}_{i\ah}-\sigma^{2}_{(i+1)\ah}\right|>\frac{\big(\Kmk\big)^{3/4}}{4}\right]\\
&+\mathbb{P}\left[\maxi\frac{1}{\an}\sum_{\ell=1}^{\an}\sum_{j=1}^{\nh}w_{(i+1)j\ell}S^{2}_{(i+1)j\ell}\left(W\right)>2\right]\,.
\end{align*}
Concerning the first term it holds that
\begin{align*}
\maxi\sqrt{\an\log\left(n\right)}\left|\sigma^{2}_{i\ah}-\sigma^{2}_{(i+1)\ah}\right|&=\mathcal{O}_{\mathbb{P}}\left(\sqrt{\an\log\left(n\right)}\left(\ah\right)^{\mathfrak{a}}\right)\,,\text{ uniformly in }i\\
&=\KLEINO_{\P}(1)\,,\text{ as }n\rightarrow\infty\,,\text{ by }\eqref{eq:an2}.
\end{align*}
It remains to show that 
\begin{align*}
\mathbb{P}\left[\maxi\frac{1}{\an}\sum_{\ell=1}^{\an}\sum_{j=1}^{\nh}w_{(i+1)j\ell}S^{2}_{(i+1)j\ell}\left(W\right)>2\right]=\KLEINO(1)\,.
\end{align*}
We conclude with a classical central limit theorem argument, using Markov's inequality with $r>0$.
\begin{align*}
&\mathbb{P}\left[\maxi\frac{1}{\an}\sum_{\ell=1}^{\an}\sum_{j=1}^{\nh}w_{(i+1)j\ell}S^{2}_{(i+1)j\ell}\left(W\right)>2\right]\\
&=\mathbb{P}\left[\maxi\frac{1}{\an}\sum_{\ell=1}^{\an}\sum_{j=1}^{\nh}w_{(i+1)j\ell}\big(S^{2}_{(i+1)j\ell}\left(W\right)-1\big)>1\right]\\
&\leq\left(\ah\right)^{-1}\mathbb{E}\left[\frac{1}{\an^{r/2}}\left|\frac{1}{\sqrt{\an}}\sum_{j=1}^{\nh}w_{(i+1)j\ell}\big(S^{2}_{(i+1)j\ell}\left(W\right)-1\big)\right|^{r}\right]\\
&=\mathcal{O}\left(\left(\ah\right)^{-1}\an^{-r/2}\right)=\KLEINO(1)\,,\text{ as }n\rightarrow\infty\,,
\end{align*}
with $r>0$ sufficiently large. We have completed \eqref{prop2III1} and so (\textbf{III}).\\
We proceed with (\textbf{I}). For any $\delta>0$ it holds that
\begin{align}
&\mathbb{P}\left[\maxi\left|\overline{Z}_{n,i}\left(\frac{1}{\absZii^{3/4}}-\frac{1}{\absZiiw^{3/4}}\right)\right|>\delta\right]\nonumber\\
&=\mathbb{P}\left[\maxi\left|\overline{Z}_{n,i}\left(\frac{\absZiiw^{3/4}-\absZii^{3/4}}{\absZii^{3/4}\absZiiw^{3/4}}\right)\right|>\delta\right]\nonumber\\
&=\mathbb{P}\left[\maxi\left|\overline{Z}_{n,i}\left(\frac{\absZiiw^{3/4}-\absZii^{3/4}}{\absZii^{3/4}\absZiiw^{3/4}}\right)\right|>\delta,\mini\absZii\absZiiw\geq\left(\Kmk\right)^{2}\right]\nonumber\\
&+\mathbb{P}\left[\maxi\left|\overline{Z}_{n,i}\left(\frac{\absZiiw^{3/4}-\absZii^{3/4}}{\absZii^{3/4}\absZiiw^{3/4}}\right)\right|>\delta,\mini\absZii\absZiiw<\left(\Kmk\right)^{2}\right]\nonumber\\
&\leq \mathbb{P}\left[\maxi\left|\overline{Z}_{n,i}\left(\absZiiw^{3/4}-\absZii^{3/4}\right)\right|>\delta\left(\Kmk\right)^{3/2}\right]\label{prop2I1}\\
&+\mathbb{P}\left[\mini\absZii\absZiiw<\left(\Kmk\right)^{2}\right]\label{prop2I2}\,.
\end{align}
We start with \eqref{prop2I2}.
\begin{align*}
&\mathbb{P}\left[\mini\absZii\absZiiw<\left(\Kmk\right)^{2}\right]\\
&\leq \mathbb{P}\left[\mini\absZii<\Kmk\right]+\mathbb{P}\left[\mini\absZiiw<\Kmk\right]\,.
\end{align*}
Only the second probability has to be considered. But, since the involved statistic only differs by a shift in the volatility, we can bound the latter from below and argue with the central limit theorem. So we have completed \eqref{prop2I2} and continue with \eqref{prop2I1}.\\
We handle \eqref{prop2I1} via a Taylor expansion. So, expanding the function $\psi\left(x\right)=x^{3/4}$ around the point $a=|\widetilde{Z}_{n,i+1}|$ yields the desired result using the procedure for (\textbf{III}). We will omit the details for (\textbf{II}), since it only differs by a shift in $i$. So Proposition \ref{prop2} is proven.
\qed
\\
We do a further approximation step, replacing the denominator in Proposition \ref{prop2} by its limit. This is the \emph{third step} outlined in Section \ref{subsec:4.2}. Here, we use the estimator $\hat{\eta}^{2}$ from \eqref{etahat}.
\begin{prop}\label{prop3n}
Given the assumptions of Theorem \ref{thm:1}, it holds under \ref{hypo} that
\begin{align*}
\sqrt{\an\log\left(h_n^{-1}\right)}\maxi\left|\left|\frac{\Zi-\Ziiw}{\sqrt{8\hat{\eta}}\absZiiw^{3/4}}\right|-\left|\frac{\Zi-\Ziiw}{\sqrt{8\eta}\sigma^{3/2}_{i\ah}}\right|\right|\overset{\mathbb{P}}{\longrightarrow}0\,.
\end{align*}
\end{prop}
\noindent
\textit{Proof of Proposition \ref{prop3n}.}\\
We have to bound the term
\begin{align*}
&\maxi\sqrt{\an\log\left(n\right)}\left|\left(\overline{Z}_{n,i}-\widetilde{Z}_{n,i+1}\right)\left(\frac{1}{\sqrt{8\hat{\eta}}\absZiiw^{3/4}}-\frac{1}{\sqrt{8\eta}\sigma^{3/2}_{i\an h_{h}}}\right)\right|\\
&=\maxi\sqrt{\an\log(n)}\left|\Zi-\Ziiw\right|\maxi\left|\frac{1}{\sqrt{8\hat{\eta}}\absZiiw^{3/4}}-\frac{1}{\sqrt{8\eta}\sigma^{3/2}_{i\an h_{h}}}\right|\,.
\end{align*}
We will employ a 2-dimensional Taylor expansion of order 1 with respect to the second term. We set
$\psi(x,y)=x^{-1/2}y^{-3/4}$ and expand around the point $(a,b)=(\eta,\sigma^{2}_{i\ah})$.
Therefore, we have to bound the term
\begin{align}\label{taylor}
\frac{\partial{\psi(\eta,\sigma^{2}_{i\ah})}}{\partial{x}}\left(\hat{\eta}-\eta\right)+\frac{\partial{\psi(\eta,\sigma^{2}_{i\ah})}}{\partial{y}}\Big(\absZiiw-\sigma^{2}_{i\ah}\Big)+\KLEINO_{\mathbb{P}}(1).
\end{align}
Since $(\sigma^{2}_{t})_{t\in[0,1]}$ can be bounded globally, we get the following uniform bounds in $i$:
\begin{align*}
\max\bigg(\frac{\partial{\psi(\eta,\sigma^{2}_{i\ah})}}{\partial{x}}\,,\frac{\partial{\psi(\eta,\sigma^{2}_{i\ah})}}{\partial{y}}\bigg)=\mathcal{O}_{\mathbb{P}}(1)\,.
\end{align*}
The first summand in \eqref{taylor} with $\left(\hat{\eta}-\eta\right)$ can be handled easily, using 
\begin{align*}
\left(\hat{\eta}-\eta\right)=\mathcal{O}_{\mathbb{P}}(n^{-1/2})\,.
\end{align*}
This implies
\begin{align*}
\maxi\frac{\partial{\psi(\eta,\sigma^{2}_{i\ah})}}{\partial{x}}\left(\hat{\eta}-\eta\right)=\KLEINO_{\mathbb{P}}(1).
\end{align*}
Proceeding with the second term in \eqref{taylor}, we need a bound for the uniform error. It can be obtained in a similar (in fact easier) way as for the term $A_n$ in \eqref{eq:2}. Such a bound is already given in \cite{BibingerReiss} on page 10 for the estimators in \eqref{pilotsigmahat} with $J=1$, and readily extends to the case $J>1$. Since $\an\propto\hn^{\frac{-2\mathfrak{a}}{2\mathfrak{a}+1}}$ is the rate-optimal choice, we get with the upper bound from \cite{BibingerReiss} that
\begin{align*}
\maxi\frac{\partial{\psi(\eta,\sigma^{2}_{i\ah})}}{\partial{y}}\Big(\absZiiw-\sigma^{2}_{i\ah}\Big)&=\mathcal{O}_{\mathbb{P}}\left(\hn^{\frac{\mathfrak{a}}{2\mathfrak{a}+1}}\left(\log(\hn^{-1})\right)^{\frac{\mathfrak{a}}{2\mathfrak{a}+1}}\right)=\KLEINO_{\mathbb{P}}(1)\,.
\end{align*}
Proceeding with the term 
\begin{align*}
\maxi\sqrt{\an\log(n)}\left|\Zi-\Ziiw\right|
\end{align*}
we conclude similarly with the triangle inequality,
\begin{align*}
&\maxi\sqrt{\an\log(n)}\left|\Zi-\Ziiw\right|\\
&\leq \maxi\sqrt{\an\log(n)}\left|\Zi-\sigma^{2}_{i\ah}\right|+\maxi\sqrt{\an\log(n)}\left|\Ziiw-\sigma^{2}_{i\ah}\right|\,,
\end{align*}
the uniform bound applied to each summand and the regularity of $\left(\sigma^{2}_{t}\right)_{t\in\left[0,1\right]}$ under the null hypothesis \ref{hypo}. This implies
\begin{align*}
\maxi\sqrt{\an\log(n)}\left|\Zi-\Ziiw\right|=\mathcal{O}_{\mathbb{P}}(1)\,,
\end{align*}
such that the convergence in Proposition \ref{prop3n} follows. 
\qed\\
In order to conclude the convergence for the adaptive statistics in Theorem \ref{thm:1}, we have to show that replacing the oracle versions by the adaptive statistics does not affect the limit. It is sufficient to show the following for the \emph{fourth step} to complete the proof of the approximation steps mentioned in Section \ref{subsec:4.2}.
\begin{prop}\label{prop4n}
Given the assumptions of Theorem \ref{thm:1}, it holds under \ref{hypo} that
\begin{align*}
\sqrt{\an\log\left(h_n^{-1}\right)}\maxi\left|\overline{RV}^{ad}_{n,i}-\overline{RV}_{n,i}\right|\overset{\mathbb{P}}{\longrightarrow}0\,.
\end{align*}
\end{prop}
\noindent
\textit{Proof of Proposition \ref{prop4n}}.\\
As we have argued above, $\eta^{2}$ can be replaced by the $\sqrt{n}$-rate consistent estimator \eqref{etahat} without affecting the limit behaviour of the statistics. Therefore it is sufficient to consider the plug-in estimation of the spot volatility in the weights $(w_{ij\ell})$. First of all, taking into account that the asymptotic order of the weights \eqref{wOrder} do not depend on $i,\ell$, we may consider them as a function $w_{j}=w_{j}(\sigma^{2})$ of the spot volatility. Calculating the first derivative, $w'_{j}$ as pursued on page 40 in \cite{Altmeyer2015}, we get the upper bound
\begin{align}\label{Dw}
w'(\sigma^{2})=\mathcal{O}_{\mathbb{P}}(w_{j}(\sigma^{2})\log^{2}(n))\,.
\end{align}
In order to bound $\maxi|\Ri^{ad}-\Ri|$, take into account that $\sum_{j}^{}w_{j}(x)=1$ for every $x$. So, it is sufficient to consider the term
\begin{align*}
\maxi\frac{1}{\an}\sum_{j=1}^{\nh}(w_{j}(\sigma^{2}_{i\ah})-w_{j}(\hat{\sigma}^{2}_{i\ah}))\sum_{\ell=1}^{\an}(S^{2}_{ij\ell}(Y)-\mathbb{E}[S^{2}_{ij\ell}(Y)])\,.
\end{align*}
The only difference compared with \cite{BibingerWinkel2018} and \cite{Altmeyer2015}, is to replace the point-wise $L^{1}$ bound for $|\hat{\sigma}^{2}_{i\ah}-{\sigma}^{2}_{i\ah}|$ by the uniform bound from Proposition \ref{prop3n}, with that the bound
\begin{align*}
|w_{j}(\sigma^{2}_{i\ah})-w_{j}(\hat{\sigma}^2_{i\ah})|=\mathcal{O}_{\mathbb{P}}(w_{j}(\sigma^{2})\log^{2}(n)\hn^{\frac{\mathfrak{a}}{2\mathfrak{a}+1}}\log(n)^{\frac{\mathfrak{a}}{2\mathfrak{a}+1}}).
\end{align*}
follows, using the mean value theorem and \eqref{Dw}.
\qed\\
The key, proving the last conclusion is to apply strong invariance principles by \cite{KMT2}. First of all, we have to take into account, that the rescaling factors in $\widetilde{U}'/\sqrt{8\eta}$ provide only an asymptotically distribution-free limit. So it is more adequate for our purpose to rescale with the exact finite-sample standard deviation, that is
\[2\bigg(\sum_{m=1}^{\nh}\left(\sigma^{2}_{i\ah}+\frac{\eta^{2}}{n}\left[\varphi_{mk},\varphi_{mk}\right]_{n}\right)^{-2}\bigg)^{-1}\,.\]
Using a Taylor approximation and the convergence of the above variances to $8\sigma_{i\ah}^3\eta$, presented in Section 6.2.\ of \cite{Altmeyer2015}, it is clear that the approximation holds.\\
Let $I_{i,\nu}$ and $\widetilde{I}_{i,\nu}$ be the exact finite-sample variances and define $\mathbb{L}^{\left(n\right)}_{i,\nu}$ and $\widetilde{\mathbb{L}}^{\left(n\right)}_{i,\nu}$ given by
\begin{align*}
\mathbb{L}^{\left(n\right)}_{i,\nu}&=\frac{\sum_{j=1}^{\nh}w_{ij\nu}\left(\left(\sigma_{i\ah}S_{j\nu}\left(W\right)+S_{j\nu}\left(\varepsilon\right)\right)^{2}-\mu_{i,\nu,j}\right)}{\sqrt{I_{i,\nu}}}\,,\\
\widetilde{\mathbb{L}}^{\left(n\right)}_{i,\nu}&=\frac{\sum_{j=1}^{\nh}w_{ij\nu}\left(\left(\sigma_{(i-1)\ah}S_{j\nu}\left(W\right)+S_{j\nu}\left(\varepsilon\right)\right)^{2}-\mu_{i,\nu,j}\right)}{\sqrt{\widetilde{I}_{i,\nu}}}\,.
\end{align*}
The distributions of $\left(\mathbb{L}^{\left(n\right)}_{i,\nu}\right)_{\nu}$ and $\left(\widetilde{\mathbb{L}}^{\left(n\right)}_{i+1,\nu}\right)_{\nu}$ do not depend on the volatility. Therefore, and due to the independence of Brownian increments, the latter are two independent families. Furthermore, the independence of Brownian increments also yields that each family itself forms an independent family in $\nu$. Taking into account the remark in \cite{KMT2} below Theorem 4, we can proceed as follows. Since we want to ensure the existence of a properly approximating independent Gaussian family $\left(Z_{i}\right)_{i}$, according to Theorem 4 in \cite{KMT2}, we have to pick a function $H$ such that 
\begin{align}\label{eq:kmtbed1}
\frac{H(x)}{x^{3+\delta}} \text{ is increasing for some } \delta>0\,,
\end{align}
\begin{align}\label{eq:kmtbed2}
\frac{\log\left(H\left(\left|x\right|\right)\right)}{x} \text{ is decreasing and}
\end{align}
\begin{align}\label{eq:kmtbed3}
\int_{}^{}H\left(\left|x\right|\right)\,d\mathbb{P}_{{\mathbb{L}}^{\left(n\right)}_{i,\nu}}<\infty\,.
\end{align}
We pick a power function $H$ and set $H(x)=\left|x\right|^{p}$ with some $p\geq 4$ such that \eqref{eq:kmtbed1} and \eqref{eq:kmtbed2} are fulfilled. For the latter condition \eqref{eq:kmtbed3}, by Jensen's inequality and Rosenthal's inequality, we require at this point \eqref{momnoise} up to $m=8$.
In order to control the remainder term in the approximation, we take into account that
\begin{align*}
&\maxi\left|\sum_{\nu=1}^{\left(i+1\right)\an}\left(\mathbb{L}^{\left(n\right)}_{i,\nu}-Z_{\nu}\right)-\sum_{\nu=1}^{i\an}\left(\mathbb{L}^{\left(n\right)}_{i,\nu}-Z_{\nu}\right)\right|\leq 4\cdot\maxi\left|\sum_{\nu=1}^{\left(i+1\right)\an}\left(\mathbb{L}^{\left(n\right)}_{i,\nu}-Z_{\nu}\right)\right|.
\end{align*}
Furthermore, the triangle inequality and the Markov inequality yield
\begin{align*}
&\mathbb{P}\left[\maxi\left|\sum_{\nu=1}^{\left(i+1\right)\an}\left(\mathbb{L}^{\left(n\right)}_{i,\nu}-Z_{\nu}\right)\right|\geq x_{n}\right]\\
&\leq \mathbb{P}\left[\maxi\sum_{\nu=1}^{\left(i+1\right)\an}\left|\mathbb{L}^{\left(n\right)}_{i,\nu}-Z_{\nu}\right|\geq x_{n}\right]\leq \sum_{\nu=1}^{h^{-1}_{n}}\mathbb{P}\left[\left|\mathbb{L}^{\left(n\right)}_{i,\nu}-Z_{\nu}\right|\geq x_{n}\right]\\
&\leq \sum_{\nu=1}^{h^{-1}_{n}}x^{-p}_{n}\mathbb{E}\left[\left|\mathbb{L}^{\left(n\right)}_{i,\nu}-Z_{\nu}\right|^{p}\right]\,.
\end{align*}
Applying (1.6) in \cite{sakhanenko}, we get
\begin{align*}
\mathbb{P}\left[\maxi\left|\sum_{\nu=1}^{\left(i+1\right)\an}\left(\mathbb{L}^{\left(n\right)}_{i,\nu}-Z_{\nu}\right)\right|\geq x_{n}\right]\leq \sum_{\nu=1}^{h^{-1}_{n}}x^{-p}_{n}\mathbb{E}\left[\left|\mathbb{L}^{\left(n\right)}_{i,\nu}\right|^{p}\right]\leq C h^{-1}_{n} x^{-p}_{n},
\end{align*}
where $C>0$ is the positive constant given in (1.6) in \cite{sakhanenko}.
We set
\begin{align*}
x_{n}=\sqrt{\an}\left(\log\left(\hn^{-1}\right)\right)^{-1/2}.
\end{align*}
Since there are more bins than big blocks, the conditions of Theorem 4 in \cite{KMT2} are fulfilled. Furthermore, we can choose $p$ by \eqref{eq:an2} such that
\begin{align*}
\an^{-p/2}h^{-1}_{n}=\KLEINO\left(\left(\log\left(\hn^{-1}\right)\right)^{-p/2}\right)\,.
\end{align*}
So the remainder term fulfills
\begin{align*}
\maxi\left|\sum_{\nu=1}^{\left(i+1\right)\an}\left(\mathbb{L}^{\left(n\right)}_{i,\nu}-Z_{\nu}\right)-\sum_{\nu=1}^{i\an}\left(\mathbb{L}^{\left(n\right)}_{i,\nu}-Z_{\nu}\right)\right|= \KLEINO_{\mathbb{P}}\left(\sqrt{\an}\left(\log\left(\hn^{-1}\right)\right)^{-1/2}\right)\,.
\end{align*}
Let $\mathbb{B}$ be the Brownian Motion in the invariance principle. This implies, that the family $(Z_{i})_{i}$ defined as
\begin{align*}
Z_{i}:=\an^{-1/2}(\mathbb{B}((i+1)\an)-\mathbb{B}(i\an))
\end{align*}
are i.i.d.\ standard normal variables. We set
\begin{align*}
\eta_{i}:=\frac{1}{\sqrt{\an}}\left(\mathbb{B}\left(\left(i+1\right)\an\right)-\mathbb{B}\left(i\an\right)+\KLEINO_{\mathbb{P}}\left(\left(\log\left(n\right)\right)^{-1/2}\right)\right)\,.
\end{align*}
The scaling properties of Brownian motion and the upper bound given for the remainder term give the desired result using Lemma 1 in \cite{wuzhao2007} applied to $(\eta_{i})_i$.
\qed\\
\subsection{Proof of Corollary  \ref{cor:1}}
The proof of Corollary \ref{cor:1} works along the same lines as the one of Theorem \ref{thm:1}. More precisely,
\begin{enumerate}[(a)]
\item in a first step, we have to show that the overlapping versions $\overline{RV}^{ov}_{n,i}$ can be replaced by $\overline{Z}^{ov}_{n,i}$. In a second step, we
\item have to do a shift in the volatility and proceed 
\item by showing that the estimated asymptotic standard deviations can be replaced by their limits and that
\item the difference between oracle and adaptive versions is asymptotically negligible, where the final step is to
\item use a limit theorem for extreme value statistics similar to Lemma 1 in \cite{wuzhao2007}. The appropriate tool for the overlapping versions is given by Lemma 2 in \cite{wuzhao2007}, which can be directly applied choosing $H$ as the rectangular kernel. The latter works, since even if the big blocks may intersect, it is crucial that the bins remain to be disjoint.
\end{enumerate}
Starting with (a) we will argue that the estimates provided in the proof of Theorem \ref{thm:1} are sufficient to conclude the limit for the overlapping statistics. 
We have to show that
\begin{align*}
\mathop{\mathrm{max}}\limits_{i=\an,\ldots,h_n^{-1}-\an}\left|\left|\frac{\overline{RV}^{ov}_{n,i}-\overline{RV}^{ov}_{n,i+1}}{|\overline{RV}^{ov}_{n,i+1}|^{3/4}}\right|-\left|\frac{\overline{Z}^{ov}_{n,i}-\overline{Z}^{ov}_{n,i+1}}{|\overline{Z}^{ov}_{n,i+1}|^{3/4}}\right|\right|=\KLEINO_{\mathbb{P}}(\an^{-1/2}\log(n)^{-1/2})\,.
\end{align*}
The triangle inequality, the decomposition \eqref{eq:1} and a Taylor expansion yield that it is sufficient to prove
\begin{align*}
\mathop{\mathrm{max}}\limits_{i=\an,\ldots,h_n^{-1}-1}\sqrt{\an\log\left(n\right)}\left|\overline{RV}^{ov}_{n,i}-\overline{Z}^{ov}_{n,i}\right|\overset{\mathbb{P}}{\longrightarrow} 0\,.
\end{align*}
The key step proving this is to consider the term corresponding to $A_{n}$ in \eqref{eq:2}. It is basically sufficient to translate the terms $A^{1,1}_{n}$ and $A^{1,2}_{n}$ to the overlapping case. Starting with $A^{1,1}_{n}$ we have to take into account the fact, that in the overlapping case, the index set, $i\in\{\an,\ldots,h_n^{-1}-\an\}$ is a factor $\an$ times larger than the index set for the non-overlapping case. But, since we can adapt the exponent $r$ in the Markov inequality by \eqref{momnoise}, we get a similar upper bound for $A^{1,1}_{n}$. Considering the corresponding part to the term $A^{1,2}_{n}$ we proceed as follows using Assumption \ref{hypo} and \eqref{xibound}:
\begin{align*}
&\frac{\sqrt{\log\left(n\right)}}{\sqrt{\an}}\sum_{\ell=1}^{\an}\sum_{j=1}^{\nh}w_{ij\ell}\int_{\left(\ell +i-1\right)\hn}^{\left(\ell +i-1+\an\right)\hn}\left(\x\left(\tau\right)\right)^{2}\left|\sigma^{2}_{\tau}-\sigma^{2}_{i\hn}\right|\,d\tau\\
&\leq  L_{n}\frac{\sqrt{\log\left(n\right)}}{\sqrt{\an}\hn}\sum_{\ell=1}^{\an}\left(\ell\hn\right)^{\mathfrak{a}}\hn\leq  L_{n}\sqrt{\log\left(n\right)}\sqrt{\an}\left(\an h_{n}\right)^{\mathfrak{a}}\longrightarrow 0\text{, as }n\rightarrow\infty.
\end{align*}
Concerning (b) we have to show that
\begin{align*}
\mathop{\mathrm{max}}\limits_{i=\an,\ldots,h_n^{-1}-\an}\left|\left|\frac{\overline{Z}^{ov}_{n,i}-\overline{Z}^{ov}_{n,i+1}}{|\overline{Z}^{ov}_{n,i+1}|^{3/4}}\right|-\left|\frac{\overline{Z}^{ov}_{n,i}-\widetilde{Z}^{ov}_{n,i+1}}{|\widetilde{Z}^{ov}_{n,i+1}|^{3/4}}\right|\right|=\KLEINO_{\mathbb{P}}(\an^{-1/2}\log(n)^{-1/2})\,.
\end{align*}
Again, after a proper decomposition of the terms and a Taylor expansion, it is sufficient to show that
\begin{align*}
\mathop{\mathrm{max}}\limits_{i=\an,\ldots,h_n^{-1}-\an}|\overline{Z}^{ov}_{n,i}-\widetilde{Z}^{ov}_{n,i}|=\KLEINO_{\mathbb{P}}(\an^{-1/2}\log(n)^{-1/2})\,.
\end{align*}
The discussion of this term works very similar as in the non-overlapping case. Using \ref{hypo} and the central limit theorem, as presented above, we can conclude the desired asymptotic behaviour by adapting the exponent $r$ in the Markov inequality. 
The third and fourth steps (c) and (d) are analogues of Propositions \ref{prop3n} and \ref{prop4n}. Since the upper bound, which is presented in \cite{BibingerReiss}, is not affected for overlapping big blocks, we omit the details. Concerning (e) let us only mention, that an additional tool which is necessary, is Lévy's modulus of continuity theorem 
in order to control the discretization error. Then, the limit \eqref{eq:limitnov} in Corollary \ref{cor:1} is an immediate consequence of Lemma 2 in \cite{wuzhao2007}.
\subsection{Proof of Proposition \ref{prop:withjumps}}
We decompose the process $Y_{t}=C_t+J_{t}+\varepsilon_{t}$ with the continuous semimartingale part
\begin{align*}
C_{t}=X_{0}+\int_{0}^{t}a_{s}\,ds+\int_{0}^{t}\sigma_{s}\,dW_{s}
\end{align*}
and write $\overline{RV}^{ad}_{n,i}(C+\varepsilon)$ for the statistics \eqref{rvad} applied to observations of a process where the jump part $\left(J_{t}\right)_{t\in\left[0,1\right]}$ is eliminated. We begin with some preliminaries for the proof. Throughout this proof, $K$ is a generic constant that may change from line to line. For $N^n(v_n)$ a sequence of counting processes with $N_t^n(v_n)=\int_0^t\int_{\R}\mathbbm{1}_{\{\gamma(x)>v_n\}}\mu(ds,dx)$, with $\gamma(x)$ from Assumption \ref{assjumps}, we have by (13.1.14) from \cite{JP} that
\begin{align}\label{hpr}\P\Big(N^n_{h_n}(v_n)\ge l\Big)\le K h_n^l v_n^{-rl}\end{align}
with $r$ from \eqref{bg}. 
We may restrict to the more difficult result for $\overline{V}_n^{ov,\tau}$ with overlapping statistics. With an analogous decomposition as in \eqref{eq:1}, the proof reduces to 
\begin{align}\label{jumpsneg}\mathop{\mathrm{max}}\limits_{i=\an,\ldots,h_n^{-1}}\big|\overline{RV}_{n,i}^{tr}-\overline{RV}^{ad}_{n,i}(C+\varepsilon)\big|=\KLEINO_{\mathbb{P}}\left(\big(\log(n)\an\big)^{-1/2}\right)\,.\end{align}
We separate bins on that truncations occur from (most) other bins
\begin{align*}&\mathop{\mathrm{max}}\limits_{i=\an,\ldots,h_n^{-1}}\big|\overline{RV}_{n,i}^{tr}-\overline{RV}^{ad}_{n,i}(C+\varepsilon)\big|\\
&=\an^{-1}\mathop{\mathrm{max}}\limits_{i}\Big|\sum_{\ell=i-\an+1}^{i}\Big(\hat{\sigma}^{2,ad}_{\left(\ell-1\right)\hn}\mathbbm{1}_{\{|\hat{\sigma}^{2,ad}_{\left(\ell-1\right)\hn}|\le h_n^{\tau-1}\}}-\hat{\sigma}^{2,ad}_{\left(\ell-1\right)\hn}(C+\varepsilon)\Big)\Big|\\
&\le \an^{-1}\mathop{\mathrm{max}}\limits_{i}\Big|\sum_{\ell=i-\an+1}^{i}\mathbbm{1}_{\{|\hat{\sigma}^{2,ad}_{\left(\ell-1\right)\hn}|> h_n^{\tau-1}\}}\hat{\sigma}^{2,ad}_{\left(\ell-1\right)\hn}(C+\varepsilon)\Big|\\
& + \an^{-1}\mathop{\mathrm{max}}\limits_{i}\Big|\sum_{\ell=i-\an+1}^{i}\hspace*{-.1cm}\mathbbm{1}_{\{|\hat{\sigma}^{2,ad}_{\left(\ell-1\right)\hn}|\le h_n^{\tau-1}\}}\hspace*{-.1cm}\sum_{j=1}^{\nh}\hspace*{-.1cm}\hat w_{j\ell}\Big(S_{j\ell}^2(J)\hspace*{-.05cm}+\hspace*{-.05cm}2S_{j\ell}(J)S_{j\ell}(\varepsilon)\hspace*{-.05cm}+\hspace*{-.05cm}2S_{j\ell}(J)S_{j\ell}(C)\Big)\hspace*{-.05cm}\Big|,
\end{align*}
and consider the two terms separately. For the second maximum the term with $S_{j\ell}^2(J)$ is the most involved one and we prove that
\begin{align}\label{h1}\mathop{\mathrm{max}}\limits_{i}\Big|\sum_{\ell=i-\an+1}^{i}\mathbbm{1}_{\{|\hat{\sigma}^{2,ad}_{\left(\ell-1\right)\hn}|\le h_n^{\tau-1}\}}\sum_{j=1}^{\nh}\hat w_{j\ell}\,S_{j\ell}^2(J)\Big|=\KLEINO_{\mathbb{P}}\left(\sqrt{\frac{\an}{\log(n)}}\right)\,.\end{align}
With some $c,\tilde c\in(0,1)$ the relation 
\[\mathbbm{1}_{\{|\hat{\sigma}^{2,ad}_{\left(\ell-1\right)\hn}|\le h_n^{\tau-1}\}}\le \mathbbm{1}_{\{\sum_{j=1}^{\nh}\hat w_{j\ell}\,S_{j\ell}^2(J)\le c\, h_n^{\tau-1}\}}+\mathbbm{1}_{\{|\hat{\sigma}^{2,ad}_{\left(\ell-1\right)\hn}(C+\varepsilon)|>\tilde c\, h_n^{\tau-1}\}}\]
can be used to decompose the term in two addends. First, we prove that
\begin{align}\label{h1_1}\mathop{\mathrm{max}}\limits_{i}\Big|\sum_{\ell=i-\an+1}^{i}\mathbbm{1}_{\{\sum_{j=1}^{\nh}\hat w_{j\ell}\,S_{j\ell}^2(J)\le c\, h_n^{\tau-1}\}}\sum_{j=1}^{\nh}\hat w_{j\ell}\,S_{j\ell}^2(J)\Big|=\KLEINO_{\mathbb{P}}\left(\sqrt{\frac{\an}{\log(n)}}\right)\,.\end{align}
Using the elementary estimate $|\Phi_j(t)|\le \sqrt{2}h_n^{-1/2}$ and that $\sum_{j\ge 1}\hat w_{j\ell}=1$, we obtain the bound
\begin{align}\label{h1_11}\hspace*{-.1cm}\sum_{j=1}^{\nh}\hspace*{-.05cm}\hat w_{j\ell}\,S_{j\ell}^2(J)=\sum_{j=1}^{\nh}\hspace*{-.05cm}\hat w_{j\ell}\bigg(\sum_{i=1}^n\Delta_i^n J\Phi_{j\ell}\left(\frac{i}{n}\right)\bigg)^2\le 2h_n^{-1}\Big(\sum_{i=\lfloor (\ell-1)nh_n\rfloor }^{\lfloor \ell nh_n\rfloor}|\Delta_i^n J|\Big)^2.\end{align}
We deduce the upper bound
\begin{align*}&\mathop{\mathrm{max}}\limits_{i}\Big|\sum_{\ell=i-\an+1}^{i}\mathbbm{1}_{\{\sum_{j=1}^{\nh}\hat w_{j\ell}\,S_{j\ell}^2(J)\le c\, h_n^{\tau-1}\}}\sum_{j=1}^{\nh}\hat w_{j\ell}\,S_{j\ell}^2(J)\Big|\\
&\le  \mathop{\mathrm{max}}\limits_{i}\Big|\sum_{\ell=i-\an+1}^{i}\Big(\Big(\sum_{j=1}^{\nh}\hat w_{j\ell}\,S_{j\ell}^2(J)\Big)\wedge c\,h_n^{\tau-1}\Big)\Big|\\
&\le \mathop{\mathrm{max}}\limits_{i}\Big|\sum_{\ell=i-\an+1}^{i}\Big(\Big(2h_n^{-1}\Big(\sum_{i=\lfloor (\ell-1)nh_n\rfloor }^{\lfloor \ell nh_n\rfloor}|\Delta_i^n J|\Big)^2\Big)\wedge c\,h_n^{\tau-1}\Big)\Big|
\,.\end{align*}
We decompose this term as follows
\begin{align*}&\mathop{\mathrm{max}}\limits_{i}\Big|\sum_{\ell=i-\an+1}^{i}\Big(\Big(2h_n^{-1}\Big(\sum_{i=\lfloor (\ell-1)nh_n\rfloor }^{\lfloor \ell nh_n\rfloor}|\Delta_i^n J|\Big)^2\Big)\wedge c\,h_n^{\tau-1}\Big)\Big|\\
&\le \mathop{\mathrm{max}}\limits_{i}\Big|\sum_{\ell=i-\an+1}^{i}2h_n^{-1}\Big(\sum_{i=\lfloor (\ell-1)nh_n\rfloor }^{\lfloor \ell nh_n\rfloor}|\Delta_i^n J|\Big)^2\1_{\big\{\sum_{i=\lfloor (\ell-1)nh_n\rfloor }^{\lfloor \ell nh_n\rfloor}|\Delta_i^n J|\le \sqrt{c/2}h_n^{\tau/2}\big\}}\Big|\\
&+\mathop{\mathrm{max}}\limits_{i}\Big|\sum_{\ell=i-\an+1}^{i}ch_n^{\tau-1}\1_{\big\{\sum_{i=\lfloor (\ell-1)nh_n\rfloor }^{\lfloor \ell nh_n\rfloor}|\Delta_i^n J|> \sqrt{c/2}h_n^{\tau/2}\big\}}\Big|\\
&\le \mathop{\mathrm{max}}\limits_{i}\Big|\sum_{\ell=i-\an+1}^{i}2h_n^{-1}\big|J_{\ell h_n}-J_{(\ell-1)h_n}\big|^2\,\1_{\{|J_{\ell h_n}-J_{(\ell-1)h_n}|\le \sqrt{c/2}h_n^{\tau/2}\}}\Big|\\
& +\mathop{\mathrm{max}}\limits_{i}\Big|\sum_{\ell=i-\an+1}^{i}2h_n^{-1}\Big(\big|J_{\ell h_n}-J_{(\ell-1)h_n}\big|^2-\Big(\sum_{i=\lfloor (\ell-1)nh_n\rfloor }^{\lfloor \ell nh_n\rfloor}|\Delta_i^n J|\Big)^2\Big)\1_{\big\{\sum_{i=\lfloor (\ell-1)nh_n\rfloor }^{\lfloor \ell nh_n\rfloor}|\Delta_i^n J|\le \sqrt{c/2}h_n^{\tau/2}\big\}}\Big|\\
&+\mathop{\mathrm{max}}\limits_{i}\Big|\sum_{\ell=i-\an+1}^{i}ch_n^{\tau-1}\1_{\big\{\sum_{i=\lfloor (\ell-1)nh_n\rfloor }^{\lfloor \ell nh_n\rfloor}|\Delta_i^n J|> \sqrt{c/2}h_n^{\tau/2}\big\}}\Big|\\
&=\Gamma_1+\Gamma_2+\Gamma_3\,,
\end{align*}
where we use the triangle inequality, that $\1_{\{B\le C\}}\le \1_{\{A\le C\}}$ if $A\le B$, and elementary inequalities for the maximum. We begin with $\Gamma_2$. It holds for $\varpi>0$ arbitrarily small that
\begin{align*}\Gamma_2&\le 4\mathop{\mathrm{max}}\limits_{i}\Big|\sum_{\ell=i-\an+1}^{i}h_n^{-1}\Big(\sum_{i=\lfloor (\ell-1)nh_n\rfloor }^{\lfloor \ell nh_n\rfloor}|\Delta_i^n J|\Big)^2\1_{\big\{\sum_{i=\lfloor (\ell-1)nh_n\rfloor }^{\lfloor \ell nh_n\rfloor}|\Delta_i^n J|\le h_n^{2/3+\varpi}\big\}}\Big|\\
&+2\mathop{\mathrm{max}}\limits_{i}\Big|\sum_{\ell=i-\an+1}^{i}h_n^{-1}\Big(\big|J_{\ell h_n}-J_{(\ell-1)h_n}\big|^2-\Big(\sum_{i=\lfloor (\ell-1)nh_n\rfloor }^{\lfloor \ell nh_n\rfloor}|\Delta_i^n J|\Big)^2\Big)\\
&\quad\times \1_{\{N^n_{\ell h_n}(h_n^{2/3+\varpi})-N^n_{(\ell-1)h_n}(h_n^{2/3+\varpi})\ge 2\}}\,\1_{\big\{h_n^{2/3+\varpi}\le \sum_{i=\lfloor (\ell-1)nh_n\rfloor }^{\lfloor \ell nh_n\rfloor}|\Delta_i^n J|\le \sqrt{c/2}h_n^{\tau/2}\big\}}\Big|\,,\end{align*}
with $N^n_t(v_n)$ from \eqref{hpr}. The additional indicator function in the last addend may be added, since $|J_{\ell h_n}-J_{(\ell-1)h_n}\big|=\sum_{i=\lfloor (\ell-1)nh_n\rfloor }^{\lfloor \ell nh_n\rfloor}|\Delta_i^n J|$ when there is at most one jump on the bin. By \eqref{hpr} with $v_n=h_n^{2/3+\varpi}$ and $l=2$, we obtain for the Poisson process $N^n_t(h_n^{2/3+\varpi})$:
\[\P\Big(N^n_{\ell h_n}(h_n^{2/3+\varpi})-N^n_{(\ell-1)h_n}(h_n^{2/3+\varpi})\ge 2\Big)\le h_n^2h_n^{-2r(2/3+\varpi)}\,,\] 
for all $\ell$, and we infer that
\begin{align*}
\Gamma_2&=\mathcal{O}\big(\an h_n^{-1}h_n^{4/3+2\varpi}\big)+\mathcal{O}_{\P}\big(\an\log(\an)h_n^{1+\tau}h_n^{-2r(2/3+\varpi)}\big)\\
&=\KLEINO\big(\an^{1/2}(\log(n))^{-1/2}\big)+\KLEINO_{\P}\big(\an^{1/2}(\log(n))^{-1/2}\big)\,.
\end{align*}
We used that $\an\le h_n^{2/3}$ by \eqref{eq:an2}, since $\aalpha\le 1$, for the first term and that by Condition \eqref{resjumps}:
\begin{align}\label{gamma2}\an^{1/2}\log(\an)\sqrt{\log(n)}h_n^{1+\tau}h_n^{-2r(2/3+\varpi)}\to 0\,.\end{align}
Define the sequence of random variables
\begin{align*}\mathcal{Z}_{\ell}=\Big(\big(J_{\ell h_n}-J_{(\ell-1)h_n}\big)\1_{\{|J_{\ell h_n}-J_{(\ell-1)h_n}|\le \sqrt{c/2}h_n^{\tau/2}\}}\Big)^2, \quad \ell=1,\ldots,h_n^{-1}\,.\end{align*}
We have that
\[\Gamma_1=\max_{i=\an,\ldots,h_n^{-1}}\sum_{\ell=i-\an+1}^{i}2h_n^{-1}\mathcal{Z}_{\ell}\,.\]
From equation (54) of \cite{aitjac10}, we obtain the bounds 
\begin{subequations}
\begin{align}\label{evarbound}\E\Big[\Big(\big|J_{\ell h_n}-J_{(\ell-1)h_n}\big| \wedge\sqrt{c/2}\,h_n^{\tau/2}\Big)^2\Big]& \le K h_n h_n^{\tau(1-r/2)}\,,\\
\label{hvarbound}\var\Big(\Big(\big|J_{\ell h_n}-J_{(\ell-1)h_n}\big| \wedge\sqrt{c/2}\,h_n^{\tau/2}\Big)^2\Big)&\le  K h_n h_n^{2\tau-r\tau/2}\,.\end{align}
\end{subequations}
Observe that
\begin{align*}\Gamma_1= \max_{i=\an,\ldots,h_n^{-1}} 2h_n^{-1}\Big(\sum_{\ell=1}^i \mathcal{Z}_{\ell}-\sum_{\ell=1}^{i-\an}\mathcal{Z}_{\ell}\Big)\le 2h_n^{-1}\max_{i=1,\ldots,h_n^{-1}} \sum_{\ell=1}^i \mathcal{Z}_{\ell}\,.\end{align*}
Since $\int_0^t\int_{\R}\delta(s,x)\1_{\{\gamma(x)\le \sqrt{c/2}\,h_n^{\tau/2}\}}(\mu-\nu)(dx,dx)$ is a martingale, \(\big(\sum_{\ell=1}^{i}2h_n^{-1}\mathcal{Z}_{\ell}\big)_{1\le i\le h_n^{-1}}\) is a submartingale as the martingale increments are uncorrelated and a squared martingale is always a submartingale. We apply Doob's submartingale maximal inequality which yields
\begin{align}\lambda\,\P\Big(\max_{i=\an,\ldots,h_n^{-1}} \sum_{\ell=1}^i 2h_n^{-1}\mathcal{Z}_{\ell}\ge \lambda \Big)\le 2h_n^{-1}\E\Big[\sum_{\ell=1}^{h_n^{-1}}\mathcal{Z}_{\ell}\Big]\propto h_n^{\tau(1-r/2)}\,,
\end{align}
such that \(\P\Big(\max_{i=\an,\ldots,h_n^{-1}} \sum_{\ell=1}^i 2h_n^{-1}\mathcal{Z}_{\ell}\ge \lambda \Big)\to 0\) for $\lambda^{-1}=\KLEINO(h_n^{\tau(r/2-1)})$.
Thus, $\Gamma_1$ is negligible as long as 
\begin{align}\label{gamma1}\an^{1/2}\sqrt{\log(n)}h_n^{\tau(1-r/2)}\rightarrow 0\,.\end{align}
From Condition \eqref{resjumps} we have that $r<2-\tau^{-1}\beta$ and it follows that
$\beta<\tau(2-r)$, what ensures the above relation. Under this condition, $\Gamma_3$ becomes negligible as well, since with \eqref{hpr} for $l=1$ and $v_n=h_n^{\tau/2}$, we obtain that
\[\Gamma_3=\mathcal{O}_{\P}\Big(\an\log(\an)h_n^{\tau(1-r/2)}\Big)=\KLEINO_{\mathbb{P}}\left(\sqrt{\frac{\an}{\log(n)}}\right)\,.\]
We have proved \eqref{h1_1}. Finally, we show that
\begin{align}\label{h2}\mathop{\mathrm{max}}\limits_{i}\Big|\sum_{\ell=i-\an+1}^{i}\mathbbm{1}_{\{|\hat{\sigma}^{2,ad}_{\left(\ell-1\right)\hn}|> h_n^{\tau-1}\}}\hat{\sigma}^{2,ad}_{\left(\ell-1\right)\hn}(C+\varepsilon)\Big|=\KLEINO_{\mathbb{P}}\left(\sqrt{\frac{\an}{\log(n)}}\right)\,,\end{align}
and discuss the similar remaining second term for \eqref{h1}. With some $c,\tilde c\in(0,1)$, we use the relation 
\[\mathbbm{1}_{\{|\hat{\sigma}^{2,ad}_{\left(\ell-1\right)\hn}|> h_n^{\tau-1}\}}\le \mathbbm{1}_{\{\sum_{j=1}^{\nh}\hat w_{j\ell}\,S_{j\ell}^2(J)> c\, h_n^{\tau-1}\}}+\mathbbm{1}_{\{|\hat{\sigma}^{2,ad}_{\left(\ell-1\right)\hn}(C+\varepsilon)|>\tilde c\, h_n^{\tau-1}\}}\,.\]
With Markov's inequality, we obtain that
\begin{align}\label{h2_m}\P\Big(\mathop{\mathrm{max}}\limits_{k=1,\ldots,h_n^{-1}}\big|\hat{\sigma}^{2,ad}_{\left(k-1\right)\hn}(C+\varepsilon)\big|>\lambda v_n\Big)&\le h_n^{-1}\P\Big(\big|\hat{\sigma}^{2,ad}_{\hn}(C+\varepsilon)\big|>\lambda v_n\Big)\\
&\le \notag h_n^{-1}K\frac{\E\Big[\big|\hat{\sigma}^{2,ad}_{\hn}(C+\varepsilon)\big|^p\Big]}{(\lambda v_n)^p}=\mathcal{O}\big(h_n^{-1}\log(n)v_n^{-p}\big)\,,
\end{align}
using moment bounds from Lemma 2 of \cite{BibingerWinkel2018} under condition \eqref{momnoise}. We derive that
\begin{align}\mathop{\mathrm{max}}\limits_{k=1,\ldots,h_n^{-1}}\big|\hat{\sigma}^{2,ad}_{\left(k-1\right)\hn}(C+\varepsilon)\big|=\mathcal{O}_{\P}\big(h_n^{-\varpi}\big)\end{align}
for arbitrary $\varpi>0$. In particular, for $\varpi$ from \eqref{eq:an2}
\[\mathop{\mathrm{max}}\limits_{k=1,\ldots,h_n^{-1}}\big|\hat{\sigma}^{2,ad}_{\left(k-1\right)\hn}(C+\varepsilon)\big|=\mathcal{O}_{\P}\big(h_n^{-\varpi/2}\big)=\KLEINO_{\mathbb{P}}\left(\sqrt{\frac{\an}{\log(n)}}\right)\,.\]
With \eqref{h1_11}, we obtain the estimate 
\begin{align*}
&\mathop{\mathrm{max}}\limits_{i}\Big|\sum_{\ell=i-\an+1}^{i}\mathbbm{1}_{\{\sum_{j=1}^{\nh}\hat w_{j\ell}\,S_{j\ell}^2(J)> c\, h_n^{\tau-1}\}}\hat{\sigma}^{2,ad}_{\left(\ell-1\right)\hn}(C+\varepsilon)\Big|\\ &\le \mathop{\mathrm{max}}\limits_{i}\Big|\sum_{\ell=i-\an+1}^{i}\1_{\big\{\sum_{i=\lfloor (\ell-1)nh_n\rfloor }^{\lfloor \ell nh_n\rfloor}|\Delta_i^n J|> \sqrt{c/2}h_n^{\tau/2}\big\}}
\hat{\sigma}^{2,ad}_{\left(\ell-1\right)\hn}(C+\varepsilon)\Big|\,.
\end{align*}
Since for $r\tau<1$, we have by \eqref{hpr} that
\[\P\Big(\bigcup_{\ell=1}^{h_n^{-1}}\big\{\big(N^n_{\ell h_n}\big(\sqrt{c/2} h_n^{\tau/2}\big)-N^n_{(\ell-1)h_n}\big(\sqrt{c/2} h_n^{\tau/2}\big)\big)\ge 2\big\}\Big)\le K h_n^{-1}h_n^2h_n^{-\tau r}\rightarrow 0\,,\]
we may consider instead
\[ \mathop{\mathrm{max}}\limits_{i}\Big|\sum_{\ell=i-\an+1}^{i}\mathbbm{1}_{\{|J_{\ell h_n}-J_{(\ell-1)h_n}|> \sqrt{c/2}\, h_n^{\tau/2}\}}\hat{\sigma}^{2,ad}_{\left(\ell-1\right)\hn}(C+\varepsilon)\Big|\,.\]
It is thus sufficient to show that
\[\P\Big(\bigcup_{k=\an}^{h_n^{-1}}\big\{\big(N^n_{(k-1)h_n}\big(\sqrt{c/2} h_n^{\tau/2}\big)-N^n_{(k-\an)h_n}\big(\sqrt{c/2} h_n^{\tau/2}\big)\big)\ge l\big\}\Big)\le K h_n^{-1+l}\an^lh_n^{-\tau rl/2}\rightarrow 0\,,\]
for some $l<\infty$ where we have applied an inequality analogous to \eqref{hpr}. This holds true, since 
\[h_n^{-1+l}\an^lh_n^{-\tau rl/2}\le K\,h_n^{-1}\big(h_n^{1-\beta-r\tau/2}\big)^l\]
and $r\tau<2(1-\beta)$ by Condition \eqref{resjumps}. Using \eqref{h1_11}, \eqref{h2_m} and that the squared jumps are summable, we obtain that for $\tau<1-(3-2\varpi)/p$ with $p\in\N$
\begin{align}&\notag\mathop{\mathrm{max}}\limits_{i}\Big|\sum_{\ell=i-\an+1}^{i}\mathbbm{1}_{\{|\hat{\sigma}^{2,ad}_{\left(\ell-1\right)\hn}(C+\varepsilon)|>\tilde c\, h_n^{\tau-1}\}}\sum_{j=1}^{\nh}\hat w_{j\ell}\,S_{j\ell}^2(J)\Big|\le h_n^{-1}\mathbbm{1}_{\{\mathop{\mathrm{max}}\limits_{\ell} |\hat{\sigma}^{2,ad}_{\left(\ell-1\right)\hn}(C+\varepsilon)|>\tilde c\, h_n^{\tau-1}\}}\\
&\label{h1_2}=\mathcal{O}_{\P}\big(h_n^{-3/2+p/2(1-\tau)}\big)=\KLEINO_{\mathbb{P}}\left(\sqrt{\frac{\an}{\log(n)}}\right)\,.\end{align}
This is sufficient for
\begin{align}\label{h2_2}\mathop{\mathrm{max}}\limits_{i}\Big|\sum_{\ell=i-\an+1}^{i}\mathbbm{1}_{\{|\hat{\sigma}^{2,ad}_{\left(\ell-1\right)\hn}(C+\varepsilon)|>\tilde c\, h_n^{\tau-1}\}}\hat{\sigma}^{2,ad}_{\left(\ell-1\right)\hn}(C+\varepsilon)\Big|=\KLEINO_{\mathbb{P}}\left(\sqrt{\frac{\an}{\log(n)}}\right)\,.\end{align}
Equations \eqref{h1_1} and \eqref{h1_2} imply Equation \eqref{h1}. Maxima of the terms with cross terms $S_{j\ell}(\varepsilon)S_{j\ell}(J)$ and $S_{j\ell}(\varepsilon)S_{j\ell}(C)$ can be handled similarly (or with Cauchy-Schwarz) and are of smaller order. Equations \eqref{h1} and \eqref{h2} imply Equation \eqref{jumpsneg} what finishes the proof of Proposition \ref{prop:withjumps}. \qed\\
\subsection{Proof of Theorem \ref{thm:2}}
We have to show that \eqref{rejectrule}, \eqref{rejectrule_ov}, \eqref{rejectrule_jump} and \eqref{rejectrule_jump_ov} yield asymptotically tests with power 1. Concerning \eqref{rejectrule}, that is, under \ref{alt},
\begin{align}\label{testpower1}
\mathbb{P}\left[\overline{V}_{n}\geq  \an^{-1/2}\left(\left(\log(m_{n})\right)^{-1/2} c_{\alpha}+\gamma_{m_{n}}\right)\right]\longrightarrow 1\quad\text{as }n\rightarrow +\infty.
\end{align}
We set
\begin{align}\label{eq:Vnhat}
\hat{V}_{n,i}=\left|\frac{\zeta_{i}^{n}-\zeta_{i+1}^{n}}{\sqrt{8\hat{\eta}}\big|\overline{RV}^{ad}_{n,i+1}\big|^{3/4}}\right|,\quad i=0,\ldots,{\lfloor \left(\ah\right)^{-1}\rfloor}-2
\end{align}
with
\begin{align*}
\zeta_{i}^{n}:=\alpha_n^{-1}\sum_{\ell=1}^{\an}\left(\hat{\sigma}^{2,ad}_{\ahlla}-\mathbb{E}\left[\hat{\sigma}^{2,ad}_{\ahlla}\right]\right).
\end{align*}
For $\theta-\lfloor(\alpha_n h_n)^{-1}\theta \rfloor\alpha_nh_n>\alpha_nh_n/2$, set $i^*=\lfloor(\alpha_n h_n)^{-1}\theta \rfloor$. For $\theta-\lfloor(\alpha_n h_n)^{-1}\theta \rfloor\alpha_nh_n\le \alpha_nh_n/2$, set $i^*=\lfloor(\alpha_n h_n)^{-1}\theta \rfloor-1$. Since $\theta\in(0,1)$, $i^*\ge 0$ for $n$ sufficiently large.
By the reverse triangle inequality, we get:
\begin{align*}
\overline{V}_{n}\geq -\hat{V}_{n,i^*}+\left|\frac{\sum_{\ell=1}^{\an}\mathbb{E}\left[\hat{\sigma}^{2,ad}_{i^*\ah+\left(\ell-1\right)h_{n}}\right]-\sum_{\ell=1}^{\an}\mathbb{E}\left[\hat{\sigma}^{2,ad}_{\left(i^*+1\right)\ah+\left(\ell-1\right)h_{n}}\right]}{\an\sqrt{8\hat{\eta}}\big|\overline{RV}^{ad}_{n,i+1}\big|^{3/4}}\right|\,.
\end{align*}
First of all, we can conclude by Theorem \ref{thm:1} that for all $i$:
\begin{align*}
\hat{V}_{n,i}=\mathcal{O}_{\mathbb{P}}\left(\an^{-1/2}\right).
\end{align*}
Then we take into account that the sum over $j$ is convex and $\hat{\sigma}^{2,ad}_{\ahlla}$ is already bias corrected with respect to the noise part. Furthermore, bounding the volatility from below, using the It\^{o} isometry and
\begin{align*}
\sum_{\ell=1}^{\an}\sum_{j=1}^{nh_{n}-1}w_{ij\ell}\int_{\ahll}^{\ahl}\left(\x\left(s\right)\right)^{2}\sigma^{2}_{s}\,ds\propto\frac{1}{\hn}\int_{i\ah}^{\left(i+1\right)\ah}\sigma^{2}_{s}\,ds\,,
\end{align*}
we obtain that with a constant $c>0$:
\begin{align*}
\overline{V}_{n}\geq -\mathcal{O}_{\mathbb{P}}\left(\frac{1}{\sqrt{\an}}\right)+c\left|\varsigma\left(i^*,n\right)-\varsigma\left(i^*+1,n\right)\right|\left(1-\KLEINO_{\mathbb{P}}\left(1\right)\right),
\end{align*}
with
\begin{align*}
\varsigma\left(i,n\right):=\frac{1}{\ah}\int_{i\ah}^{\left(i+1\right)\ah}\sigma_{s}^{2}\,ds\,.
\end{align*}
Note that the denominator in \eqref{eq:Vnhat} can be `absorbed' by the constant $c$. We give a lower bound on $\left|\varsigma\left(i^*,n\right)-\varsigma\left(i^*+1,n\right)\right|$. 
Under the alternative hypothesis \ref{alt}, we have for the continuous volatility part that
\begin{align*}\left|\int_{i^* \ah}^{\left(i^* +1\right)\ah}\big(\sigma^{2,(c)}_{s}-\sigma^{2,(c)}_{s+\ah}\big)\,ds\right|\le \ah \KK_n\left(\ah\right)^{\mathfrak{a}}\,.\end{align*} 
The jump component of the volatility is most difficult to handle for $r=2$. If it satisfies \eqref{bg} with some $r\ge 1$, we derive for some constant $K_p$ dependent on $p$ the bound
\begin{align}\notag\forall s,t\ge 0:~\E\left[\big|\sigma^{2,(j)}_{t}-\sigma^{2,(j)}_s\big|^p\big|\mathcal{F}_s\right]&\le K_p\,\E\Big[\Big(\int_s^{t}\int_{\R}(\gamma^r(x)\wedge 1)\lambda(dx)ds\Big)^{p/r}\Big]\\ &  \le K_p |t-s|^{((p/r)\wedge 1)}\,.\end{align}
With $r=2$ and for $p=1$, we thus obtain for $i^*=\lfloor(\alpha_n h_n)^{-1}\theta \rfloor$ that
\begin{align*}\left|\int_{i^* \ah}^{\left(i^* +1\right)\ah}\big(\sigma^{2,(j)}_{s}-\sigma^{2,(j)}_{s+\ah}-\Delta\sigma_{\theta}^2\1_{[0,\theta)}(s)\big)ds\right|&=\mathcal{O}_{\P}\Big(\int_{i^* \ah}^{\left(i^* +1\right)\ah}|\ah|^{1/2}\,ds\Big) \\
&=\mathcal{O}_{\P}((\ah)^{3/2})\,,\end{align*}
and an analogous bound for $i^*=\lfloor(\alpha_n h_n)^{-1}\theta \rfloor-1$. Thus, we obtain that
\begin{align*}
&\left|\varsigma\left(i^*,n\right)-\varsigma\left(i^*+1,n\right)\right|=(\ah)^{-1}\left|\int_{i^* \ah}^{\left(i^* +1\right)\ah}\big(\sigma^2_{s}-\sigma^2_{s+\ah}\big)\,ds\right|\\
&\ge (\ah)^{-1}\left(\left|\int_{i^* \ah}^{\left(i^* +1\right)\ah}\big(\sigma^{2,(j)}_{s}-\sigma^{2,(j)}_{s+\ah}\big)\,ds\right|-\left|\int_{i^* \ah}^{\left(i^* +1\right)\ah}\big(\sigma^{2,(c)}_{s}-\sigma^{2,(c)}_{s+\ah}\big)\,ds\right|\right)\\
&\ge (\ah)^{-1}\min\bigg(\Big|\int_{i^* \ah}^{\theta}\Delta\sigma^{2}_{\theta}\,ds\Big|,\Big|\int_{\theta}^{(i^*+2) \ah}\Delta\sigma^{2}_{\theta}\,ds\Big|\bigg)-\mathcal{O}_{\P}((\ah)^{1/2})-L_{n}\left(\ah\right)^{\mathfrak{a}}\\
&\geq \frac{1}{2}\Delta\sigma^{2}_{\theta}-\mathcal{O}_{\P}((\ah)^{1/2})-L_{n}\left(\ah\right)^{\mathfrak{a}}\,,
\end{align*}
where we have applied the reverse triangle inequality. This implies \eqref{testpower1}. In the non-overlapping case, two neighboring differences $\left|\varsigma\left(i,n\right)-\varsigma\left(i+1,n\right)\right|$ incorporate the volatility jump. Our above definition of $i^*$ ensures that we consider the most affected one for the lower bound. A corresponding lower bound for $\overline{V}^{ov}_{n}$ in the overlapping case becomes simpler as we always include statistics over two neighboring blocks, such that $\theta$ is close to the end-point between the two blocks. Proving that
\begin{align*}
\mathbb{P}\left[\overline{V}_{n}\leq  {\an}^{-1/2}\left(\left(\log(m_{n})\right)^{-1/2} c_{\alpha}+\gamma_{m_{n}}\right)\right]\longrightarrow 1-\alpha\quad\text{as }n\rightarrow +\infty
\end{align*}
under \ref{hypo}, is an immediate consequence of Theorem \ref{thm:1}. This completes the proof for \eqref{rejectrule}. We omit further details concerning \eqref{rejectrule_ov}, \eqref{rejectrule_jump} and \eqref{rejectrule_jump_ov}, since the estimates we have presented above can be readily adapted. \hfill\qed

\subsection{Proof of Proposition \ref{propCP}}
We adopt the following elementary lemma, related to Lemma B.1 in \cite{aue} and Lemma D.1. in \cite{Bibinger2017}. 
\begin{lem}\label{lem_argmax_bound}
Let $f(t)$ and $g(t)$ be functions on $[0,\theta]$ such that $f(t)$ is non-negative and increasing. As long as $f(\theta) - f(\theta - \gamma) > 2\sup_{0 \leq t \leq \theta}|g(t)|$ for some $\gamma \in [0,\theta]$, we have that
\begin{align*}
\operatorname{argmax}_{0 \leq t \leq \theta}\bigl(f(t)+g(t)\bigr) \geq \theta- \gamma.
\end{align*}
An analogous result holds if $f(t)$ and $g(t)$ are functions on $[\theta,1]$ and $f(t)$ is decreasing.
\end{lem}
\begin{proof}
Since
\begin{align*}\sup_{0\le t<\theta-\gamma}|g(t)|-g(\theta)\le 2\sup_{0 \leq t \leq \theta}|g(t)|<f(\theta) - f(\theta - \gamma)\,,
\end{align*}
we derive that
\begin{align*}\max_{0 \leq t < \theta-\gamma}\big(f(t)+g(t)\big)&\le \sup_{0 \leq t < \theta-\gamma}\big(f(t)\big)+\sup_{0 \leq t <\theta-\gamma}\big|g(t)\big|\\
&\le f(\theta-\gamma)+\sup_{0 \leq t <\theta-\gamma}\big|g(t)\big|<f(\theta)+g(\theta)\,,
\end{align*}
such that \(\operatorname{argmax}_{0 \leq t \leq \theta}\bigl(f(t)+g(t)\bigr) \geq \theta - \gamma\).
\end{proof}
Let $\theta\in\left(0,1\right)$ be the change point, that is, the jump time of the volatility. Without loss of generality $\delta=\Delta\sigma^2_{\theta}>0$. Define $(i^{*}-1)=\lceil\theta\hn^{-1}\rceil$, the smallest integer such that $(i^{*}-1)\hn\geq\theta$ holds. We use the following decomposition of $\overline{V}^{\diamond}_{n,i}$ for $i=\an,\ldots,\hn^{-1}-\an$:
\begin{align*}
\overline{V}^{\diamond}_{n,i}=\an^{-1/2}\left|A_{n,i}+B_{n,i}+C_{n,i}+D_{n,i}\right|\,,
\end{align*}
where
\begin{align*}
A_{n,i}&=\sum_{\ell=i-\an+1}^{i}\big(\hat{\sigma}^{2,ad}_{(\ell-1)\hn}-\mathbb{E}[\hat{\sigma}^{2,ad}_{(\ell-1)\hn}]\big)-\sum_{\ell=i+1}^{i+\an}\big(\hat{\sigma}^{2,ad}_{(\ell-1)\hn}-\mathbb{E}[\hat{\sigma}^{2,ad}_{(\ell-1)\hn}]\big)\,,\\
B_{n,i}&=\sum_{\ell=i-\an+1}^{i}\big(\mathbb{E}[\hat{\sigma}^{2,ad}_{(\ell-1)\hn}]-\sigma^{2}_{(\ell-1)\hn}\big)-\sum_{\ell=i+1}^{i+\an}\big(\mathbb{E}[\hat{\sigma}^{2,ad}_{(\ell-1)\hn}]-\sigma^{2}_{(\ell-1)\hn}\big)\,,\\
C_{n,i}&=\sum_{\ell=i-\an+1}^{i}\widetilde{\sigma}^{2}_{(\ell-1)\hn}-\sum_{\ell=i+1}^{i+\an}\widetilde{\sigma}^{2}_{(\ell-1)\hn}\,,\\
D_{n,i}&=\sum_{\ell=i-\an+1}^{i}\delta\mathbbm{1}_{\left\{\ell\geq i^{*}\right\}}-\sum_{\ell=i+1}^{i+\an}\delta\mathbbm{1}_{\left\{\ell\geq i^{*}\right\}}\,,
\end{align*}
with $\left(\widetilde{\sigma}^{2}_{t}\right)_{t\in\left[0,1\right]}$ the path of the volatility from that the jump is eliminated:
\begin{align*}
\sigma^{2}_{(\ell-1)\hn}=\widetilde{\sigma}^2_{(\ell-1)\hn}+\delta\mathbbm{1}_{\left\{\ell\geq i^{*}\right\}}\,.
\end{align*}
By definition, $\left(\widetilde{\sigma}^{2}_{t}\right)_{t\in\left[0,1\right]}$ then fulfills the regularity properties on \ref{hypo}. This implies that 
\begin{align*}
\left|C_{n,i}\right|&=\left|\sum_{\ell=i-\an+1}^{i}\big(\widetilde{\sigma}^{2}_{(\ell-1)\hn}-\widetilde{\sigma}^{2}_{(i-1)\hn}\big)-\sum_{\ell=i+1}^{i+\an}\big(\widetilde{\sigma}^{2}_{(\ell-1)\hn}-\widetilde{\sigma}^{2}_{(i-1)\hn}\big)\right|\\
&\leq 2\max\left(\left|\sum_{\ell=i-\an+1}^{i}\big(\widetilde{\sigma}^{2}_{(\ell-1)\hn}-\widetilde{\sigma}^{2}_{(i-1)\hn}\big)\right|,\left|\sum_{\ell=i+1}^{i+\an}\big(\widetilde{\sigma}^{2}_{(\ell-1)\hn}-\widetilde{\sigma}^{2}_{(i-1)\hn}\big)\right|\right).
\end{align*}
Under \ref{hypo}, we obtain uniformly in $i$ that almost surely
\begin{align*}
&\left|\sum_{\ell=i-\an+1}^{i}\big(\widetilde{\sigma}^{2}_{(\ell-1)\hn}-\widetilde{\sigma}^{2}_{(i-1)\hn}\big)\right|\leq \sum_{\ell=i-\an+1}^{i}\left|(\ell-1)\hn-(i-1)\hn\right|^{\mathfrak{a}}\leq \an(\ah)^{\mathfrak{a}}\,.
\end{align*}
This is sufficient for
\begin{align*}
\maxi\left|C_{n,i}\right|=\mathcal{O}_{\mathbb{P}}\big(\sqrt{\an\log((\ah)^{-1})}\big)\,.
\end{align*}
From the proof of Theorem \ref{thm:1}, we can thus conclude the following bound:
\begin{align*}
\maxi\Big(\left|A_{n,i}\right|+\left|B_{n,i}\right| \Big)=\mathcal{O}_{\mathbb{P}}\big(\sqrt{\an\log((\ah)^{-1})}\big)\,.
\end{align*}
Next, we consider a step-wise defined function $\left(g(t)\right)_{t\in\left[0,1\right]}$ given by
\begin{align*}
g(i\hn)=\an^{-1/2}(A_{n,i}+B_{n,i}+C_{n,i})
\end{align*}
and $(f(t))_{t\in\left[0,1\right]}$ being step-wise defined via
\begin{equation*}
f(i\hn) = \begin{cases}
        \hphantom{-}0, & \text{for $i+\an <i^{*}$}\,,\\
        \delta\an^{-1/2}(i-i^{*}+\an+1), & \text{for $i^{*}-\an\leq i\leq i^{*}-1$}\,,\\
        \delta\an^{-1/2}(\an-i+i^{*}-1), & \text{for $i^{*}-1\leq i\leq i^{*}+\an-1$}\,,\\
        \hphantom{-}0, & \text{for $i>i^{*}+\an-1$}\,.
      \end{cases}
\end{equation*}
The function $f$ fulfills
\begin{itemize}
\item $f{\big|}_{\left[0,\theta\right]}$ is monotonically increasing and
\item $f{\big|}_{\left[\theta,1\right]}$ is monotonically decreasing.
\end{itemize}
We get the following representation of $\overline{V}^{\diamond}_{n,i}$:
\begin{align*}
\overline{V}^{\diamond}_{n,i}=\left|g(i\hn)-f(i\hn)\right|\,.
\end{align*}
The calculations above imply that 
\begin{align}\label{supgbound}
\sup_{t\in\left[0,\theta\right]}\left|g(t)\right|=\mathcal{O}_{\mathbb{P}}(\sqrt{\log((\ah)^{-1})})\,.
\end{align}
Furthermore, for $i^{*}-c\an\leq i\leq i^{*}+c\an$, with any $0<c<1$, it holds that
\begin{align*}
f(i\hn)>\left|g(i\hn)\right|>0\,, 
\end{align*}
with a probability tending to one as $n\rightarrow+\infty$. Therefore, 
\begin{align*}
\overline{V}_{n,i}^{\diamond}=f(i\hn)-g(i\hn)\,,
\end{align*}
for those $i$ with a probability tending to one as $n\rightarrow+\infty$. For a sequence $\gamma_{n}$, with $\gamma_{n}\in\left[0,\ah\right]$, it holds that
\begin{align*}
f((i^*-1)\hn-\gamma_{n})=\delta\an^{-1/2}(-\lfloor\gamma_{n}\hn^{-1}\rfloor+\an)
\end{align*}
and 
\begin{align*}
f((i^*-1)\hn)-f((i^*-1)\hn-\gamma_{n})=\lfloor\gamma_{n}\hn^{-1}\rfloor\delta\an^{-1/2}\,.
\end{align*}
When we set
\begin{align*}
\gamma_{n}=h_n\delta^{-1}\sqrt{\an\log(n)}\leq\ah\,,
\end{align*}
we derive with \eqref{supgbound} that almost surely for $n$ sufficiently large:
\begin{align*}
f((i^*-1)\hn)-f((i^*-1)\hn-\gamma_{n})\geq 2\,\sup_{t\in\left[0,\theta\right]}\left|g(t)\right|\,.
\end{align*}
Therefore, $f{\big|}_{\left[0,\theta\right]}$ satisfies the conditions of Lemma \ref{lem_argmax_bound}. This implies that
\begin{align*}
(i^{*}-1)\hn\geq \operatorname{argmax}_{i=\an,\ldots, \hn^{-1}-\an} \overline{V}_{n,i}^{\diamond}\hn\geq (i^{*}-1)\hn-\gamma_{n}\,.
\end{align*}
An analogous procedure applied to the function $f{\big|}_{\left[\theta,1\right]}$ yields that
\begin{align*}
(i^{*}-1)\hn\leq \operatorname{argmax}_{i=\an,\ldots, \hn^{-1}-\an} \overline{V}_{n,i}^{\diamond}\hn\leq (i^{*}-1)\hn+\gamma_{n}\,.
\end{align*}
Overall, this yields
\begin{align*}
\left|\hat{\theta}_{n}-(i^{*}-1)\hn\right|&=\mathcal{O}_{\mathbb{P}}(\gamma_{n})=\KLEINO_{\mathbb{P}}(1)\,,
\end{align*}
which completes the proof of Proposition \ref{propCP}.\hfill\qed
\addcontentsline{toc}{section}{References}
\bibliographystyle{chicago}
\bibliography{literatur}
\end{document}